\documentclass[
a4paper 
,11pt 
,leqno
,english 
]{scrartcl} 
\usepackage{my-preamble-article} 
\usepackage{my-macros} 
\usepackage{my-counters-related-preprints} 
\usepackage{my-specific-macros} 
\title{Global behaviour of solutions stable at infinity for gradient systems in higher space dimension: the no invasion case}
\usepackage{authblk} 
\author{Emmanuel \textsc{Risler}}
\begin{document}
\maketitle
\begin{abstract}
This paper is concerned with parabolic gradient systems of the form
\[
u_t = -\nabla V(u) + \Delta_x u
\,,
\] 
where the space variable $x$ and the state variable $u$ are multidimensional, and the potential $V$ is coercive at infinity. For such systems, the asymptotic behaviour of solutions \emph{stable at infinity}, that is approaching a stable homogeneous equilibrium as $\abs{x}$ goes to $+\infty$, is investigated. A partial description of the global asymptotic behaviour of such a solution is provided, depending on the mean speed of growth of the spatial domain where the solution is not close to this equilibrium, in relation with the \emph{asymptotic energy} of the solution. If this mean speed is zero, then the asymptotic energy is nonnegative, and the time derivative $u_t$ goes to $0$ uniformly in space. If conversely the mean speed is nonzero, then the asymptotic energy equals $-\infty$. This result is called upon in a companion paper where the global behaviour of radially symmetric solutions stable at infinity is described. The proof relies mainly on energy estimates in the laboratory frame and in frames travelling at a small nonzero velocity. 
\end{abstract}
\nnfootnote{%
}
\thispagestyle{empty} 
\pagestyle{empty}
\hypersetup{pageanchor=false} 
\newpage
\tableofcontents
\newpage
\hypersetup{pageanchor=true} 
\pagestyle{plain}
\setcounter{page}{1}
\newpage
\section{Introduction}
This paper deals with the global dynamics of nonlinear parabolic systems of the form
\begin{equation}
\label{syst_sf}
u_t = -\nabla V(u) + \Delta_x u
\end{equation}
where the time variable $t$ is real, the space variable $x$ lies in the spatial domain $\rr^{\dSpace}$ with $\dSpace$ an integer not smaller than $2$, the function $(x,t)\mapsto u(x,t)$ takes its values in the state domain $\rr^{\dState}$ with $\dState$ a positive integer, and the nonlinearity is the gradient of a scalar \emph{potential} function $V:\rr^{\dState}\to\rr$, which is assumed to be regular (of class $\ccc^2$) and coercive at infinity (see hypothesis \cref{hyp_coerc} in \vref{subsec:coerc_glob_exist}).

A fundamental feature of system~\cref{syst_sf} is that it can be recast, at least formally, as the gradient flow of an energy functional. If $(w,w')$ is a pair of vectors of $\rr^{\dSpace}$ or $\rr^{\dState}$ or $\rr^{\dSpace\times\dState}$, let $w\cdot w'$ and $\abs{w} =\sqrt{w\cdot w}$ denote the usual Euclidean scalar product and the usual Euclidean norm, respectively, and let us simply write $w^2$ for $\abs{w}^2$. For every function $v:x\mapsto v(x)$ defined on $\rr^{\dSpace}$ with values in $\rr^{\dState}$ its \emph{energy} (or \emph{Lagrangian} or \emph{action}) is defined, at least formally, as
\begin{equation}
\label{form_en}
\eee[v] = \int_{\rr^{\dSpace}} \biggl(\frac{1}{2}\abs{\nabla_x v(x)}^2+V\bigl(v(x)\bigr)\biggr)\, dx
\end{equation}
where
\begin{equation}
\label{square_gradient}
\abs{\nabla_x v(x)}^2 = \sum_{i=1}^{\dSpace} \sum_{j=1}^{\dState} \bigl(\partial_{x_i}v_j(x)\bigr)^2
\,.
\end{equation}
Formally, the differential of this functional reads (skipping border terms in the integration by parts):
\[
\begin{aligned}
d\eee[v]\cdot \delta v &= \int_\rr \bigl( \nabla_x v \cdot \nabla_x(\delta v) + \nabla V(v)\cdot \delta v \bigr) \, dx \\
&= \int_\rr \bigl( - \Delta_x v + \nabla V(v) \bigr) \cdot \delta v \, dx
\,.
\end{aligned}
\]
In other words, the (formal) gradient of this functional with respect to the $L^2(\rr^{\dSpace},\rr^{\dState})$-scalar product reads: 
\[
\nabla\eee[v] = \nabla V(v) - \Delta_x v
\,,
\]
and system \cref{syst_sf} can formally be rewritten under the form: 
\[
u_t = - \nabla \eee[u(\cdot,t)]
\,.
\]
Accordingly, if $(x,t)\mapsto u(x,t)$ is a solution of this system, then (formally):
\begin{equation}
\label{dt_form_en}
\frac{d}{d t}\eee[u(\cdot,t)] = - \int_{\rr^{\dSpace}} u_t(x,t)^2 \, dx \le 0
\,.
\end{equation}
Here
\[
u_t(x,t)^2 = \sum_{j=1}^{\dState} \bigl( \partial_t u_j(x,t)\bigr)^2
\,.
\]

An additional and related feature of system~\cref{syst_sf} is that a formal gradient structure exists not only in the laboratory frame, but also in every frame travelling at a constant velocity. For every $c$ in $\rr^{\dSpace}$, if two functions $(x,t)\mapsto u(x,t)$ and $(\xi,t)\mapsto v(\xi,t)$ are related by 
\[
u(x,t) = v(\xi,t) 
\quad\text{for}\quad
x = ct + \xi
\,,
\]
then $u$ is a solution of \cref{syst_sf} if and only if $v$ is a solution of
\begin{equation}
\label{syst_tf}
v_t-c \cdot \nabla_\xi v = -\nabla V(v)+\Delta_\xi v
\,.
\end{equation}
For every function $w:\xi\mapsto w(\xi)$ defined on $\rr^{\dSpace}$ with values in $\rr^{\dState}$ its (formal) energy with respect to system \cref{syst_tf} is defined as
\begin{equation}
\label{form_en_tf}
\eee_{c}[w]=\int_{\rr^{\dSpace}} e^{c\cdot \xi} \biggl( \frac{1}{2}\abs{\nabla_\xi w(\xi)}^2 + V\bigl(w(\xi)\bigr) \biggr)\, d\xi
\,.
\end{equation}
Formally, the differential of this functional reads (skipping border terms in the integration by parts):
\[
\begin{aligned}
d\eee_{c}[w] \cdot \delta w &= \int_{\rr^{\dSpace}} e^{c\cdot \xi} \bigl( \nabla_\xi w \cdot \nabla_\xi (\delta w) + \nabla V(w) \cdot \delta w \bigr) \, d\xi \\
&= \int_{\rr^{\dSpace}} e^{c\cdot \xi} \bigl( - \Delta_\xi w - c\cdot \nabla_\xi w + \nabla V(w) \bigr) \cdot \delta w \, d\xi 
\,.
\end{aligned}
\]
In other words, the (formal) gradient of this functional with respect to the $L^2$-scalar product with weight $e^{c\cdot \xi}$ on functions $\rr^{\dSpace}\to\rr^{\dState}$ reads:
\[
\nabla_{c} \eee_{c}[w] = - \Delta_\xi w - c\cdot \nabla_\xi w + \nabla V(w)
\,,
\]
system \cref{syst_tf} can formally be rewritten under the form:
\begin{equation}
\label{form_grad_tf}
v_t = - \nabla_{c} \eee_{c}[v(\cdot,t)]
\,,
\end{equation}
and if $(\xi,t)\mapsto v(\xi,t)$ is a solution of system \cref{syst_tf}, then (formally):
\begin{equation}
\label{dt_form_en_tf}
\frac{d}{d t}\eee_{c}[v(\cdot,t)] = -\int_{\rr^{\dSpace}} e^{c\cdot \xi} v_t(\xi,t)^2\, d\xi \le 0
\,.
\end{equation}
This gradient structure has been known for a long time \cite{FifeMcLeod_approachTravFront_1977}, but it is only more recently that it received a more detailed attention from several authors (among which S. Heinze, C. B. Muratov, Th. Gallay, R. Joly, and the author \cite{Heinze_variationalApproachTW_2001,Muratov_globVarStructPropagation_2004,GallayRisler_globStabBistableTW_2007,Risler_globCVTravFronts_2008,GallayJoly_globStabDampedWaveBistable_2009}), and that is was shown that this structure is sufficient (in itself, that is without the use of the maximum principle) to prove results of global convergence towards travelling fronts. These ideas have been applied since in different contexts, to prove either global convergence or just existence results, see for instance \cite{Chapuisat_existenceCurvedFront_2007,ChapuisatJoly_asymptProfilesTravFrontBiolEqu_2010,MuratovNovaga_frontPropIVariational_2008,MuratovNovaga_frontPropIISharpReaction_2008,MuratovNovaga_globExpConvTW_2012,AlikakosKatzourakis_heteroclinicTW_2011,AlikakosFusco_ellipticSystemsPhaseTransType_2018,Luo_globStabDampedWaveEqu_2013,BouhoursNadin_variationalApproachRDForcedSpeedDim1_2015,BouhoursGiletti_extinctSpreadClimateAllee_2016,BouhoursGiletti_spreadVanishMonStabRDEqu_2018,OliverBonafoux_heteroclinicTW1dParabSystDegenerate_2021,OliverBonafoux_TWparabAllenCahn_2021,ChenChienHuang_varApproach3PhaseTWgradSyst_2021,ChenCotiZelati_TWSolAllenCahnEqu_2022,OliverBonafouxRisler_globCVPushedTravFronts_2023}. More recently, this gradient structure enabled the author to push further the program initiated by P. C. Fife and J. McLeod in the late seventies with the aim of describing the global asymptotic behaviour (when space is one-dimensional) of every \emph{bistable} solution, that is every solution close to stable homogeneous equilibria at both ends of space (\cite{FifeMcLeod_approachTravFront_1977,FifeMcLeod_phasePlaneDisc_1981,Fife_longTimeBistable_1979}): those results were extended to parabolic gradient systems \cite{Risler_globalRelaxation_2016,Risler_globalBehaviour_2016}, hyperbolic gradient systems \cite{Risler_globalBehaviourHyperbolicGradient_2017}, and radially symmetric solutions of parabolic gradient systems in higher space dimension \cite{Risler_globalBehaviourRadiallySymmetric_2017}. 

For general solutions (without the radial symmetry hypothesis) stable at infinity of system \cref{syst_sf}, little seems to be known. This is in strike contrast with the scalar case $\dState$ equals $1$, which is the subject of a large amount of literature: for extinction/invasion (threshold) results in relation with the initial condition and the reaction term see for instance \cite{AronsonWeinberger_multidimNonlinDiffPopGen_1978,DuMatano_cvSharpThresholdNonlinDiff_2010,MuratovZhong_thresholdSymSol_2013,MuratovZhong_thresholdPhenSymDecrRadialSolRDEquations_2017,Zlatos_sharpTransitionExtinctPropag_2005}, for local convergence and quasi-convergence results see for instance \cite{DuMatano_cvSharpThresholdNonlinDiff_2010,DuPolacik_locUnifCVNLparabEquRN_2015,MatanoPolacik_dynNonnegSolOneDimRDI_2016,MatanoPolacik_dynNonnegSolOneDimRDII_2020,Polacik_cvQuasiCvSolParabEquRoverview_2017,HamelRossi_spreadSpeedsOneDimSymRDEqu_2021,Polacik_bdedRadialSolParabEquQuasiconvStableInfinity_2023}, and for further estimates on the location and shape at large positive times of the level sets see for instance \cite{Jones_sphericallySymmetricSolutionsRDEquation_1983,Jones_asymptBehavRDEquHigherSpaceDim_1983,Uchiyama_asymptBehavSolRDEquaVaryingDrift_1985,RoussierMichon_stabilityRadiallySymmTWRDEqu_2004,DuMatano_radialTerraceSolutionsPropProfileRN_2020, RoquejoffreRoussier_sharpLargeTimeBehavNdimBistable_2021,HamelNinomiya_locExpEntiireSolRDequ_2020,HamelNinomiya_locExpEntiireSolRDequ_2020,HamelRossi_spreadSpeedsOneDimSymRDEqu_2021}.

The present paper can be viewed as a first attempt to pursue the ``variational'' strategy for general (not necessarily radially symmetric) solutions stable at infinity of parabolic gradient systems in higher space dimension. The achievement resulting from this attempt is rather modest at this stage since it reduces to providing the ``relaxation'' conclusion in the ``no invasion'' case. However, the description of the global behaviour of such solutions in the radially symmetric case \cite{Risler_globalBehaviourRadiallySymmetric_2017} actually relies on this achievement: in the proof the main result of \cite{Risler_globalBehaviourRadiallySymmetric_2017}, a variant (\vref{thm:no_inv_implies_relaxation}) of the main result of this paper (\vref{thm:main}) is indeed an essential step.
\section{Assumptions, notation, and statement of the results}
Throughout all the paper it is be assumed that the integer $\dSpace$ is not smaller than $2$; the case $\dSpace$ equals $1$ is treated in \cite{Risler_globalRelaxation_2016,Risler_globalBehaviour_2016} and presents some peculiarities (for instance the domain ``far away in space'' is not connected if $\dSpace=1$ while it is connected if $\dSpace$ is not smaller than $2$). 
\subsection{Semi-flow and coercivity hypothesis}
\label{subsec:semi_flow_coerc_hyp}
Let us consider the Banach space of continuous and uniformly bounded functions equipped with the uniform norm: 
\[
X = \bigl(\cccb{0},\norm{\dots}_{L^\infty(\rr^{\dSpace},\rr^{\dState})}\bigr)
\,.
\]
System \cref{syst_sf} defines a local semi-flow in $X$ (see for instance D. B. Henry's book \cite{Henry_geomSemilinParab_1981}). 

As in \cite{Risler_globalRelaxation_2016,Risler_globalBehaviour_2016,Risler_globalBehaviourRadiallySymmetric_2017}, let us assume that the potential function $V:\rr^{\dState}\to\rr$ is of class $\ccc^2$ and that this potential function is strictly coercive at infinity in the following sense: 
\begin{gather}
\tag{$\text{H}_\text{coerc}$}
\lim_{R\to+\infty}\quad  \inf_{\abs{u}\ge R}\ \frac{u\cdot \nabla V(u)}{\abs{u}^2} >0
\label{hyp_coerc}
\end{gather}
(in other words there exists a positive quantity $\varepsilon$ such that the quantity $u\cdot \nabla V(u)$ is greater than or equal to $\varepsilon\abs{u}^2$ as soon as $\abs{u}$ is large enough). 

According to this hypothesis \cref{hyp_coerc}, the semi-flow of system \cref{syst_sf} on $X$ is actually global (see \vref{prop:attr_ball}). Let us denote by $(S_t)_{t\ge0}$ this semi-flow. 

In the following, a \emph{solution of system \cref{syst_sf}} will refer to a function 
\[
\rr\times\rr^{\dSpace}\to\rr^{\dState}\,, \quad (x,t)\mapsto u(x,t)
\,,
\]
such that the function $u_0:x\mapsto u(x,t=0)$ (initial condition) is in $X$ and $u(\cdot,t)$ equals $(S_t u_0)(\cdot)$ for every nonnegative time $t$. 
\subsection{Minimum points and solutions stable at infinity}
%
\subsubsection{Minimum points}
Everywhere in this paper, the term ``minimum point'' denotes a point where a function --- namely the potential $V$ --- reaches a local \emph{or} global minimum. Let $\mmm$ denote the set of \emph{nondegenerate} minimum points:
\[
\mmm=\{u\in\rr^{\dState}: \nabla V(u)=0 
\quad\text{and}\quad 
D^2V(u)\text{ is positive definite}\}
\,.
\]
\subsubsection{Solutions stable at infinity}
\label{subsubsec:solutions_stable_at_infinity}
\begin{definition}[solution stable at infinity]
\label{def:solution_stable_at_infinity}
A solution $(x,t)\mapsto u(r,t)$ of system \cref{syst_sf} is said to be \emph{stable at infinity} if there exists a point $m$ of $\mmm$ such that
\[
\limsup_{r\to+\infty} \abs{(u(r,t)-m}\to 0
\quad\text{as}\quad
t\to+\infty
\,.
\]
More precisely, such a solution is said to be \emph{stable close to $m$ at infinity}. A function (initial condition) $u_0$ in $X$ is said to be \emph{stable (close to $m$) at infinity} if the solution $u(\cdot,t) = (S_t u_0)(\cdot)$ corresponding to this initial condition is stable (close to $m$) at infinity. 
\end{definition}
\begin{notation}
For every point $m$ of $\mmm$, let 
\[
\XstabInfty(m)
\]
denote the subset of $X$ made of initial conditions that are stable close to $m$ at infinity. 
\end{notation}
By definition, this set is positively invariant under the semi-flow of system~\cref{syst_sf}. 
\subsubsection{Invasion speed of a solution stable at infinity}
\begin{definition}[invasion speed of a solution stable at infinity]
\label{def:invasion_speed}
For every point $m$ of $\mmm$ and every solution $u:(x,t)\mapsto u(x,t)$ of system~\cref{syst_sf} which is in $\XstabInfty(m)$ (that is, stable close to $m$ at infinity), let us call \emph{set of no invasion speeds} of the solution, and let us denote by $\Snoinv[u]$, the set
\[
\Snoinv[u] = \bigl\{c>0:\sup_{x\in\rr^{\dSpace}, \ \abs{x}\ge ct}\abs{u(x,t)-m}\to 0 \text{ when } t\to+\infty\bigr\}
\,;
\]
and let us call \emph{invasion speed} of the solution, and let us denote by $\cInv[u]$ the (nonnegative) infimum of this set:
\[
\cInv[u] = \inf(\Snoinv[u])
\,.
\]
\end{definition}
According to \cref{lem:upper_bound_on_invasion_speed} below, the set $\Snoinv[u]$ is nonempty, so that this invasion speed $\cInv[u]$ is finite.
\subsection{Preliminary results}
\label{subsec:prel_res}
As everywhere in the paper, let us assume that $V$ is of class $\ccc^2$ and satisfies the coercivity hypothesis \cref{hyp_coerc}. Let $m$ denote a point of $\mmm$. 
\subsubsection{Sufficient condition for stability at infinity, bound on invasion speed, and exponential decrease beyond invasion}
\begin{lemma}[sufficient condition for stability at infinity]
\label{lem:sufficient_condition_stability_at_infinity}
Let $\delta$ denote a positive quantity, small enough so that, for every $v$ in $\rr^{\dState}$ such that $\abs{v-m}$ is not larger than $\delta$, the Hessian matrix $D^2V(v)$ is positive definite. Then, every solution $(x,t)\mapsto u(x,t)$ of system \cref{syst_sf} satisfying
\begin{equation}
\label{hyp_lem_sufficient_condition_stability_at_infinity}
\lim_{r\to+\infty}\sup_{x\in\rr^{\dSpace},\, \abs{x}\ge r}\abs{u(x,0)-m} \le \delta
\end{equation}
is stable close to $m$ at infinity. 
\end{lemma}
\begin{corollary}[to be stable at infinity is an open condition]
\label{cor:bist}
The set $\XstabInfty(m)$ is open in $X$. 
\end{corollary}
\begin{lemma}[upper bound on the invasion speed of a solution stable at infinity]
\label{lem:upper_bound_on_invasion_speed}
For every solution $u:(x,t)\mapsto u(x,t)$ of system \cref{syst_sf} which is stable close to $m$ at infinity, the quantity $\cInv[u]$ is bounded from above by a quantity depending on $V$ and $m$, but not on the particular solution $u$. 
\end{lemma}
\begin{lemma}[exponential decrease beyond invasion speed]
\label{lem:exponential_decrease_beyond_invasion_speed}
For every solution $u$ $(x,t)\mapsto u(x,t)$ of system \cref{syst_sf} which is stable close to $m$ at infinity, and for every positive quantity $c$ larger than $\cInv[u]$, there exist positive quantities $\nu$ and $K[u]$ such that, for every nonnegative time $t$, 
\begin{equation}
\label{exponential_decrease_beyond_invasion_speed}
\sup_{x\in\rr^{\dSpace},\, \abs{x}\ge ct}\abs{(u(x,t)-m} \le K[u] \exp(-\nu t)
\,.
\end{equation}
The quantity $\nu$ depends on $V$ and $m$ and the difference $c-\cInv[u]$ (only), whereas $K[u]$ depends additionally on $u$. 
\end{lemma}
\subsubsection{Asymptotic energy of a solution stable at infinity: definition and upper semi-continuity}
\label{subsubsec:asympt_en}
\begin{notation}
For every nonnegative quantity $r$, let $B(r)$ denote the open ball of centre $0_{\rr^{\dSpace}}$ and radius $r$ in $\rr^{\dSpace}$:
\[
B(r) = \{v\in\rr^{\dSpace}:\abs{v}<r\}
\,.
\]
\end{notation}
\begin{proposition}[asymptotic energy of a solution stable at infinity]
\label{prop:asympt_en}
For every solution $(x,t)\mapsto u(x,t)$ of system \cref{syst_sf} which is stable close to a point $m$ of $\mmm$ at infinity, there exists a quantity $\eeeAsympt[u]$ in $\{-\infty\}\cup\rr$ such that, for every positive quantity $c$ larger than $\cInv[u]$, the following limit holds:
\[
\int_{B(ct)}\biggl(\frac{1}{2}\abs{\nabla_x u(x,t)}^2+V\bigl(u(x,t)\bigr)-V(m)\biggr)\, dx \to \eeeAsympt[u]
\quad\text{as}\quad
t\to +\infty
\,.
\]
\end{proposition}
\begin{definition}[asymptotic energy of a solution stable at infinity]
If $u:(x,t)\mapsto u(x,t)$ is a solution stable at infinity of system \cref{syst_sf}, let us call \emph{asymptotic energy of $u$} the quantity $\eeeAsympt[u]$ provided by \cref{prop:asympt_en}.  

Similarly, if a function $u_0$ in $X$ is an initial condition stable at infinity, let us call \emph{asymptotic energy of $u_0$} the asymptotic energy of the solution of \cref{syst_sf} corresponding to this initial condition, and let us denote by
\[
\eeeAsympt[u_0]
\]
this asymptotic energy. 
\end{definition}
Thus the \emph{asymptotic energy functional} may be defined, for every point $m$ in $\mmm$, as
\begin{equation}
\label{E_infty}
\eeeAsymptm : \XstabInfty(m)  \to\{-\infty\}\sqcup\rr 
\,, \quad
u_0  \mapsto \eeeAsympt[u_0]
\,.
\end{equation}
As for the (descendent) gradient flow of a regular function on a finite-dimensional manifold, this asymptotic energy is upper semi-continuous with respect to initial condition, as stated by the following proposition. Its statements hold with respect to the topology induced on $\XstabInfty(m)$ by the $X$-norm and the usual topology on $\{-\infty\}\sqcup\rr$. 
\begin{proposition}[upper semi-continuity of the asymptotic energy]
\label{prop:scs_asympt_en}
For every $m$ in $\mmm$, the asymptotic energy functional $\eeeAsymptm$ is upper semi-continuous; equivalently, for every real quantity $E$, the set 
\[
\eeeAsymptm^{-1}\bigl([E,+\infty)\bigr)
=
\bigl\{u_0\in \XstabInfty(m): \eeeAsympt[u_0]\ge E\bigr\}
\] 
is closed. 
\end{proposition}
According to \cref{thm:main} below, this asymptotic energy is actually either nonnegative or equal to $-\infty$. 
\subsection{Main result}
%
\begin{theorem}[invasion/no invasion dichotomy]
\label{thm:main}
Let $V$ denote a function in $\ccc^2(\rr^{\dState},\rr)$ satisfying the coercivity hypothesis \cref{hyp_coerc}. Then, for every $m$ in $\mmm$ and for every solution $(x,t)\mapsto u(x,t)$ of system \cref{syst_sf} which is stable close to $m$ at infinity, one of the following two conclusions holds:
\begin{enumerate}
\item either $\cInv[u]=0$ (``no invasion''); in this case the asymptotic energy of the solution is nonnegative and the time derivative of the solution goes to $0$ as time goes to $+\infty$, uniformly with respect to $x$ in $\rr^{\dSpace}$:
\begin{equation}
\label{time_derivative_goes_uniformly_to_zero}
\sup_{x\in\rr^{\dSpace}}\abs{u_t(x,t)}\to0 
\quad\text{as}\quad
t\to+\infty
\,;
\end{equation}
\label{item:thm_main_no_invasion}
\item or $\cInv[u]>0$ (``invasion''); in this case the asymptotic energy of the solution equals $-\infty$.
\label{item:thm_main_invasion}
\end{enumerate}
\end{theorem}
\section{Preliminaries}
As everywhere else, let us consider a function $V$ in $\ccc^2(\rr^{\dState},\rr)$ satisfying the coercivity hypothesis \cref{hyp_coerc}. 
\subsection{Global existence of solutions and attracting ball for the semi-flow}
\label{subsec:coerc_glob_exist}
\begin{proposition}[global existence of solutions and attracting ball]
\label{prop:attr_ball}
For every function $u_0$ in $X$, system \cref{syst_sf} has a unique globally defined solution $t\mapsto S_t u_0$ in $\ccc^0([0,+\infty),X)$ with initial condition $u_0$. In addition, there exist a positive quantity $\RattInfty$ (radius of attracting ball for the $L^\infty$-norm), depending only on $V$, such that, for every large enough positive time $t$, 
\[
\norm{x\mapsto (S_t u_0)(x)}_{\Linfty} \le \RattInfty 
\,.
\]
\end{proposition}
The following proof was explained to me by Thierry Gallay. 
\begin{proof}
As mentioned in \cref{subsec:semi_flow_coerc_hyp} system \cref{syst_sf} is locally well-posed in $X$. As a consequence and due to the smoothing properties of the semi-flow it is sufficient to prove that solutions remain bounded in $L^\infty(\rr^{\dSpace},\rr^{\dState})$ on their maximal interval of existence to ensure that they are globally defined. 

Let $u_0$ denote a function in $X$, and let 
\[
u:\rr^{\dSpace}\times[0,\tMax), \quad (x,t)\mapsto u(x,t)
\]
denote the (maximal) solution of system \cref{syst_sf} with initial condition $u_0$, where $\tMax$ in $(0,+\infty]$ denotes the upper bound of the maximal time interval where this solution is defined. For all $(x,t)$ in $\rr^{\dSpace}\times[0,\tMax)$, let
\[
q(x,t) = \frac{1}{2} \abs{u(x,t)}^2
\,.
\]
It follows from this definition that, for $(x,t)$ in $\rr^{\dSpace}\times[0,\tMax)$,
\[
q_t= - u\cdot \nabla V(u) + u\cdot \Delta_x u
\quad\text{and}\quad
\Delta_x q = \sum_{i=1}^{\dSpace}\partial_{x_i}(u\cdot \partial_{x_i} u)
= \abs{\nabla_x u}^2 + u \cdot \Delta_x u 
\]
(the quantity $\abs{\nabla_x u}^2$ involves a sum over both space and field dimensions, see \vref{square_gradient}). 
Thus the function $q(\cdot,\cdot)$ is a solution of the system:
\begin{equation}
\label{syst_u_square}
q_t = - u\cdot\nabla V(u) + \Delta_x q - \abs{\nabla_x u}^2
\,.
\end{equation}
Besides, according to the coercivity hypothesis \cref{hyp_coerc}, there exist positive quantities $\epsCoerc$ and $\Kcoerc$ such that, for all $w$ in $\rr^{\dState}$, 
\begin{equation}
\label{explicit_coercivity}
w\cdot \nabla V(w) \ge \epsCoerc w^2 - \Kcoerc
\,.
\end{equation}
It follows from \cref{syst_u_square,explicit_coercivity} that
\[
\begin{aligned}
q_t &\le -\epsCoerc u^2 + \Kcoerc + \Delta_x q - \abs{\nabla_x u}^2 \\
&\le -2 \epsCoerc \, q + \Kcoerc + \Delta_x q
\,, 
\end{aligned}
\]
and as a consequence, introducing the solution $t\mapsto \widebar{q}(t)$ of the differential equation (and initial condition):
\begin{equation}
\label{ode_upper_bound}
\widebar{q}'(t) =  -2 \epsCoerc \, \widebar{q} + \Kcoerc
\,, \quad 
\widebar{q}(0) = \sup_{x\in\rr^{\dSpace}} q(x,0)
\,,
\end{equation}
it follows from the maximum principle that, for all $(x,t)$ in $\rr^{\dSpace}\times[0,\tMax)$,
\begin{equation}
\label{max_principle}
q(x,t) \le \widebar{q}(t)
\,.
\end{equation}
As a consequence blow-up cannot occur, the solution $u(x,t)$ must be defined up to $+\infty$ in time, and introducing the quantity
\[
\RattInfty = \sqrt{\frac{\Kcoerc}{\epsCoerc}+1}
\,,
\]
it follows from \cref{ode_upper_bound,max_principle} that, for every large enough positive time $t$, 
\[
\sup_{x\in\rr^{\dSpace}} \abs{u(x,t)} \le \RattInfty
\,.
\]
\Cref{prop:attr_ball} is proved.
\end{proof}
In addition, system~\cref{syst_sf} has smoothing properties (Henry \cite{Henry_geomSemilinParab_1981}). Due to these properties, since $V$ is of class $\ccc^2$, for every quantity $\alpha$ in $(0,1)$ every solution $t\mapsto S_t u_0$ in $\ccc^0([0,+\infty),X)$ actually belongs to
\[
\ccc^0\left((0,+\infty),\cccb{2,\alpha}\right)\cap \ccc^1\left((0,+\infty),\cccb{0,\alpha}\right),
\]
and, for every positive quantity $\varepsilon$, the quantities
\begin{equation}
\label{bound_u_ut_ck}
\sup_{t\ge\varepsilon}\norm{S_t u_0}_{\cccb{2,\alpha}}
\quad\text{and}\quad
\sup_{t\ge\varepsilon}\norm{\frac{d(S_t u_0)}{dt}(t)}_{\cccb{0,\alpha}}
\end{equation}
are finite. Among these estimates, the one provided by the following corollary will be especially relevant. 
\begin{corollary}[upper bound on the $\HoneulAlone$-norm]
\label{cor:att_ball_Honeul}
There exists a positive quantity $\RattHoneul$ (radius of attracting ball for the $\HoneulAlone$-norm), depending only on $\dSpace$ and $V$, such that, for every function $u_0$ in $X$ and for every large enough positive time $t$, 
\[
\norm{x\mapsto(S_t u_0)(x)}_{\Honeul} \le \RattHoneul
\,.
\]
\end{corollary}
\begin{remark}
The method used in the proof of \cref{prop:attr_ball} does not seem to extend easily to other settings, for instance if in system \cref{syst_sf} the Laplace operator $\Delta_x u$ is replaced with $\ddd\Delta_x u$, with $\ddd$ a diffusion matrix (a positive definite real symmetric $\dState\times\dState$ matrix) which differs from the identity matrix. 
\end{remark}
\subsection{Asymptotic compactness}
\begin{lemma}[asymptotic compactness]
\label{lem:compactness}
For every solution $(x,t)\mapsto u(x,t)$ of system \cref{syst_sf}, and for every sequence $(x_n,t_n)_{n\in\nn}$ in $\rr^{\dSpace}\times[0,+\infty)$ such that $t_n\to+\infty$ as $n\to+\infty$, there exists a entire solution $\widebar{u}$ of system \cref{syst_sf} in 
\[
\ccc^0\left(\rr,\cccb{2}\right)\cap \ccc^1\left(\rr,\cccb{0}\right)
\,,
\]
such that, up to replacing the sequence $(x_n,t_n)_{n\in\nn}$ by a subsequence, 
\begin{equation}
\label{compactness}
D^{2,1}u(x_n+\cdot,t_n+\cdot)\to D^{2,1}\widebar{u}
\quad\text{as}\quad
n\to+\infty
\,,
\end{equation}
uniformly on every compact subset of $\rr^{\dSpace+1}$, where the symbol $D^{2,1}v$ stands for\\
$(v,\nabla_x v,\Delta_x v,v_t)$ (for $v$ equal to $u$ or $\widebar{u}$).
\end{lemma}
\begin{proof}
See \cite[1963]{MatanoPolacik_entireSolutionBistableParabEquTwoCollidingPulses_2017} or the proof of \cite[\GlobalRelaxationLemAsymptCompactness]{Risler_globalRelaxation_2016}.
\end{proof}
\subsection{Time derivative of localized energy and \texorpdfstring{$L^2$}{L2}-norm of a solution}
Let $m$ be a point of $\mmm$ and $u:(x,t)\mapsto u(x,t)$ be a solution of system \cref{syst_sf}. In the calculations below it is assumed that time $t$ is positive, so that according to \cref{bound_u_ut_ck} the regularity properties of the solution ensure the existence of all derivatives and integrals. 
\subsubsection{Standing frame}
Let $x\mapsto \psi(x)$ denote a function in $W^{2,1}(\rr^{\dSpace},\rr)$ and let us introduce the energy (Lagrangian) and the $L^2$-distance to $m$ for the solution, localized by the weight function $\psi$:
\[
\int_{\rr^{\dSpace}} \psi(x)\biggl(\frac{1}{2}\abs{\nabla_x u(x,t)}^2+V\bigl(u(x,t)\bigr)-V(m)\biggr) \, dx
\quad\text{and}\quad
\int_{\rr^{\dSpace}} \psi(x) \frac{1}{2}\bigl(u(x,t)-m\bigr)^2 \, dx
\,.
\]
To simplify the presentation, let us assume here that 
\[
m=0_{\rr^{\dState}}
\quad\text{and}\quad
V(m)=V(0_{\rr^{\dState}}) = 0
\,.
\]
Then, the time derivatives of these two functionals read:
\begin{equation}
\label{ddt_loc_en_stand_fr}
\frac{d}{dt} \int_{\rr^{\dSpace}} \psi \biggl(\frac{1}{2}\abs{\nabla_x u}^2+V(u)\biggr) \, dx = 
\int_{\rr^{\dSpace}} \bigl( - \psi u_t^2 - \nabla_x \psi \cdot \nabla_x u \cdot u_t \bigr) \, dx 
\end{equation}
and
\begin{equation}
\label{ddt_loc_L2_stand_fr}
\begin{aligned}
\frac{d}{dt} \int_{\rr^{\dSpace}} \psi \frac{1}{2}u^2 \, dx &= 
\int_{\rr^{\dSpace}} \Bigl[( \psi \bigl( - u \cdot \nabla V(u) - \abs{\nabla_x u}^2 \bigr) - \nabla_x \psi \cdot u \cdot \nabla_x u \Bigr] \, dx \\
&= \int_{\rr^{\dSpace}} \Bigl[ \psi \bigl( - u \cdot \nabla V(u) - \abs{\nabla_x u}^2 \bigr) +  \frac{1}{2}\Delta_x \psi\,\, u^2 \Bigr] \, dx
\,.
\end{aligned}
\end{equation}
\subsubsection{Travelling frame}
Let $c$ denote a vector of $\rr^{\dSpace}$ (velocity vector) and let us introduce the same solution viewed in a frame travelling at velocity $c$, that is the function $(\xi,t)\mapsto v(\xi,t)$ defined 
for every $\xi$ in $\rr^{\dSpace}$ and nonnegative time $t$ by
\[
v(\xi,t) = u(x,t) 
\quad\text{for}\quad
x = ct + \xi 
\,.
\]
This function $v$ is a solution of system \vref{syst_tf}. This time, let us introduce a weight function $(\xi,t)\mapsto \psi(\xi,t)$ (depending both on space and time and defined on $\rr^{\dSpace} \times [0,+\infty)$, with values in $\rr$), and such that, for every nonnegative time $t$ the function $\xi\mapsto \psi(\xi,t)$ belongs to $W^{2,1}(\rr^{\dSpace},\rr)$ and the function $\xi\mapsto \psi_t(\xi,t)$ is defined and belongs to $L^1(\rr^{\dSpace},\rr)$. Again, let us introduce the energy (Lagrangian) and the $L^2$-norm of the solution (in travelling frame), localized by the weight function $\psi$:
\[
\int_{\rr^{\dSpace}} \psi(\xi,t)\biggl(\frac{1}{2}\abs{\nabla_\xi v(\xi,t)}^2+V\bigl(v(\xi,t)\bigr)\biggr) \, d\xi
\quad\text{and}\quad
\int_{\rr^{\dSpace}} \psi(\xi,t) \frac{1}{2}v(\xi,t)^2 \, d\xi
\,.
\]
The time derivatives of these two quantities read:
\begin{equation}
\label{ddt_loc_en_trav_fr}
\begin{aligned}
\frac{d}{dt} \int_{\rr^{\dSpace}} \psi &\biggl( \frac{1}{2}\abs{\nabla_\xi v}^2+  V(v)\biggr) \, d\xi = \\
& \int_{\rr^{\dSpace}} \Biggl[ -\psi v_t^2 + \psi_t \biggl(\frac{1}{2}\abs{\nabla_\xi v}^2+V(v)\biggr) + (\psi c  - \nabla_\xi \psi) \cdot \nabla_\xi v \cdot v_t \Biggr] \, d\xi 
\end{aligned}
\end{equation}
and
\begin{equation}
\label{ddt_loc_L2_trav_fr}
\begin{aligned}
\frac{d}{dt} &\int_{\rr^{\dSpace}} \psi \frac{1}{2}v^2 \, d\xi \\
&= \int_{\rr^{\dSpace}} \left[ \psi \bigl( - v\cdot \nabla V(v) - \abs{\nabla_\xi v}^2 \bigr) + \frac{1}{2}(\psi_t - c \cdot \nabla_\xi \psi)v^2 - \nabla_\xi \psi \cdot v\cdot \nabla_\xi v \right] \, d\xi \\
&= \int_{\rr^{\dSpace}} \left[ \psi \bigl( - v\cdot \nabla V(v) - \abs{\nabla_\xi v}^2 \bigr) + \frac{1}{2}(\psi_t + \Delta_\xi \psi - c \cdot \nabla_\xi \psi)v^2 \right] \, d\xi 
\,.
\end{aligned}
\end{equation}
\subsection{Miscellanea}
\label{subsec:misc}
\subsubsection{Escape distance}
\label{subsubsec:esc_dist}
\begin{notation}
For every $u$ in $\rr^{\dState}$, let $\sigma\bigl(D^2V(u)\bigr)$ denote the spectrum (the set of eigenvalues) of the Hessian matrix of $V$ at $u$, and let $\eigVmin(u)$ denote the minimum of this spectrum:
\begin{equation}
\label{def_eigVmin_of_u}
\eigVmin(u) = \min \Bigl(\sigma\bigl(D^2V(u)\bigr)\Bigr)
\,.
\end{equation}
\end{notation}
\begin{definition}[Escape distance of a nondegenerate minimum point]
\label{def:Escape_distance}
For every $m$ in $\mmm$, let us call \emph{Escape distance of $m$}, and let us denote by $\dEsc(m)$, the supremum of the set
\begin{equation}
\label{set_for_definition_Escape_distance}
\Bigl\{\delta \in[0,1]: \text{ for all } u \text{ in } \rr^{\dState} \text{ satisfying } \abs{u-m}\le \delta, \quad\eigVmin(u) \ge\frac{1}{2} \eigVmin(m) \Bigr\}
\,.
\end{equation}
\end{definition}
Since the quantity $\eigVmin(u)$ varies continuously with $u$, this Escape distance $\dEsc(m)$ is positive (thus in $(0,1]$). In addition, for all $u$ in $\rr^{\dState}$ such that $\abs{u-m}$ is not larger than $\dEsc(m)$, the following inequality holds:
\begin{equation}
\label{property_dEsc}
\eigVmin(u) \ge\frac{1}{2} \eigVmin(m)
\,.
\end{equation}

\begin{lemma}[second order estimates for the potential around a minimum point]
\label{lem:estim_from_def_escape}
For every minimum point $m$ in $\mmm$ and every vector $u$ in $\rr^{\dState}$ satisfying $\abs{u-m} \le\dEsc(m)$, the following estimates hold:
\begin{align}
\label{posit_pot_around_loc_min}
Vu) &\ge \frac{\eigVmin(m)}{4} (u-m)^2 \,, \\
\text{and}\qquad
\label{v_nablaV_controls_square_around_loc_min}
(u-m)\cdot \nabla V(u) &\ge \frac{\eigVmin(m)}{2} (u-m)^2 \,, \\
\text{and}\qquad
\label{v_nablaV_controls_pot_around_loc_min}
(u-m)\cdot \nabla V(u) &\ge V(u)
\,.
\end{align}
\end{lemma}
\begin{proof}
The three inequalities follow from inequality \cref{property_dEsc} and from three variants of Taylor's Theorem with Lagrange remainder applied to the function defined on $[0,1]$ by: $f(\theta) =  V\bigl(m+\theta (u-m)\bigr)$ (see \cite[\GlobalRelaxationLemEstimFromEscDist]{Risler_globalRelaxation_2016}).
\end{proof}
\subsubsection{Lower quadratic hulls of the potential at minimum points}
\label{subsubsec:low_quad_hull}
It will be convenient to introduce the quantity $\qLowHull$ defined as the minimum of the convexities of the lower quadratic hulls of $V$ at the points of $\mmm$. 
\begin{figure}[!htbp]
\centering
\includegraphics[width=.5\textwidth]{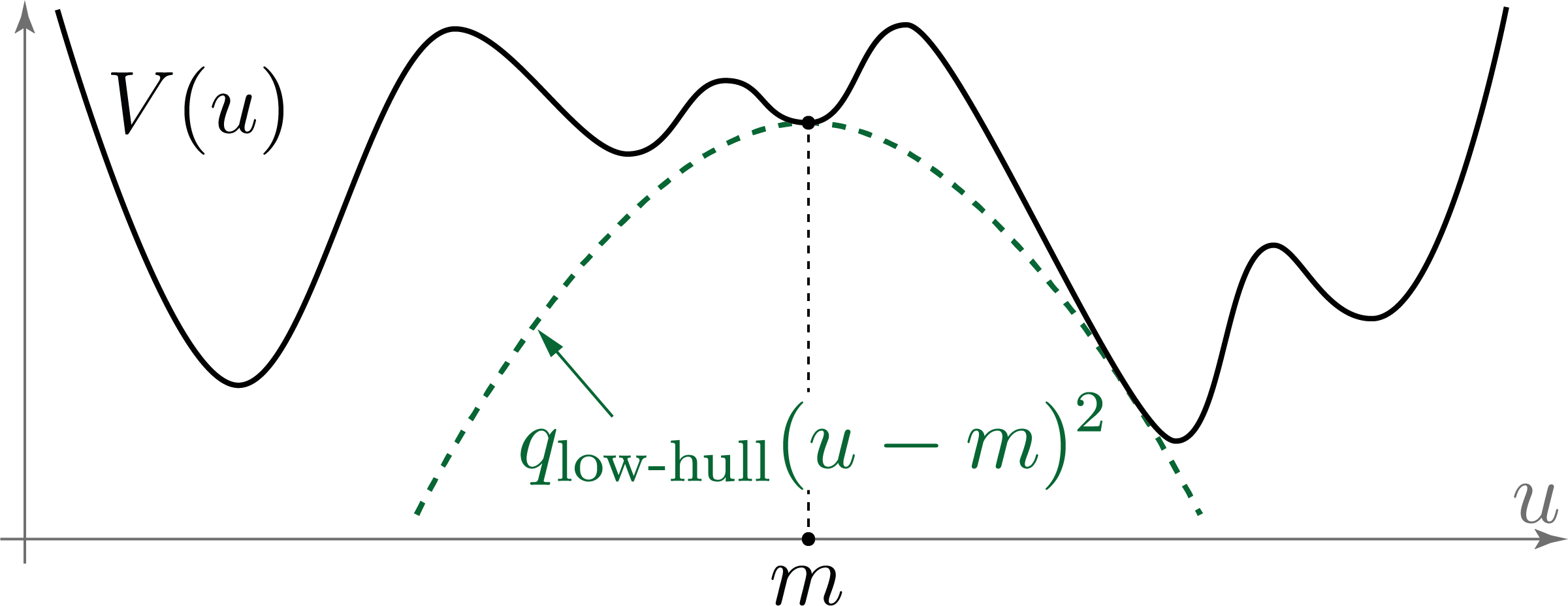}
\caption{Lower quadratic hull of the potential at a minimum point (definition of the quantity $\qLowHull$).}
\label{fig:low_quad_hull}
\end{figure}
With symbols:
\[
\qLowHull = \min_{m\in\mmm}\ \inf_{u\in\rr^{\dState}\setminus \{m\}}\ \frac{V(u)-V(m)}{(u-m)^2}
\,,
\]
see \cref{fig:low_quad_hull}. This quantity $\qLowHull$ is negative as soon as $m$ is not a global minimum point of $V$ (and nonnegative otherwise), and according to hypothesis \cref{hyp_coerc} it is finite (in other words it is not equal to $-\infty$). This definition ensures that, for every $m$ in $\mmm$ and for all $u$ in $\rr^{\dState}$,
\begin{equation}
\label{ineq_neg_hull}
V(u)-V(m) - \qLowHull (u-m)^2\ge 0
\,,
\end{equation}
see \cref{fig:low_quad_hull}. Let us introduce the following quantity:
\[
\coeffEn=\frac{1}{\max(1, - 4\, \qLowHull)}
\,.
\]
It follows from this definition that $\coeffEn$ is in $(0,1]$ and that, for every $m$ in $\mmm$ and for all $u$ in $\rr^{\dState}$,
\begin{equation}
\label{def_weight_en}
\coeffEn \bigl(V(u)-V(m)\bigr) +  \frac{1}{4}(u-m)^2\ge 0
\,.
\end{equation}
\subsubsection{Notation for the values of certain integrals}
\label{subsubsec:notation_for_integrals}
The following notation and the expression \cref{integral_of_r_power_n_exp_minus_r} below will be useful at several places throughout the paper (to provide expressions for integrals over $\rr^{\dSpace}$). 
\begin{notation}
Let $S_{\dSpace-1}$ denote the surface area (more precisely the $\dSpace-1$-dimensional volume) of the unit sphere in $\rr^{\dSpace}$. 

For every nonnegative integer $n$, let $e_n$ denote the \emph{exponential sum function}: 
\[
e_n(\tau) = \sum_{k=0}^{n} \frac{\tau^k}{k!}
\,.
\]
\end{notation}
For every nonnegative quantity $\rho_0$ and every nonnegative integer $n$, it follows from repeated integrations by parts that
\begin{equation}
\label{integral_of_r_power_n_exp_minus_r}
\int_{\rho_0}^{+\infty} \rho^n e^{-\rho}\, d\rho = n!\, e^{-\rho_0}\, e_n(\rho_0)
\,.
\end{equation}
\section{Stability at infinity}
\subsection{Set-up}
\label{subsec:set_up_stability_at_infinity}
As everywhere else, let us consider a function $V$ in $\ccc^2(\rr^{\dState},\rr)$ satisfying the coercivity hypothesis \cref{hyp_coerc}. Let $m$ be a point of $\mmm$, and let $(x,t)\mapsto u(x,t)$ be a solution of system \cref{syst_sf}. According to \cref{prop:attr_ball}, there exists a positive quantity $\RmaxInfty$ (depending on the solution $u$ under consideration) such that, for every nonnegative time $t$, 
\begin{equation}
\label{maximal_radius_excursion_Linfty}
\sup_{x\in\rr^{\dSpace}}\abs{u(x,t)}\le \RmaxInfty
\,.
\end{equation} 
For notational convenience, let us introduce the ``normalized potential'' $V^\dag$ and the ``normalized solution'' $u^\dag$ defined as
\begin{equation}
\label{def_normalized_potential_solution}
V^\dag(v)=V(m +v)-V(m)
\quad\text{and}\quad
u^\dag(x,t) = u(x,t)-m
\,.
\end{equation}
Thus the origin $0_{\rr^{\dState}}$ of $\rr^{\dState}$ is to $V^\dag$ what $m$ is to $V$, and $u^\dag$ is a solution of system \cref{syst_sf} with potential $V^\dag$ instead of $V$.
\subsection{Application of the maximum principle}
\label{subsec:application_max_principle}
\subsubsection{Local coercivity}
\label{subsubsec:local_coercivity}
Let $\delta$ be a positive quantity, small enough so that, for every $v$ in $\rr^{\dState}$ such that $\abs{v-m}$ is not larger than $\delta$, the Hessian matrix $D^2V(v)$ is positive definite. 
Since the fact that $D^2V(v)$ is positive definite is an open condition with respect to $v$, there exist quantities $\delta'$ and $\delta''$, satisfying
\[
\delta <\delta'<\delta''
\,,
\]
and such that, for every $v$ in $\rr^{\dState}$ such that $\abs{v-m}$ is not larger than $\delta''$, the Hessian matrix $D^2V(v)$ is positive definite. Thus the quantity $\lambda$ defined as
\[
\lambda = 2\min_{v\in\rr^{\dState},\,\abs{v-m}\le\delta''}\eigVmin(v)
\]
is positive (the factor ``$2$'' is here only to ensure homogeneity with the notation of \cref{subsubsec:esc_dist}).
\begin{lemma}[local coercivity]
\label{lem:local_coercivity}
For every $v$ in $\rr^{\dState}$ such that $\abs{v-m}$ is not larger than $\delta''$, the following inequality holds:
\begin{equation}
\label{lower_bound_v_minus_m_cdot_nabla_V_of_v}
(v-m)\cdot \nabla V(v) \ge \frac{\lambda}{2} v^2
\,.
\end{equation}
\end{lemma}
\begin{proof}
Let us introduce the function 
\[
f:[0,1]\to\rr,\quad v\mapsto V(m+\theta v)
\,.
\]
According to Taylor's Theorem with Lagrange remainder, there exists $\theta$ in $(0,1)$ such that 
\[
f'(1)-f'(0) = f''(\theta)
\quad\text{or in other words}\quad
v\cdot \nabla V(m+v) - 0 = D^2 V(\theta v)\cdot v \cdot v 
\,,
\]
and inequality \cref{lower_bound_v_minus_m_cdot_nabla_V_of_v} follows. 
\end{proof}
With the notation introduced in \cref{def_normalized_potential_solution}, it follows from \cref{lower_bound_v_minus_m_cdot_nabla_V_of_v} that, for every $w$ in $\rr^{\dState}$ such that $\abs{w}$ is not larger than $\delta''$, 
\begin{equation}
\label{lower_bound_w_minus_m_cdot_nabla_Vdag_of_w}
w\cdot \nabla V^\dag(w) \ge \frac{\lambda}{2} w^2
\,.
\end{equation}
\subsubsection{Potential \texorpdfstring{$V^\ddag$}{Vddag} quadratic at infinity}
\label{subsubsec:_potential_quadratic_at_infinity}
Let $\chi:\rr\to[0,1]$ be a smooth cutoff function satisfying:
\[
\text{$\chi(x)=1$ for all $x$ in $(-\infty,0]$ and $\chi(x)=0$ for all $x$ in $[1,+\infty)$}, 
\]
and let us introduce the potential function $V^\ddag$ defined as
\begin{equation}
\label{Vddag}
V^\ddag(v) = \chi\bigl(\abs{v}-\RmaxInfty-\abs{m}\bigr) V^\dag(v) + \Bigl(1-\chi\bigl(\abs{v}-\RmaxInfty-\abs{m}\bigr)\Bigr)\frac{\abs{v}^2}{2}
\,.
\end{equation}
Thus 
\[
V^\ddag(v) = V^\dag(v) \text{ for } \abs{v}\le \RmaxInfty+\abs{m}
\quad\text{and}\quad
V^\ddag(v) = \abs{v}^2/2 \text{ for } \abs{v}\ge \RmaxInfty+\abs{m} + 1
\,.
\]
It follows from \cref{maximal_radius_excursion_Linfty} that $u^\dag$ is still a solution of system \cref{syst_sf} with potential $V^\ddag$ instead of $V$:
\begin{equation}
\label{syst_sf_udag}
u^\dag_t = -\nabla V^\ddag (u^\dag) + \Delta_x u^\dag
\,.
\end{equation}
Let us introduce the quantity 
\begin{equation}
\label{def_Rcoerc}
\Rcoerc = \RmaxInfty + \abs{m} + 1 + \sqrt{\frac{\lambda(\delta')^2}{2}}
\,.
\end{equation}
Since $v\cdot\nabla V^\ddag (v)$ equals $\abs{v}^2$ for $\abs{v}$ larger than $\RmaxInfty + \abs{m} + 1$, the following property holds (this property will be used to prove inequality \cref{upper_bound_bar_N_of_q_for_q_large} in the proof of \cref{lem:super_solution} below):
\begin{equation}
\label{coercivity_away_from_radius_Rcoerc}
\abs{v}\ge \Rcoerc \implies v\cdot \nabla V^\ddag(v) \ge \frac{\lambda(\delta')^2}{2}
\,.
\end{equation}
\subsubsection{Quadratic state function}
Let us introduce the ``quadratic state'' function $q^\dag:\rr^{\dSpace}\times[0,+\infty)\to\rr$ defined as
\[
q^\dag(x,t) = \frac{1}{2} \abs{u^\dag(x,t)}^2
\,,
\]
and the function 
\begin{equation}
\label{def_bar_N}
\widebar{N}:[0,+\infty)\to \rr\,, 
\quad
q\mapsto \max\Bigl\{
-v\cdot\nabla V^\ddag(v) : v \in\rr^{\dState} \text{ and } \frac{1}{2}v^2 = q 
\Bigr\}
\,, 
\end{equation}
(``upper-bound on the nonlinear part'') and the system
\begin{equation}
\label{system_governing_super_solution_q}
q_t = \widebar{N}(q) + \Delta_x q
\,.
\end{equation}
\begin{lemma}[the quadratic state function is a subsolution of system \cref{system_governing_super_solution_q}]
\label{lem:subsolution}
The function $(x,t)\mapsto q^\dag(x,t)$ is a subsolution of system \cref{system_governing_super_solution_q}. 
\end{lemma}
\begin{proof}
The same calculation as in the proof of \cref{prop:attr_ball} shows that the function $q^\dag(\cdot,\cdot)$ is a solution of the system:
\[
q^\dag_t = - u^\dag\cdot\nabla V^\ddag(u^\dag) + \Delta_x q^\dag - \abs{\nabla_x u^\dag}^2
\,,
\]
and as a consequence, for all $(x,t)$ in $\rr^{\dSpace}\times[0,+\infty)$, 
\[
q^\dag_t \le - u^\dag\cdot\nabla V^\ddag(u^\dag) + \Delta_x q^\dag
\,.
\]
The conclusion follows from the definition \cref{def_bar_N} of $\widebar{N}(\cdot)$. \Cref{lem:subsolution} is proved. 
\end{proof}
The next goal is to construct an appropriate \emph{supersolution} above this subsolution $q^\dag(\cdot,\cdot)$, for the same system \cref{system_governing_super_solution_q}. 
\subsubsection{Construction of a supersolution}
To simplify the expressions below, let us introduce the quantities
\[
\gamma' = \frac{1}{2}(\delta')^2 
\quad\text{and}\quad
\gamma'' = \frac{1}{2}(\delta'')^2
\,,
\]
and
\begin{equation}
\label{def_qMax}
\qMax = \frac{1}{2}\Rcoerc^2
\,.
\end{equation}
Observe that, according to the definition \cref{def_Rcoerc} of the quantity $\Rcoerc$, 
\begin{equation}
\label{global_upper_bound_qMax}
\sup_{x\in\rr^{\dSpace}}q^\dag(x,0)\le \qMax
\,.
\end{equation}
Let us introduce the function $\eta_0:\rr\to\rr$, defined as
\begin{equation}
\label{def_eta_zero}
\eta_0(\rho) = \left\{
\begin{aligned}
\qMax + (\gamma''-\gamma') \quad&\text{if}\quad \rho\le 0\,,\\
\qMax + (\gamma''-\gamma')(1-\rho) \quad&\text{if}\quad  0\le \rho \le \frac{\qMax}{\gamma''-\gamma'}\,,\\
(\gamma''-\gamma')\exp\left(-\rho+\frac{\qMax}{\gamma''-\gamma'}\right) \quad&\text{if}\quad \frac{\qMax}{\gamma''-\gamma'} \le \rho 
\,,
\end{aligned}
\right.
\end{equation}
see \cref{fig:graph_of_eta_zero}.
\begin{figure}[!htbp]
\centering
\includegraphics[width=.7\textwidth]{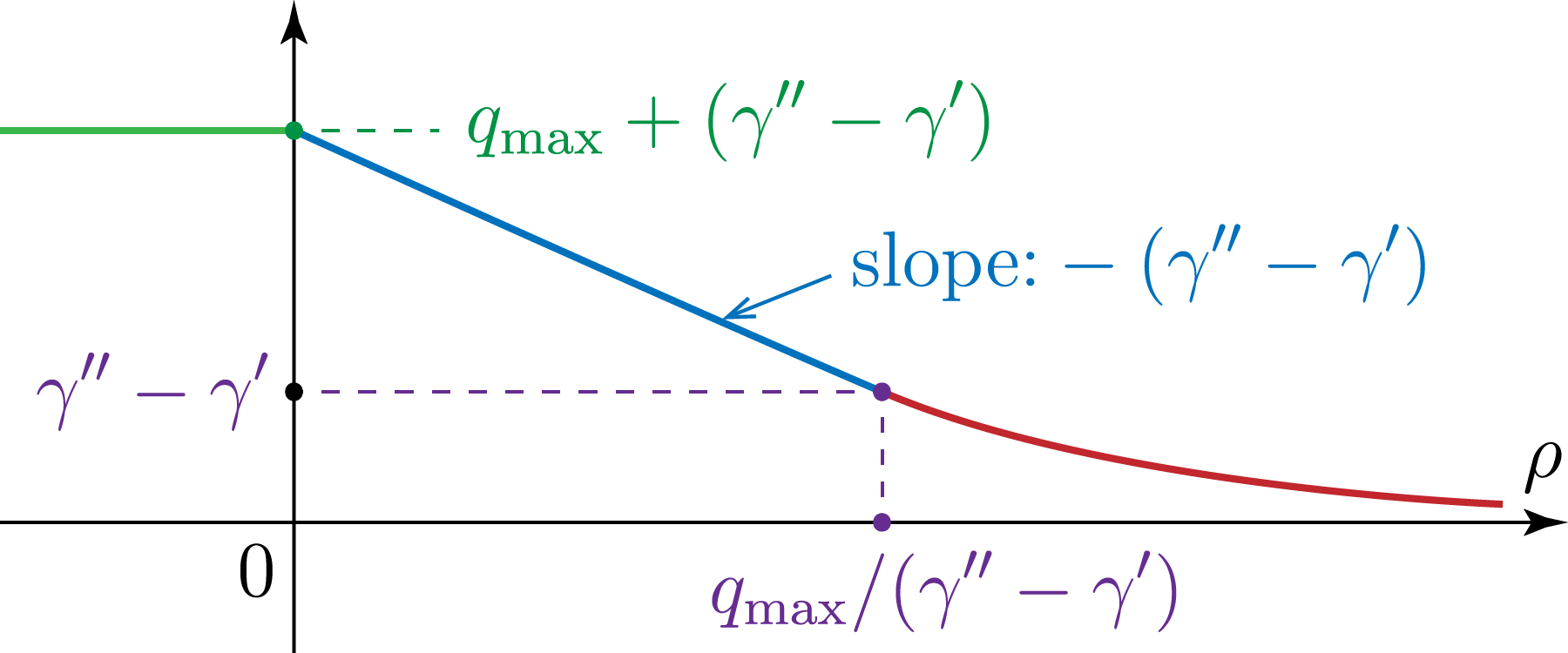}
\caption{Graph of $\rho\mapsto\eta_0(\rho)$.}
\label{fig:graph_of_eta_zero}
\end{figure}
Let $c$ and $\rEscInit$ denote positive quantities, and let us introduce the function
\begin{equation}
\label{def_eta}
\eta:[0,+\infty)\times[0,+\infty)\to\rr, \quad (r,t) \mapsto \eta_0\bigl(r-ct-\rEscInit\bigr)
\,,
\end{equation}
see \cref{fig:graph_of_eta}.
\begin{figure}[!htbp]
\centering
\includegraphics[width=\textwidth]{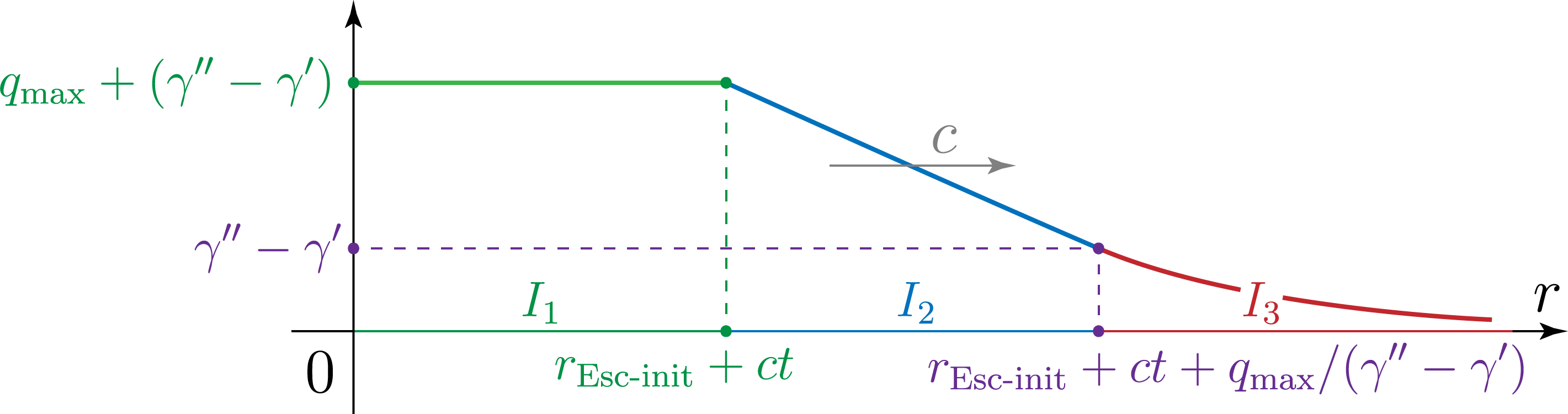}
\caption{Graph of $r\mapsto\eta(r,t)$.}
\label{fig:graph_of_eta}
\end{figure}
Finally, let us introduce the function
\[
\widebar{\eta}:\rr^{\dSpace}\times[0,+\infty)\to\rr, \quad (x,t) \mapsto \eta(\abs{x},t) + \gamma' \exp(-\lambda t)
\,,
\]
and the quantity
\begin{equation}
\label{def_cnoesc}
\cnoesc = \max\Biggl(\frac{1}{\gamma''-\gamma'}\biggl(\lambda\gamma' + \sup_{q\in[0,\qMax+\gamma'']}\widebar{N}(q)\biggr),1-\lambda\Biggr)
\,.
\end{equation}
\begin{lemma}[supersolution]
\label{lem:super_solution}
If the positive quantity $c$ is not smaller than $\cnoesc$, then the function $(x,t)\mapsto\widebar{\eta}(x,t)$ is a supersolution of system \cref{system_governing_super_solution_q} (for all $(x,t)$ in $\rr^{\dSpace}\times[0,+\infty)$.
\end{lemma}
\begin{proof}
This amounts to prove that, if $c$ is not smaller than $\cnoesc$, then the following inequality holds for all $(x,t)$ in $\rr^{\dSpace}\times[0,+\infty)$:
\[
\widebar{\eta}_t(x,t) \ge \widebar{N}\bigl(\widebar{\eta}(x,t)\bigr) + \Delta_x \widebar{\eta} (x,t)
\,,
\]
or equivalently, for all $(r,t)$ in $[0,+\infty)\times[0,+\infty)$, 
\[
-c \partial_r\eta(r,t) - \lambda\gamma' \exp(-\lambda t) \ge \widebar{N}\bigl(\eta(r,t)+ \gamma' \exp(-\lambda t)\bigr) + \frac{\dSpace-1}{r}\partial_r\eta(r,t) + \partial_{rr}\eta(r,t)
\,.
\]
Since $\partial_r\eta(\cdot,\cdot)$ is nonpositive, it is sufficient to prove this inequality without the curvature term, that is, for all $(r,t)$ in $[0,+\infty)\times[0,+\infty)$, 
\begin{equation}
\label{super_solution_inequality}
-c \partial_r\eta(r,t) - \lambda\gamma' \exp(-\lambda t) \ge \widebar{N}\bigl(\eta(r,t)+ \gamma' \exp(-\lambda t)\bigr) + \partial_{rr}\eta(r,t)
\,.
\end{equation}
Let us introduce the three intervals
\[
\begin{aligned}
I_1 &= [0,\rEscInit+ct]
\,, \\
\text{and}\quad
I_2 &= [\rEscInit+ct,\rEscInit+ct+\qMax/(\gamma''-\gamma')]
\,, \\
\text{and}\quad
I_3 &= \bigl[\rEscInit+ct+\qMax/(\gamma''-\gamma'),+\infty\bigr)
\,,
\end{aligned}
\]
see \cref{fig:graph_of_eta}. Three cases are to be distinguished, depending on which among the three intervals above contains the argument $r$. Observe that the discontinuity of $\partial_r\eta(\cdot,t)$ at $\rEscInit+ct$ can be ignored, since its contribution only enforces inequality \cref{super_solution_inequality} (it adds a negative Dirac mass to the right-hand side of this inequality).
\paragraph*{Case 1:} $r$ is in $I_1$.
In this case both quantities $\partial_r\eta(r,t)$ and $\partial_{rr}\eta(r,t)$ vanish and $\eta(r,t)$ equals $\qMax + (\gamma''-\gamma')$, so that inequality \cref{super_solution_inequality} reduces to
\[
\widebar{N}\bigl(\qMax + (\gamma''-\gamma')+ \gamma' \exp(-\lambda t)\bigr)\le - \lambda\gamma' \exp(-\lambda t)
\,, 
\]
which is implied by
\[
\widebar{N}\bigl(\qMax + (\gamma''-\gamma')+ \gamma' \exp(-\lambda t)\bigr) \le - \lambda\gamma'
\,,
\]
and thus by 
\begin{equation}
\label{upper_bound_bar_N_of_q_for_q_large}
\sup_{q\ge\qMax} \widebar{N}(q)\le - \lambda\gamma'
\,.
\end{equation}
According to the definition \cref{def_qMax} of $\qMax$, for every quantity $q$ not smaller than $\qMax$ and for every $v$ in $\rr^{\dState}$ such that $v^2/2$ equals $q$, the following holds:
\[
q\ge\frac{1}{2}\Rcoerc^2
\ \text{thus}\ 
\abs{v}\ge \Rcoerc
\ \text{thus, according to \cref{coercivity_away_from_radius_Rcoerc},}\ 
-v\cdot\nabla V^\ddag(v)\le -\frac{\lambda(\delta')^2}{2} = -\lambda\gamma'
\,,
\]
and inequality \cref{upper_bound_bar_N_of_q_for_q_large} follows.
\paragraph*{Case 2:} $r$ is in $I_2$.
In this case the quantity $\partial_{rr}\eta(r,t)$ vanishes and the quantity $\partial_r\eta(r,t)$ equals $-(\gamma''-\gamma')$, so that inequality \cref{super_solution_inequality} reduces to
\[
c(\gamma''-\gamma')\ge \lambda\gamma' \exp(-\lambda t) + \widebar{N}\bigl(\eta+ \gamma' \exp(-\lambda t)\bigr)
\,,
\]
which is implied by
\[
c\ge\frac{1}{\gamma''-\gamma'}\biggl(\lambda\gamma' + \sup_{q\in[0,\qMax+\gamma'']}\widebar{N}(q)\biggr)
\,.
\]
\paragraph*{Case 3:} $r$ is $I_3$. 
In this case, it follows from the expressions \cref{def_eta_zero,def_eta} of $\eta_0$ and $\eta$ that
\[
\partial_r\eta(r,t) = - \eta(r,t)
\quad\text{and}\quad
\partial_{rr}\eta(r,t) = \eta(r,t)
\,.
\]
Let us introduce the function $q$ defined as
\[
q(r,t) = \eta(r,t) + \gamma'\exp(-\lambda t)
\,.
\]
It follows from the expressions of $\eta_0$ and $\eta$ that $q(r,t)$ is not larger than $\gamma''$. As a consequence, for every $v$ in $\rr^{\dSpace}$, if $v^2/2$ is equal to $q(r,t)$ then 
\[
\abs{v} = \sqrt{2q(r,t)} \le \sqrt{2\gamma''} = \delta''
\,,
\]
thus it follows from \cref{lower_bound_w_minus_m_cdot_nabla_Vdag_of_w} that
\[
-v\cdot\nabla V^\ddag(v)\le - \frac{\lambda}{2}v^2 = - \lambda q
\,,\quad\text{and it follows that}\quad
\widebar{N}\bigl(q(r,t)\bigr)\le-\lambda q(r,t)
\,.
\]
As a consequence, inequality \cref{super_solution_inequality} is implied in this case by the following inequality:
\[
c\eta(r,t) - \lambda\gamma'\exp(-\lambda t)\ge -\lambda\eta(r,t) - \lambda \gamma'\exp(-\lambda t) + \eta(r,t)
\,,
\]
or equivalently
\[
(c+\lambda-1)\eta(r,t)\ge 0
\,,
\quad\text{or equivalently}\quad
c\ge 1-\lambda
\,.
\]
\paragraph*{Conclusion.} In short, if $c$ is not smaller than the quantity $\cnoesc$ defined in \cref{def_cnoesc}, then inequality \cref{super_solution_inequality} holds for all $(r,t)$ in $[0,+\infty)\times[0,+\infty)$. \Cref{lem:super_solution} is proved.
\end{proof}
\subsubsection{Proof of \texorpdfstring{\cref{lem:sufficient_condition_stability_at_infinity}}{Lemma \ref{lem:sufficient_condition_stability_at_infinity}}}
\label{subsubsec:proof_of_lem_sufficient_condition_stability_at_infinity}
\begin{proof}[Proof of \cref{lem:sufficient_condition_stability_at_infinity}]
Let us assume that the solution $(x,t)\mapsto u(x,t)$ under consideration satisfies hypothesis \cref{hyp_lem_sufficient_condition_stability_at_infinity} of \cref{lem:sufficient_condition_stability_at_infinity}. Then, if the positive quantity $\rEscInit$ is large enough, 
\begin{align}
\sup_{x\in\rr^{\dSpace}, \, \abs{x}\ge\rEscInit}\abs{u^\dag(x,t=0)}&\le \delta'
\,, 
\nonumber \\
\text{or in other words}\quad
\sup_{x\in\rr^{\dSpace}, \, \abs{x}\ge\rEscInit}\abs{q^\dag(x,t=0)}&\le \gamma'
\,.
\label{q_of_x_zero_small_for_x_large}
\end{align}
Then it follows from inequalities \cref{q_of_x_zero_small_for_x_large,global_upper_bound_qMax} that, for all $x$ in $\rr^{\dSpace}$, 
\[
q^\dag(x,0) \le \widebar{\eta}(x,0)
\,.
\]
Since $q^\dag(\cdot,\cdot)$ is a subsolution of system \cref{system_governing_super_solution_q} and since, according to \cref{lem:super_solution} and provided that $c$ is not smaller than the quantity $\cnoesc$, the function $\widebar{\eta}(\cdot,\cdot)$ is a supersolution of the same system, it follows from the maximum principle that, for all $(x,t)$ in $\rr^{\dSpace}\times[0,+\infty)$,
\[
q^\dag(x,t) \le \widebar{\eta}(x,t)
\,.
\]
Since
\[
\lim_{r\to+\infty}\sup_{x\in\rr^{\dSpace},\, \abs{x}\ge r} \widebar{\eta}(x,t) \to 0
\quad\text{as}\quad
t\to+\infty
\,,
\]
it follows that the solution $u(\cdot,\cdot)$ is stable close to $m$ at infinity (\cref{def:solution_stable_at_infinity}). \Cref{lem:sufficient_condition_stability_at_infinity} is proved.
\end{proof}
\subsubsection{Supersolution with uniform features}
\label{subsubsec:uniform_supersolution}
The aim of this \namecref{subsubsec:uniform_supersolution} is to formulate a more uniform version of \cref{lem:super_solution} leading to a proof of \cref{lem:upper_bound_on_invasion_speed}. Let us introduce the potential function $\VddagAtt$ defined as
\[
\VddagAtt(v) = \chi\bigl(\abs{v}-\RattInfty-\abs{m}\bigr) V^\dag(v) + \Bigl(1-\chi\bigl(\abs{v}-\RattInfty-\abs{m}\bigr)\Bigr)\frac{\abs{v}^2}{2}
\,,
\]
and the function
\[
\barNAtt:[0,+\infty)\to \rr, 
\quad
q\mapsto \max\Bigl\{
-v\cdot\nabla \VddagAtt(v) : v \in\rr^{\dState} \text{ and } \frac{1}{2}v^2 = q 
\Bigr\}
\,, 
\]
and the system
\begin{equation}
\label{system_governing_super_solution_q_att}
q_t = \barNAtt(q) + \Delta_x q
\,,
\end{equation}
and the quantities
\[
\RcoercAtt = \RattInfty + \abs{m} + 1 + \sqrt{\frac{\lambda(\delta')^2}{2}}
\quad\text{and}\quad
\qAtt = \frac{1}{2}\RcoercAtt^2 
\,,
\]
and the function $\etaZeroAtt:\rr\to\rr$, defined as
\[
\etaZeroAtt(\rho) = \left\{
\begin{aligned}
\qAtt + (\gamma''-\gamma') \quad&\text{if}\quad \rho\le 0\,,\\
\qAtt + (\gamma''-\gamma')(1-\rho) \quad&\text{if}\quad  0\le \rho \le \frac{\qAtt}{\gamma''-\gamma'}\,,\\
(\gamma''-\gamma')\exp\left(-\rho+\frac{\qAtt}{\gamma''-\gamma'}\right) \quad&\text{if}\quad \frac{\qAtt}{\gamma''-\gamma'} \le \rho 
\,,
\end{aligned}
\right.
\]
and the quantity 
\begin{equation}
\label{def_cnoescAtt}
\cnoescAtt = \max\Biggl(\frac{1}{\gamma''-\gamma'}\biggl(\lambda\gamma' + \sup_{q\in[0,\qAtt+\gamma'']}\widebar{N}(q)\biggr),1-\lambda\Biggr)
\,,
\end{equation}
and the functions
\[
\begin{aligned}
\etaAtt&:\times[0,+\infty)\times[0,+\infty)\to\rr, &\quad (r,t) &\mapsto \etaZeroAtt\bigl(r-\cnoescAtt t-\rEscInit\bigr) \\
\text{and}\quad
\baretaAtt&:\rr^{\dSpace}\times[0,+\infty)\to\rr, &\quad (x,t) &\mapsto \etaAtt(\abs{x},t) + \gamma' \exp(-\lambda t)
\,.
\end{aligned}
\]
Observe that the quantity $\cnoescAtt$ depends on $V$ and $\dSpace$ and the choice of $\delta$ and $\delta'$ and $\delta''$ (in other words on $m$) but not on the solution $u(\cdot,\cdot)$ under consideration.
\begin{lemma}[supersolution]
\label{lem:super_solution_att}
Whatever the value of the positive quantity $\rEscInit$, the function $(x,t)\mapsto\baretaAtt(x,t)$ is a supersolution of system \cref{system_governing_super_solution_q_att} (for all $(x,t)$ in $\rr^{\dSpace}\times[0,+\infty)$.
\end{lemma}
\begin{proof}
The proof is the same as that of \cref{lem:super_solution}. 
\end{proof}
\subsubsection{Proof of \texorpdfstring{\cref{lem:upper_bound_on_invasion_speed}}{Lemma \ref{lem:upper_bound_on_invasion_speed}}}
\Cref{lem:upper_bound_on_invasion_speed} follows from the next more precise lemma. 
\begin{lemma}[``explicit'' upper bound on the invasion speed]
\label{lem:explicit_upper_bound_invasion_speed}
If the solution $(x,t)\mapsto u(x,t)$ under consideration is stable close to $m$ at infinity, then 
\begin{equation}
\label{explicit_upper_bound_invasion_speed}
\cInv[u]\le \cnoescAtt
\,.
\end{equation}
\end{lemma}
\begin{proof}
Let us assume that the solution $(x,t)\mapsto u(x,t)$ under consideration is stable close to $m$ at infinity. Then, according to \cref{prop:attr_ball}, there exists a positive time $\tInit$ such that, for every time $t$ greater than or equal to $\tInit$, 
\begin{equation}
\label{upper_bound_RattInfty_proof_a_priori_bound_invasion_speed}
\sup_{x\in\rr}\abs{u(x,t)}\le\RattInfty
\,.
\end{equation}
In addition, according to \cref{def:solution_stable_at_infinity} of a solution stable at infinity, it may be assumed, up to replacing $\tInit$ by a greater quantity, that for every time $t$ greater than or equal to $\tInit$, the limit
\[
\lim_{r\to+\infty}\sup_{x\in\rr^{\dSpace},\, \abs{x}\ge r}\abs{u^\dag(x,t)}
\]
is not larger than any given positive quantity; for instance, it may be assumed that this limit is not larger than the quantity $\delta$ defined in \cref{subsubsec:local_coercivity}. As a consequence, if the positive quantity $\rEscInit$ is chosen large enough, then 
\begin{equation}
\label{upper_bound_delta_proof_a_priori_bound_invasion_speed}
\sup_{x\in\rr^{\dSpace},\abs{x}\ge\rEscInit + \cnoescAtt\tInit}\abs{u^\dag(x,\tInit)}\le \delta'
\,,
\end{equation}
and it follows from inequalities \cref{upper_bound_RattInfty_proof_a_priori_bound_invasion_speed,upper_bound_delta_proof_a_priori_bound_invasion_speed} that, for all $x$ in $\rr^{\dSpace}$, 
\[
q^\dag(x,\tInit)\le \baretaAtt(x,\tInit)
\,.
\]
Besides, the functions $q^\dag(\cdot,\cdot)$ and $\baretaAtt(\cdot,\cdot)$ are, according to \cref{lem:subsolution,lem:super_solution_att}, a subsolution and a supersolution of system \cref{system_governing_super_solution_q_att}, respectively. As a consequence, it follows from the maximum principle that, for all $(x,t)$ in $\rr^{\dSpace}\times[\tInit,+\infty)$,
\[
q^\dag(x,t)\le \baretaAtt(x,t)
\,.
\]
It follows from the definition of $\baretaAtt$ that, for every positive quantity $c$ larger than $\cnoescAtt$, 
\[
\begin{aligned}
\sup_{x\in\rr^{\dSpace},\, \abs{x}\ge ct} \baretaAtt(x,t) &\to 0
\quad\text{as}\quad
t\to+\infty
\,, \\
\text{thus}\quad
\sup_{x\in\rr^{\dSpace},\, \abs{x}\ge ct} u^\dag(x,t) &\to 0
\quad\text{as}\quad
t\to+\infty
\,,
\end{aligned}
\]
and it follows that $\cInv[u]$ cannot be larger than $c$. This shows that $\cInv[u]$ is actually not larger than $\cnoescAtt$. 
\Cref{lem:explicit_upper_bound_invasion_speed} is proved. 
\end{proof}
\subsection{Firewall function}
\subsubsection{Preliminaries}
Let us keep the notation of \cref{subsec:set_up_stability_at_infinity}. It follows from inequality \cref{def_weight_en} satisfied by $\coeffEn$ that, for all $v$ in $\rr^{\dState}$, 
\begin{equation}
\label{def_weight_en_V_dag}
\coeffEn \, V^\dag(v) +  \frac{v^2}{4}\ge 0
\,,
\end{equation}
and it follows from inequalities \cref{v_nablaV_controls_square_around_loc_min,v_nablaV_controls_pot_around_loc_min} that, for all $v$ in $\rr^{\dState}$ satisfying $\abs{v}\le\dEsc(m)$, 
\begin{align}
\label{v_nablaV_controls_square_around_loc_min_dag}
v\cdot \nabla V^\dag(v) &\ge \frac{\eigVmin(m)}{2} v^2 \,, \\
\text{and}\qquad
\label{v_nablaV_controls_pot_around_loc_min_dag}
v\cdot \nabla V^\dag(v) &\ge V^\dag(v)
\,.
\end{align}
\subsubsection{Definition}
Let $\kappa_0$ denote a positive quantity, small enough so that
\begin{equation}
\label{def_kap_spat_as}
\frac{\coeffEn \, \kappa_0 ^2}{4}\le\frac{1}{2}
\quad\text{and}\quad
\frac{\kappa_0 ^2}{2}\le\frac{\eigVmin(m)}{8}
\end{equation}
(those properties will be used to prove inequality \cref{dt_fire_spat_as} below); this quantity may be chosen as
\[
\kappa_0  = \min\Biggl(\sqrt{\frac{2}{\coeffEn}},\frac{\sqrt{\eigVmin(m)}}{2}\Biggr)
\,.
\]
For every $x$ in $\rr^{\dSpace}$ and every nonnegative time $t$, let 
\begin{equation}
\label{def_E_F_integrands_of_eee_fff}
\begin{aligned}
E^\dag(x,t) &= \frac{1}{2}\abs{\nabla_x u^\dag(x,t)}^2 + V^\dag\bigl( u^\dag(x,t)\bigr) \,, \\
\text{and}\quad
F^\dag(x,t) &= E^\dag(x,t) + \frac{1}{2}u^\dag(x,t)^2
\,.
\end{aligned}
\end{equation}
Let us introduce the weight function $\psi_0:\rr^{\dSpace}\times\rr$ defined as 
\[
\psi_0(x) = \exp(-\kappa_0 \abs{x}) 
\,,
\]
and, for every $\widebar{x}$ in $\rr^{\dSpace}$, its translate $T_{\widebar{x}}\psi_0$ defined as
\[
T_{\widebar{x}}\psi_0(x) = \psi_0(x-\widebar{x})
\,,
\]
and let us introduce the ``firewall'' function defined, for every $\widebar{x}$ in $\rr^{\dSpace}$ and every positive time $t$, as
\begin{equation}
\label{def_firewall_sf}
\fff_0(\widebar{x},t) = \int_{\rr^{\dSpace}}T_{\widebar{x}}\psi_0(x) F^\dag(x,t) \, dx
\,.
\end{equation}
\begin{remark}
The subscript ``$0$'' (in $\kappa_0$ and $\psi_0$ and $\fff_0$) is here to distinguish these objects from others, similar but unequal, introduced in \cref{subsec:relax_sc_stand}. 
\end{remark}
\subsubsection{Coercivity} 
%
\begin{lemma}[coercivity of firewall function]
\label{lem:coerc_fire}
For all $t$ in $(0,+\infty)$ and $\widebar{x}$ in $\rr$, 
\begin{equation}
\label{coerc_fire}
\fff_0(\widebar{x},t) \ge \min\Bigl(\frac{\coeffEn}{2},\frac{1}{4}\Bigr) \int_{\rr^{\dSpace}}T_{\widebar{x}}\psi_0(x)\left( \abs{\nabla_x u^\dag(x,t)}^2 + u^\dag(x,t)^2 \right)\, dx
\,.
\end{equation}
\end{lemma}
\begin{proof}
Inequality \cref{coerc_fire} follows from inequality \cref{def_weight_en_V_dag}.
\end{proof}
\subsubsection{Linear decrease up to pollution.} 
%
For every nonnegative time $t$, let us introduce the set:
\begin{equation}
\label{def_SigmaEsc}
\SigmaEsc(t)=\{x\in\rr^{\dSpace}: \abs{u^\dag(x,t)} >\dEsc(m)\} 
\,.
\end{equation}
\begin{lemma}[firewall linear decrease up to pollution]
\label{lem:approx_decrease_fire}
There exist positive quantities $\nuFzero$ and $\KFzero$ such that, for every $\widebar{x}$ in $\rr^{\dSpace}$ and every positive time $t$,
\begin{equation}
\label{dt_fire}
\partial_t \fff_0(\widebar{x},t)\le-\nuFzero\, \fff_0(\widebar{x},t) + \KFzero\, \int_{\SigmaEsc(t)}T_{\widebar{x}}\psi_0(x)\, dx
\,.
\end{equation}
The quantity $\nuFzero$ depends on $V$ and $m$ (only), whereas $\KFzero$ depends additionally on the upper bound $\RmaxInfty$ on the $L^\infty$-norm of the solution. 
\end{lemma}
\begin{proof}
It follows from expressions \cref{ddt_loc_en_stand_fr,ddt_loc_L2_stand_fr} that, for every $\widebar{x}$ in $\rr^{\dSpace}$ and every positive time $t$,
\[
\begin{aligned}
\partial_t \fff_0(\widebar{x},t) = \int_{\rr^{\dSpace}} \biggl[ & T_{\widebar{x}}\psi_0\Bigl( -\coeffEn (u^\dag_t)^2 - u^\dag\cdot \nabla V^\dag(u^\dag) - \abs{\nabla_x u^\dag}^2 \Bigr) \\
&  - \coeffEn \nabla_x T_{\widebar{x}}\psi_0 \cdot \nabla_x u^\dag \cdot u^\dag_t + \frac{ \Delta_x T_{\widebar{x}} \psi_0}{2}(u^\dag)^2 \biggr] \, dx
\,.
\end{aligned}
\]
Since (for every $x$ in $\rr^{\dSpace}$)
\[
\abs{\nabla_x T_{\widebar{x}}\psi_0(x)} \le \kappa_0 T_{\widebar{x}}\psi_0(x) 
\quad\text{and}\quad
\Delta_x T_{\widebar{x}}\psi_0 (x) \le \kappa_0 ^2 T_{\widebar{x}}\psi_0(x)
\]
(indeed the measure with support at $0_{\rr^{\dSpace}}$ involved in $\Delta_x \psi_0$ is negative), it follows that
\[
\begin{aligned}
&\partial_t \fff_0(\widebar{x},t) \le \\
&\quad\int_{\rr^{\dSpace}} T_{\widebar{x}}\psi_0\Bigl( -\coeffEn (u^\dag_t)^2 - u^\dag\cdot \nabla V^\dag(u^\dag) - \abs{\nabla_x u^\dag}^2 + \coeffEn \kappa_0  \abs{\nabla_x u^\dag} \abs{u^\dag_t} 
+ \frac{\kappa_0 ^2}{2} (u^\dag)^2 \Bigr)\, dx
\,,
\end{aligned}
\]
thus, using the inequality
\[
\kappa_0  \abs{\nabla_x u^\dag} \abs{u^\dag_t} \le (u^\dag_t)^2 + \frac{\kappa_0^2}{4}\abs{\nabla_x u^\dag}^2
\,,
\]
it follows that
\[
\partial_t \fff_0(\widebar{x},t) \le \int_{\rr^{\dSpace}} T_{\widebar{x}}\psi_0\biggl( \Bigl( \frac{\coeffEn \kappa_0 ^2}{4} - 1 \Bigr) \abs{\nabla_x u^\dag}^2 - u^\dag\cdot \nabla V^\dag(u^\dag) + \frac{\kappa_0 ^2}{2} (u^\dag)^2 \biggr) \, dx 
\,,
\]
and according to inequalities \cref{def_kap_spat_as} satisfied by the quantity $\kappa_0 $,
\begin{equation}
\label{dt_fire_spat_as}
\partial_t\fff_0(\widebar{x},t)\le
\int_{\rr^{\dSpace}} T_{\widebar{x}} \psi_0 \left(-\frac{1}{2}\abs{\nabla_x u^\dag}^2 - u\cdot \nabla V^\dag(u^\dag) + \frac{\eigVmin(m)}{8} (u^\dag)^2 \right)\, dx 
\,. 
\end{equation} 
Let $\nuFzero$ denote a positive quantity to be chosen below. It follows from the previous inequality and from the definition \cref{def_firewall_sf} of $\fff(\widebar{x},t)$ that
\begin{equation}
\label{dt_Fire_preliminary}
\begin{aligned}
\partial_t\fff_0(\widebar{x},t) + \nuFzero\fff_0(\widebar{x},t) \le \int_{\rr^{\dSpace}}& T_{\widebar{x}} \psi_0 \biggl[-\frac{1}{2}(1-\nuFzero\,\coeffEn)\abs{\nabla_x u^\dag}^2 - u^\dag\cdot \nabla V^\dag(u^\dag) \\
&  + \nuFzero\,\coeffEn\, V^\dag(u^\dag) +   \Bigl(\frac{\eigVmin(m)}{8} + \frac{\nuFzero}{2}\Bigr)(u^\dag)^2 \biggr]\, dx 
\,.
\end{aligned}
\end{equation}
In view of this expression and of inequalities \vref{v_nablaV_controls_square_around_loc_min_dag,v_nablaV_controls_pot_around_loc_min_dag}, let us assume that $\nuFzero$ is small enough so that
\begin{equation}
\label{def_nu_spat_as}
\nuFzero \, \coeffEn\le 1
\quad\text{and}\quad
\nuFzero \, \coeffEn\le \frac{1}{2}
\quad\text{and}\quad
\frac{\nuFzero}{2} \le \frac{\eigVmin(m)}{8}
\,;
\end{equation}
the quantity $\nuFzero$ may be chosen as
\[
\nuFzero = \min\Bigl(\frac{1}{2\coeffEn}, \frac{\eigVmin(m)}{4}\Bigr)
\,.
\]
Then, it follows from inequalities \cref{dt_Fire_preliminary,def_nu_spat_as} that
\begin{equation}
\label{dt_Fire_preliminary_bis}
\partial_t\fff_0(\widebar{x},t) + \nuFzero\fff_0(\widebar{x},t) \le \int_{\rr^{\dSpace}} T_{\widebar{x}} \psi_0 \left[- u^\dag\cdot \nabla V^\dag(u^\dag) +\frac{1}{2}\abs{V^\dag(u^\dag)} +\frac{\eigVmin(m)}{4} (u^\dag)^2 \right] \, dx 
\,.
\end{equation}
According to inequalities \cref{v_nablaV_controls_square_around_loc_min_dag,v_nablaV_controls_pot_around_loc_min_dag}, the integrand of the integral at the right-hand side of this inequality is nonpositive as long as $x$ is \emph{not} in $\SigmaEsc(t)$. Therefore this inequality still holds if the domain of integration of this integral is changed from $\rr^{\dSpace}$ to $\SigmaEsc(t)$. Besides, observe that, in terms of the ``initial'' potential $V$ and solution $u(x,t)$, the factor of $T_{\widebar{x}}\psi_0$ under the integral of the right-hand side of this last inequality reads
\[
- (u-m)\cdot \nabla V(u)+\frac{1}{2} \abs{V(u)-V(m)}+ \frac{\eigVmin(m)}{4} (u-m)^2
\,.
\]
Thus, if $\KFzero$ denotes the maximum of this expression over all possible values for $u$, that is (according to the $L^\infty$ bound \cref{maximal_radius_excursion_Linfty} on the solution) the quantity
\begin{equation}
\label{def_KFzero}
\KFzero = \max_{v\in\rr^{\dState},\ \abs{v}\le \RmaxInfty}\Bigl[ - (v-m)\cdot \nabla V(v)+ \frac{1}{2} \abs{V(v)-V(m)} + \frac{\eigVmin(m)}{4} (v-m)^2\Bigr]
\,,
\end{equation}
then inequality~\cref{dt_fire} follows from inequality \cref{dt_Fire_preliminary_bis} (with the domain of integration of the integral on the right-hand side restricted to $\SigmaEsc(t)$). This finishes the proof of \cref{lem:approx_decrease_fire}.
\end{proof}
\subsubsection{Exponential decrease and proof of \texorpdfstring{\cref{lem:exponential_decrease_beyond_invasion_speed}}{Lemma \ref{lem:exponential_decrease_beyond_invasion_speed}}}
Let $c_1$ and $c_2$ denote two positive quantity with $c_1$ smaller than $c_2$. The proof of \cref{lem:exponential_decrease_beyond_invasion_speed} will follow from the 
next lemma.
\begin{lemma}[exponential decrease of firewall]
\label{lem:exponential_decrease_firewall}
Assume that there exists a positive quantity $\tInit$ such that, for every $t$ in $[\tInit,+\infty)$, 
\begin{equation}
\label{hyp_Sigma_Esc_in_domain_growing_at_speed_cone}
\SigmaEsc(t)\subset B(c_1 t)
\,,
\end{equation}
and let us introduce the quantities $\nuFzero'$ and $\KFzero'$ defined as
\begin{equation}
\label{def_nu_and_K_exp_decrease_firewall}
\begin{aligned}
\nuFzero' &= \min\left(\nuFzero,\frac{\kappa_0(c_2-c_1)}{2}\right) \,, \quad\text{and} \\
\KFzero' &= \exp(\nuFzero' \tInit)\left(\sup_{x'\in\rr^{\dSpace},\,\abs{x'}\ge c_2 \tInit}\fff_0(x',\tInit)\right) + \frac{2^{\dSpace}\KFzero S_{\dSpace-1} c_1^{\dSpace-1} (\dSpace-1)!}{\kappa_0^{\dSpace+1}(c_2-c_1)^{\dSpace}}
\,.
\end{aligned}
\end{equation}
Then, for every $t$ in $[\tInit,+\infty)$, the following inequality holds:
\begin{equation}
\label{exponential_decrease_firewall}
\sup_{x\in\rr^{\dSpace},\,\abs{x}\ge c_2 t}\fff_0(x,t)\le \KFzero' \exp(-\nuFzero' t)
\,.
\end{equation}
\end{lemma}
\begin{proof}
According to inclusion \cref{hyp_Sigma_Esc_in_domain_growing_at_speed_cone} and to inequality \cref{dt_fire} of \cref{lem:approx_decrease_fire}, for all $x$ in $\rr$, 
\[
\begin{aligned}
\partial_t \fff_0(x,t) + \nuFzero \fff_0(x,t) &\le \KFzero \int_{B(c_1 t)} \exp\bigl(-\kappa_0\abs{x-y}\bigr)\, dy \\
&\le \KFzero e^{-\kappa_0\abs{x}}\int_{B(c_1 t)}e^{\kappa_0\abs{y}}\, dy \\
&= \KFzero e^{-\kappa_0\abs{x}}\int_0^{c_1 t} S_{\dSpace-1} r^{\dSpace-1}e^{\kappa_0 r}\, dr \\
&= \KFzero e^{-\kappa_0\abs{x}} S_{\dSpace-1}(c_1 t)^{\dSpace-1}\int_0^{c_1 t}e^{\kappa_0 r}\, dr \\
&= \KFzero e^{-\kappa_0\abs{x}} S_{\dSpace-1} (c_1 t)^{\dSpace-1}\frac{1}{\kappa_0}\exp(\kappa_0 c_1 t)
\,,
\end{aligned}
\]
so that, if in addition it is assumed that $\abs{x}$ is not smaller than $c_2 t$, then
\begin{equation}
\label{upper_bound_partial_t_Fzero_plus_nu_Fzero_bis}
\partial_t \fff_0(x,t) + \nuFzero \fff_0(x,t) \le \frac{\KFzero S_{\dSpace-1}}{\kappa_0} (c_1 t)^{\dSpace-1}\exp\bigl(-\kappa_0(c_2-c_1)t\bigr)
\,.
\end{equation}
For every $x$ in $\rr^{\dSpace}$ such that $\abs{x}$ is not smaller than $c_2 \tInit$, and for every time $t$ in the interval $[\tInit,x/c_2]$, let us introduce the quantity 
\begin{equation}
\label{def_ggggZero}
\gggg_0(x,t) = \exp(\nuFzero' t) \fff_0(x,t)
\,,
\end{equation}
so that
\[
\partial_t\gggg_0(x,t) = \exp(\nuFzero' t)\bigl( \partial_t\fff_0(x,t) + \nuFzero' \fff_0(x,t)\bigr)
\,.
\]
It follows from inequality \cref{upper_bound_partial_t_Fzero_plus_nu_Fzero_bis} that 
\[
\begin{aligned}
\partial_t\gggg_0(x,t)&\le \\
&\exp(\nuFzero' t)\left((\nuFzero'-\nuFzero )\fff_0(x,t) + \frac{\KFzero S_{\dSpace-1}}{\kappa_0} (c_1 t)^{\dSpace-1}\exp\bigl(-\kappa_0(c_2-c_1)t\bigr)\right)
\,.
\end{aligned}
\]
Since according to the coercivity inequality \cref{coerc_fire} of \cref{lem:coerc_fire} the quantity $\fff_0(x,t)$ is nonnegative, it follows from the definition \cref{def_nu_and_K_exp_decrease_firewall} of $\nuFzero'$ that
\begin{equation}
\label{upper_bound_partial_t_ggggZero}
\partial_t\gggg_0(x,t) \le \frac{\KFzero S_{\dSpace-1}}{\kappa_0}(c_1 t)^{\dSpace-1}\exp\left(-\frac{\kappa_0(c_2-c_1)}{2}t\right)
\,,
\end{equation}
and thus, integrating this inequality between $\tInit$ and $t$, it follows that
\[
\begin{aligned}
\gggg_0(x,t) &\le \gggg_0(x,\tInit) + \frac{\KFzero S_{\dSpace-1}}{\kappa_0} c_1^{\dSpace-1}\int_{\tInit}^{+\infty}(t')^{\dSpace-1}\exp\left(-\frac{\kappa_0(c_2-c_1)}{2}t'\right)\, dt' \\
&\le \exp(\nuFzero' \tInit)\sup_{x'\in\rr^{\dSpace},\,\abs{x'}\ge c_2 \tInit}\fff_0(x',\tInit) \\
&\quad + \frac{\KFzero S_{\dSpace-1}}{\kappa_0} c_1^{\dSpace-1} \left(\frac{2}{\kappa_0(c_2-c_1)}\right)^{\dSpace}\int_0^{+\infty}\tau ^{\dSpace-1} e^{-\tau}\, d\tau
\,.
\end{aligned}
\]
According to \cref{integral_of_r_power_n_exp_minus_r}, the quantity at the right-hand side of this last inequality is equal to the quantity $K$ introduced in \cref{def_nu_and_K_exp_decrease_firewall}. 
Thus it follows from the definition \cref{def_ggggZero} of $\gggg_0(x,t)$ that inequality \cref{exponential_decrease_firewall} holds. \Cref{lem:exponential_decrease_firewall} is proved. 
\end{proof}
\begin{proof}[Proof of \cref{lem:exponential_decrease_beyond_invasion_speed}]
Let us assume that the solution $(x,t)\mapsto u(x,t)$ under consideration is stable close to $m$ at infinity (\cref{def:solution_stable_at_infinity}), and let $c$ denote a positive quantity, larger than than the invasion speed $\cInv[u]$. Let us write
\[
c_2 = c 
\quad\text{and}\quad
c_1 = \frac{1}{2}(c+\cInv[u])
\,,
\quad\text{so that}\quad 
\cInv[u]<c_1<c_2 
\,.
\]
According to \cref{def:invasion_speed} of $\cInv[u]$, there exists a positive time $\tInit$ such that, for every time $t$ greater than or equal to $\tInit$, 
\begin{equation}
\label{sigmaEsc_included_in_B_c_prime_t}
\SigmaEsc(t)\subset B(c_1t)
\,,
\end{equation}
so that assumption \cref{hyp_Sigma_Esc_in_domain_growing_at_speed_cone} of the previous \cref{lem:exponential_decrease_firewall} is fulfilled. The conclusion \cref{exponential_decrease_beyond_invasion_speed} of \cref{lem:exponential_decrease_beyond_invasion_speed} follows from the conclusion \cref{exponential_decrease_firewall} of \cref{lem:exponential_decrease_firewall}, the coercivity \cref{coerc_fire} of $\fff_0(\cdot,\cdot)$, and the bounds \cref{bound_u_ut_ck} on the solution. \Cref{lem:exponential_decrease_beyond_invasion_speed} is proved. 
\end{proof}
\section{Asymptotic energy}
As everywhere else, let us consider a function $V$ in $\ccc^2(\rr^{\dState},\rr)$ satisfying the coercivity hypothesis \cref{hyp_coerc}. 
\subsection{Definition, proof of \texorpdfstring{\cref{prop:asympt_en}}{Proposition \ref{prop:asympt_en}}}
\begin{proof}[Proof of \cref{prop:asympt_en}]
Let $m$ be a point of $\mmm$, and let $(x,t)\mapsto u(x,t)$ be a solution of system \cref{syst_sf} which is stable close to $m$ at infinity. For every positive quantity $c$ and every positive time $t$, with the notation $V^\dag$ and $u^\dag$ introduced in \cref{def_normalized_potential_solution} and the notation $E^\dag(x,t)$ introduced in \cref{def_E_F_integrands_of_eee_fff}, let 
\begin{equation}
\label{def_eee_c_ddd_c_bbb_c}
\begin{aligned}
\eee_c(t) &= \int_{B(c t)} E^\dag(x,t) \, dx \,, \\
\text{and}\quad
\ddd_c(t) &= \int_{B(c t)} u_t^\dag(x,t)^2 \, dx \,, \\
\text{and}\quad
\bbb_c(t) &= \int_{\partial B(c t)}E^\dag(x,t) \, dx 
\,.
\end{aligned}
\end{equation}
It follows from system \cref{syst_sf} that, for every positive time $t$, 
\begin{equation}
\label{time_derivavive_eee_c}
\eee_c'(t) = - \ddd_c(t) + c \bbb_c(t)
\,.
\end{equation}
According to \cref{lem:exponential_decrease_beyond_invasion_speed}, for every positive quantity $c'$ larger than $\cInv[u]$, the quantity
\begin{equation}
\label{sup_of_u_dag_beyond_radius_c_prime_t}
\sup_{x\in\rr^{\dSpace},\,\abs{x}\ge c't}\abs{u^\dag(x,t)}
\end{equation}
goes to $0$ at an exponential rate as $t$ goes to $+\infty$. And according to the bounds \cref{bound_u_ut_ck}, the same is true for the quantity
\begin{equation}
\label{sup_of_nabla_u_dag_beyond_radius_c_prime_t}
\sup_{x\in\rr^{\dSpace},\,\abs{x}\ge c' t}\abs{\nabla_x u^\dag(x,t)}
\,.
\end{equation}
It follows that, if $c$ is positive larger than $\cInv[u]$, then $\bbb_c(t)$ goes to $0$ as $t$ goes to $+\infty$. As a consequence, the quantity 
\[
\eeeAsympt[u] = \liminf_{t\to+\infty}\eee_c(t)
\]
is in $\{-\infty\}\cup\rr$ and $\eee_c(t)$ goes to $\eeeAsympt[u]$ as $t$ goes to $+\infty$. The convergences at exponential rates of the quantities \cref{sup_of_u_dag_beyond_radius_c_prime_t,sup_of_nabla_u_dag_beyond_radius_c_prime_t} also show that, for every couple $(c,c')$ of positive quantities both larger than $\cInv[u]$, the difference
\[
\eee_c(t)-\eee_{c'}(t)
\] 
goes to $0$ as $t$ goes to $+\infty$. In other words, the limit $\eeeAsympt[u]$ does not depend on the choice of $c$ in the interval $\bigl(\cInv[u],+\infty\bigr)$. \Cref{prop:asympt_en} is proved. 
\end{proof}
\subsection{Upper semi-continuity, proof of \texorpdfstring{\cref{prop:scs_asympt_en}}{Proposition \ref{prop:scs_asympt_en}}}
\label{subsec:upper_semi_continuity_of_asymptotic_energy}
\begin{proof}[Proof of \cref{prop:scs_asympt_en}]
Let $m$ be a point of $\mmm$, let $(\uzeron)_{n\in\nn}$ denote a sequence of functions in the set $\XstabInfty(m)$ of initial conditions that are stable close to $m$ at infinity (\cref{subsubsec:solutions_stable_at_infinity}), and let $\uzeroinfty$ denote a function in $\XstabInfty(m)$ such that 
\[
\norm{\uzeron-\uzeroinfty}_X\to0
\quad\text{as}\quad
n\to+\infty
\,.
\]
The goal is to prove that
\begin{equation}
\label{eeeAsympt_of_u_zero_infty_larger_than_limsup_of_eeeAsympt_of_u_zero_n}
\eeeAsympt[\uzeroinfty]\ge\limsup_{n\to+\infty}\eeeAsympt[\uzeron]
\,.
\end{equation}
For every $n$ in $\nn\cup\{\infty\}$, let $u_n(\cdot,\cdot)$ denote the solution of system \cref{syst_sf} with initial condition $\uzeron$, and let us introduce the ``normalized'' solution $u_n^\dag(\cdot,\cdot)$ and the function $q_n^\dag(\cdot,\cdot)$ defined as
\[
u_n^\dag(x,t) = u(x,t) - m
\quad\text{and}\quad
q_n^\dag(x,t) = \frac{1}{2}u_n^\dag(x,t)^2 
\,,
\]
and let us define the potentials $V^\dag$ and $V^\ddag$ as in \cref{def_normalized_potential_solution,subsubsec:_potential_quadratic_at_infinity}. 
Let us introduce three positive quantities $\delta$ and $\delta'$ and $\delta''$ with the same properties as in \cref{subsubsec:local_coercivity}. According to \cref{prop:attr_ball}, there exists a positive quantity $\RmaxInfty$ such that, for every $t$ in $[0,+\infty)$, 
\[
\sup_{x\in\rr^{\dSpace}} \abs{u_\infty^\dag(x,t)}\le \RmaxInfty - 1 
\,;
\]
and since $\uzeroinfty$ was assumed to be stable close to $m$ at infinity, there exist positive quantities $\tInit$ and $\rEscInit$ such that
\[
\sup_{x\in\rr^{\dSpace}\, \abs{x}\ge\rEscInit} \abs{u_\infty^\dag(x,\tInit)} \le \delta
\,.
\]
By continuity of the semi-flow $(S_t)_{t\ge0}$ of system \cref{syst_sf} with respect to initial conditions in $X$, there exists $n_0$ in $\nn$ such that, for every integer $n$ greater than or equal to $n_0$, 
\begin{align}
\label{u_n_smaller_than_RmaxInfty_everywhere}
\sup_{x\in\rr^{\dSpace}} \abs{u_n^\dag(x,\tInit)}&\le \RmaxInfty \,, \\
\label{u_n_smaller_than_delta_prime_at_infty}
\text{and}\quad
\sup_{x\in\rr^{\dSpace}\, \abs{x}\ge\rEscInit} \abs{u_n^\dag(x,\tInit)} &\le \delta'
\,.
\end{align}
Let us define the quantities $\lambda$, $\Rcoerc$, $\gamma'$, $\gamma''$, $\qMax$, $\cnoesc$, and the functions $\widebar{N}$ and $\eta_0(\cdot)$ as in \cref{subsec:application_max_principle}, together with the functions $\eta(\cdot,\cdot)$ and $\widebar{\eta}(\cdot,\cdot)$ with the parameter $c$ chosen to be equal to $\cnoesc$. Finally, let us introduce the function
\[
\tilde{\eta}:\rr^{\dSpace}\times[\tInit,+\infty)\to\rr\,,
\quad 
(x,t)\mapsto \widebar{\eta}(x,t-\tInit)
\,.
\]
According to \cref{lem:subsolution}, for every $n$ in $\nn\cup\{\infty\}$, the function $q_n^\dag(\cdot,\cdot)$ is a subsolution of system \cref{system_governing_super_solution_q}; and according to \cref{lem:super_solution}, the function $\tilde{\eta}$ is a supersolution of the same system. And it follows from inequalities \cref{u_n_smaller_than_RmaxInfty_everywhere,u_n_smaller_than_delta_prime_at_infty} that, for every $n$ in $\{n_0,n_0+1,\dots\}\cup\{\infty\}$ and for every $x$ in $\rr^{\dSpace}$, 
\[
q_n^\dag(x,\tInit) \le \tilde{\eta}(x,\tInit)
\,.
\]
It follows that, for every time $t$ greater than or equal to $\tInit$, the same inequality still holds at time $t$; that is, for every $n$ in $\{n_0,n_0+1,\dots\}\cup\{\infty\}$ and for every $x$ in $\rr^{\dSpace}$, 
\begin{equation}
\label{q_n_dag_not_larger_than_tilde_eta_for_n_large_enough}
q_n^\dag(x,t) \le \tilde{\eta}(x,t)
\,.
\end{equation}
Let us introduce the quantity 
\[
c' = \cnoesc + 1 
\,,
\]
and, for every $n$ in $\nn\cup\{+\infty\}$ and $t$ in $[0,+\infty)$, let us introduce the quantities
\[
\eee_{c',n}(t)
\quad\text{and}\quad
\bbb_{c',n}(t)
\]
defined exactly as in \cref{def_eee_c_ddd_c_bbb_c} for the solution $u_n$ and the parameter $c'$. It follows from \cref{time_derivavive_eee_c} that
\begin{equation}
\label{upper_bound_eeeAsympt_minus_eee_c_prime_n_of_t}
\eeeAsympt[\uzeron] - \eee_{c',n}(t) \le c' \int_t^{+\infty} \bbb_{c',n}(s)\, ds
\,.
\end{equation}
Besides, it follows from inequality \cref{q_n_dag_not_larger_than_tilde_eta_for_n_large_enough} and from the definition of the function $\tilde{\eta}(\cdot,\cdot)$ that, for every $n$ in $\{n_0,n_0+1,\dots\}\cup\{\infty\}$, the quantity 
\[
\sup_{x\in\rr^{\dSpace}\, \abs{x}\ge c' t} \abs{u_n^\dag(x,t)}
\]
goes to as $t$ goes to $+\infty$, at an exponential rate and uniformly with respect to $n$ in $\{n_0,n_0+1,\dots\}\cup\{\infty\}$. And thus, according to the bounds \cref{bound_u_ut_ck}, the same is true for the quantity
\[
\sup_{x\in\rr^{\dSpace}\, \abs{x}\ge c' t} \abs{\nabla_x u_n^\dag(x,t)}
\,,
\]
and thus the same is true for the quantity $\bbb_{c',n}(t)$. As a consequence, there exist positive quantities $\nu$ and $K$ such that, for every $n$ in $\{n_0,n_0+1,\dots\}\cup\{\infty\}$ and for every $t$ greater than or equal to $\tInit$, 
\[
\int_t^{+\infty} \bbb_{c',n}(s)\, ds \le K\exp(-\nu t)
\,.
\]
Thus it follows from inequality \cref{upper_bound_eeeAsympt_minus_eee_c_prime_n_of_t} that, for every $n$ in $\{n_0,n_0+1,\dots\}\cup\{\infty\}$ and for every $t$ greater than or equal to $\tInit$, 
\[
\eee_{c',n}(t) \ge \eeeAsympt[\uzeron] - K\exp(-\nu t)
\,.
\]
Passing to the limit as $n$ goes to $+\infty$, it follows from the continuity of the semi-flow $(S_t)_{t\ge0}$ of system \cref{syst_sf} with respect to initial conditions in $X$ that, for every $t$ greater than or equal to $\tInit$, 
\[
\eee_{c'}^{(\infty)}(t) \ge \limsup_{n\to+\infty} \eeeAsympt[\uzeron] - K\exp(-\nu t)
\,.
\]
Finally, passing to the limit as time goes to $+\infty$, inequality \cref{eeeAsympt_of_u_zero_infty_larger_than_limsup_of_eeeAsympt_of_u_zero_n} follows. \Cref{prop:scs_asympt_en} is proved. 
\end{proof}
\section{No invasion implies relaxation}
\label{sec:inv_no_inv_dichotomy}
\subsection{Definitions and hypotheses}
\label{subsec:def_hyp}
As everywhere else, let us consider a function $V$ in $\ccc^2(\rr^{\dState},\rr)$ satisfying the coercivity hypothesis \cref{hyp_coerc}. Let us consider a point $m$ in $\mmm$ and a solution $(x,t)\mapsto u(x,t)$ of system \cref{syst_sf}.
In this section it will not be assumed that the solution stable at infinity, but that it satisfies, instead, the following (less restrictive) hypothesis \hypHom. The more general statement issued from this setting is appropriate to be applied (in the radially symmetric case \cite{Risler_globalBehaviourRadiallySymmetric_2017}) to the behaviour (relaxation) of a solution behind a chain of fronts travelling to infinity (in space). 
\begin{description}
\item[\hypHom] There exist a positive quantity $\cHom$ and a $\ccc^1$-function
\[
\rHom:[0,+\infty)\to[0,+\infty)
\quad\text{satisfying}\quad
\rHom'(t)\to \cHom 
\quad\text{as}\quad
t\to +\infty 
\]
such that, for every positive quantity $L$, 
\[
\sup_{x\in\rr^{\dSpace}, \ \rHom(t)-L\le \abs{x} \le \rHom(t)+L} \abs{u(x,t)-m} \to 0
\quad\text{as}\quad
t\to +\infty
\,.
\]
\end{description}
If the solution $(x,t)\mapsto u(x,t)$ is stable close to $m$ at infinity then hypothesis \hypHom holds (for every positive quantity $\cHom$ larger than the invasion speed of the solution and a function $\rHom$ defined as $t\mapsto\cHom t$), but the converse is not true. 
\begin{figure}[!htbp]
\centering
\includegraphics[width=\textwidth]{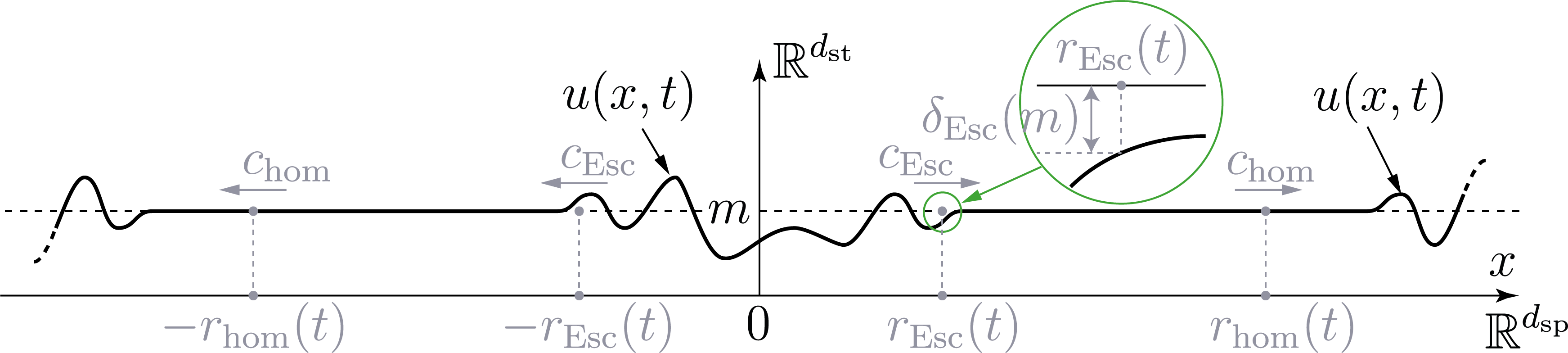}
\caption{Illustration of hypothesis \hypHom and of the notation $\rHom(t)$ and $\cHom$ and $\rEsc(t)$ (for simplicity the figure is made with a space dimension $\dSpace$ equal to $1$, although in the paper $\dSpace$ is assumed to be not smaller than $2$).}
\label{fig:notation_hyp_hom}
\end{figure}
For every nonnegative time $t$, let us denote by $\rEsc(t)$ the supremum of the set
\[
\left\{ r\in\bigl[0,\rHom(t)\bigr) : \sup_{x\in\rr^{\dSpace},\, \abs{x}=r} \abs{u(x,t)-m}\ge \dEsc(m) \right\}
\,,
\]
with the convention that $\rEsc(t)$ equals $0$ if this set is empty (see \cref{fig:notation_hyp_hom}). Let us assume that the following ``no invasion'' assumption holds.
\begin{description}
\item[ \hypNoInv] $\displaystyle \frac{\rEsc(t)}{t}\to0$ as $t\to+\infty$. 
\end{description}
\subsection{Statement}
Recall (see \cref{subsubsec:asympt_en}) that, for every positive quantity $r$, $B(r)$ denotes the (open) ball of radius $r$ and centred at the origin in $\rr^{\dSpace}$. The goal of \cref{sec:inv_no_inv_dichotomy} is to prove the following theorem, which is a  variant of conclusion \cref{item:thm_main_no_invasion} of \cref{thm:main} in the slightly more general setting considered here. 
\begin{theorem}[no invasion implies relaxation]
\label{thm:no_inv_implies_relaxation}
Let $V$ denote a function in $\ccc^2(\rr^{\dState},\rr)$ satisfying the coercivity hypothesis \cref{hyp_coerc}. Then, for every point $m$ of $\mmm$ and every solution $(x,t)\mapsto u(x,t)$ of system \cref{syst_sf} satisfying hypotheses \hypHom and \hypNoInv, the following conclusions hold.
\begin{enumerate}
\item There exists a nonnegative quantity $\eeeResAsympt[u]$ (``residual asymptotic energy'') such that, for every quantity $c$ in $(0,\cHom)$, 
\[
\int_{B(ct)} \biggl( \frac{1}{2}\abs{\nabla_x u(x,t)}^2 + V\bigl( u(x,t)\bigr) -V(m)\biggr) \, dx \to\eeeResAsympt[u]
\quad\text{as}\quad
t\to+\infty
\,.
\]
\label{item:thm_no_inv_implies_relaxation_nonnegative_res_asympt_en}
\item The quantity
\[
\sup_{x\in B\bigl(\rHom(t)\bigr)}\abs{u_t(x,t)}
\]
goes to $0$ as time goes to $+\infty$.
\label{item:thm_no_inv_implies_relaxation_time_derivative_goes_to_zero}
\item For every quantity $c$ in $(0,\cHom)$, the function 
\[
t\mapsto \int_{B(ct)} u_t(x,t)^2 \, dx 
\]
is integrable on a neighbourhood of $+\infty$. 
\label{item:thm_no_inv_implies_relaxation_integrability_of_dissipation}
\end{enumerate} 
\end{theorem}
\begin{remark}
Conclusion \cref{item:thm_no_inv_implies_relaxation_integrability_of_dissipation} of this theorem is somehow redundant with conclusion \cref{item:thm_no_inv_implies_relaxation_nonnegative_res_asympt_en}, therefore it could have been omitted in this statement. The reason why this additional conclusion \cref{item:thm_no_inv_implies_relaxation_integrability_of_dissipation} is added to the statement above is that this turns out to be convenient in the proof of the main result of the companion paper \cite{Risler_globalBehaviourRadiallySymmetric_2017} (see \GlobalBehavRadPropNoInvasionImpliesRelaxation{} in this reference).
\end{remark}
\subsection{Relaxation scheme in a standing or almost standing frame}
\label{subsec:relax_sc_stand}
\subsubsection{Set-up}
Let us keep the notation and assumptions of \cref{subsec:def_hyp}, and let us assume that the hypotheses \cref{hyp_coerc} and \hypHom of \cref{thm:no_inv_implies_relaxation} hold. According to \cref{prop:attr_ball,cor:att_ball_Honeul}, it may be assumed, up to changing the origin of time, that, for all $t$ in $[0,+\infty)$, 
\begin{align}
\label{hyp_attr_ball_infty}
\norm{x\mapsto u(x,t)}_{\Linfty} &\le \RattInfty \\
\text{and}\qquad
\label{hyp_attr_ball_Honeul}
\norm{x\mapsto u(x,t)}_{\Honeul} &\le \RattHoneul
\,.
\end{align}
Let us introduce the ``normalized potential'' $V^\dag$ and the ``normalized solution'' $u^\dag$ defined in \cref{def_normalized_potential_solution}. 
\subsubsection{Notation for the travelling frame}
\label{subsubsec:notation_trav_frame}
Let $c$ denote a vector of $\rr^{\dSpace}$. It may be written as
\[
c = \abs{c}\vec{n}
\,,
\]
where $\vec{n}$ is in $\rr^{\dSpace}$ and its euclidean norm is equal to $1$. Actually the choice of $\vec{n}$ will be of no importance, thus this vector may very well be chosen as $(1,0,\dots,0)$.

Let us consider the solution viewed in a frame travelling at velocity $c$, that is the function $(\xi,t)\mapsto v(\xi,t)$ defined for every $\xi$ in $\rr^{\dSpace}$ and nonnegative time $t$ by
\[
v(\xi,t) = u^\dag(x,t)
\quad\text{for}\quad
x=ct+\xi
\,.
\]
This function is a solution of the differential system
\[
v_t - c \cdot \nabla_\xi v = - \nabla V^\dag(v) + \Delta_\xi v
\,.
\]
In the forthcoming calculations, the notation $r$ is used to refer to $\abs{x}$ and $\rho$ to refer to $\abs{\xi}$. 
\subsubsection{Choice of the parameters to define energy and firewalls}
%
Let $\kappa$ (rate of decrease of the weight functions) and $\cCut$ (speed of cut-off radius) be positive quantities, small enough so that
\begin{equation}
\label{conditions_kappa_cCut_coeffEn}
\coeffEn \kappa\Bigl( \frac{\cCut}{2} + \frac{\kappa}{4}\Bigr) \le \frac{1}{2}
\quad\text{and}\quad
\coeffEn \, \kappa\,  \cCut \le  \frac{1}{4}
\quad\text{and}\quad
\frac{\kappa(\dSpace \kappa +\cCut)}{2} \le \frac{\eigVmin(m)}{8}
\,.
\end{equation}
Conditions \cref{conditions_kappa_cCut_coeffEn} will be used to prove inequality \vref{dt_fire_before_nufire}. 
These quantities may, for instance, be chosen as
\begin{equation}
\label{def_kappa_cCut}
\kappa = \min\Bigl(\frac{1}{\sqrt{\coeffEn}}, \frac{\sqrt{\eigVmin(m)}}{4\sqrt{\dSpace}}  \Bigr) 
\quad\text{and}\quad
\cCut = \min\Bigl(\frac{1}{4\sqrt{\coeffEn}}, \frac{\sqrt{\eigVmin(m)}}{4}, \frac{\cHom}{2} \Bigr) 
\,. 
\end{equation}
Let us assume that
\begin{equation}
\label{conditions_c}
\abs{c}\le\frac{\sqrt{\eigVmin(m)}}{4}
\quad\text{and}\quad
\abs{c}\le \frac{\kappa}{20}
\quad\text{and}\quad
\abs{c}\le\frac{\cCut}{6}
\,.
\end{equation}
According to hypotheses \hypHom and \hypNoInv and to the value of the quantity $\cCut$ chosen above, there exists a nonnegative time $T$ such that, for every time $t$ greater than or equal to $T$,
\begin{equation}
\label{hyp_rHom_and_rEsc_up_to_origin_of_time}
\rEsc(t)\le \frac{1}{6}\cCut t
\quad\text{and}\quad
\rHom(t)\ge \frac{11}{6} \cCut t
\,.
\end{equation}
\subsubsection{Localized energy}
\label{subsubsec:def_loc_en}
Let us introduce the function $(\rho,t)\mapsto \chiScalar(\rho,t)$ defined on $\rr\times [0,+\infty)$ as
\[
\chiScalar(\rho,t) = 
\left\{
\begin{aligned}
1 & \quad\text{if} \quad \abs{\rho} \le \cCut t \\
\exp\bigl( - \kappa  ( \abs{\rho}-\cCut t ) \bigr) & \quad \text{if} \quad \abs{\rho}\ge \cCut t 
\,,
\end{aligned}
\right. 
\]
\begin{figure}[!htbp]
	\centering
	\includegraphics[width=\textwidth]{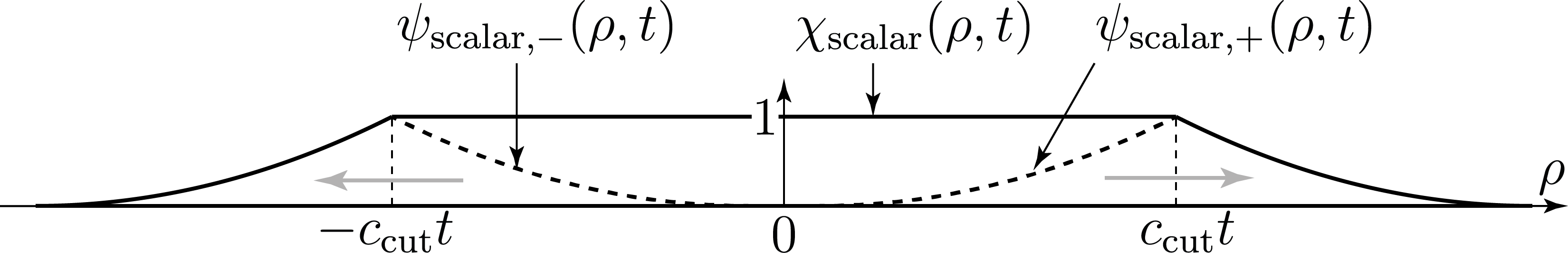}
	\caption{Graphs of functions $\rho\mapsto\chiScalar(\rho,t)$ and $\rho\mapsto \psiScalarPlus(\rho,t)$ and $\rho\mapsto \psiScalarMinus(\rho,t)$.}
	\label{fig:graph_weight_one_energy_firewall}
\end{figure}
see \cref{fig:graph_weight_one_energy_firewall}; and let us introduce the functions $(\xi,t)\mapsto \chiStand(\xi,t)$ and $(\xi,t)\mapsto \chi(\xi,t)$ defined on $\rr^{\dSpace} \times [0,+\infty)$ as
\begin{equation}
\label{def_chi_zero_chi}
\chiStand(\xi,t) = \chiScalar\bigl(\abs{\xi},t\bigr)
\quad\text{and}\quad
\chi(\xi,t) = \exp(c\cdot \xi) \chiStand(\xi,t)
\,.
\end{equation}
In this notation the index ``scalar'' refers to the scalar argument of the function, and the index ``stand'' refers to the fact that $\chiStand$ equals $\chi$ when the velocity $c$ vanishes (``standing'' frame).
For every nonnegative time $t$, let us introduce the following ``localized energy'':
\begin{equation}
\label{def_E}
\eee(t) = \int_{\rr^{\dSpace}} \chi(\xi,t)\biggl( \frac{1}{2}\abs{\nabla_\xi v(\xi,t)}^2 + V^\dag\bigl( v(\xi,t)\bigr) \biggr) \, d\xi
\,. 
\end{equation}
\subsubsection{Time derivative of localized energy}
\label{subsubsec:der_loc_en}
For every nonnegative time $t$, let us introduce the ``dissipation'' defined as
\begin{equation}
\label{def_dissip_tr_fr}
\ddd(t) = \int_{\rr^{\dSpace}} \chi(\xi,t)\, v_t(\xi,t)^2 \, d\xi
\,.
\end{equation}
\begin{lemma}[time derivative of localized energy]
\label{lem:time_der_energy_without_fire}
For every nonnegative time $t$, 
\begin{equation}
\label{ds_en_trav_f_prel_relax}
\eee'(t) \le -\frac{1}{2} \ddd(t) + \kappa\int_{\rr^{\dSpace}\setminus B(\cCut t)} \chi \biggl( \frac{\cCut+\kappa}{2}\abs{\nabla_\xi v}^2  + \cCut V^\dag(v)\biggr)\, d\xi 
\,.
\end{equation}
\end{lemma}
\begin{proof}
It follows from expression \vref{ddt_loc_en_trav_fr} (time derivative of localized energy) that, for every nonnegative time $t$, 
\begin{equation}
\label{ds_en_trav_f_prel_relax_proof}
\eee'(t) = - \ddd(t) + \int_{\rr^{\dSpace}} \Biggl[ \chi_t \biggl( \frac{1}{2}\abs{\nabla_\xi v}^2 + V^\dag(v) \biggr) + (\chi c - \nabla_\xi \chi) \cdot \nabla_\xi v \cdot v_t \Biggr]\, d\xi 
\,.
\end{equation}
It follows from the definitions of $\chiStand$ and $\chi$ that:
\[
\chi_t(\xi,t) = 
\left\{
\begin{aligned}
0 & \quad\text{if} \quad \abs{\xi}\le \cCut t \\
\kappa \cCut \chi(\xi,t) & \quad \text{if} \quad \abs{\xi}\ge \cCut t 
\,,
\end{aligned}
\right.
\]
and
\[
\nabla_\xi\chi = \chi c + e^{c\cdot\xi}\nabla_\xi\chiStand
\quad\text{thus}\quad
\chi c - \nabla_\xi\chi = - e^{c\cdot\xi}\nabla_\xi\chiStand
\,,
\]
and 
\[
\nabla_\xi \chiStand(\xi,t) = 
\left\{
\begin{aligned}
0 & \quad\text{if} \quad \abs{\xi}\le \cCut t  \\
-\kappa \frac{\xi}{\abs{\xi}} \chiStand(\xi,t) & \quad \text{if} \quad \abs{\xi}\ge \cCut t 
\,,
\end{aligned}
\right.
\]
thus
\[
(\chi c - \nabla_\xi \chi)(\xi,t) = 
\left\{
\begin{aligned}
0 & \quad\text{if} \quad \abs{\xi}\le \cCut t \\
\kappa \frac{\xi}{\abs{\xi}} \chi(\xi,t) & \quad \text{if} \quad \abs{\xi}\ge \cCut t \,.
\end{aligned}
\right.
\]
Thus it follows from \cref{ds_en_trav_f_prel_relax_proof} that
\[
\begin{aligned}
\eee'(t) &= - \ddd(t) + \kappa\int_{\rr^{\dSpace}\setminus B(\cCut t)} \chi \Biggl[  \cCut \biggl( \frac{1}{2}\abs{\nabla_\xi v}^2 + V^\dag(v) \biggr) +  \frac{\xi}{\abs{\xi}}\cdot \nabla_\xi v \cdot v_t\Biggr]\, d\xi \\
&\le - \ddd(t) + \kappa\int_{\rr^{\dSpace}\setminus B(\cCut t)} \chi \Biggl[  \cCut \biggl( \frac{1}{2}\abs{\nabla_\xi v}^2 + V^\dag(v) \biggr) +  \abs{\nabla_\xi v \cdot v_t} \Biggr]\, d\xi
\,.
\end{aligned}
\]
Using the polar inequality
\[
\kappa  \abs{\nabla_\xi v \cdot v_t} \le \frac{1}{2} v_t^2 + \frac{\kappa ^2}{2}\abs{\nabla_\xi v}^2
\,,
\]
inequality \cref{ds_en_trav_f_prel_relax} follows. \Cref{lem:time_der_energy_without_fire} is proved. 
\end{proof}
\subsubsection{Firewall function}
\label{subsubsec:def_fire_relax}
A firewall function will now be introduced to control the last term of the right-hand side of inequality \cref{ds_en_trav_f_prel_relax} above. 
First let us introduce the functions $\psiScalarMinus$ and $\psiScalarPlus$ and $\psiScalar$ defined on $\rr\times[0,+\infty)$ as
\[
\begin{aligned}
\psiScalarMinus(r,t) &= \exp\bigl(-\kappa \abs{-\cCut t -r}\bigr) \,, \\
\text{and}\quad
\psiScalarPlus(r,t) &= \exp\bigl(-\kappa \abs{\cCut t -r}\bigr) \,, \\
\text{and}\quad
\psiScalar(r,t) &= \psiScalarPlus(r,t)  + \psiScalarMinus(r,t)\,,
\end{aligned}
\]
see \cref{fig:graph_weight_one_energy_firewall}, and the functions $(\xi,t)\mapsto \psiStand(\xi,t)$ and $(\xi,t)\mapsto \psi(\xi,t)$ defined on $\rr^{\dSpace}\times[0,+\infty)$ by
\begin{equation}
\label{def_psi_zero_and_psi}
\psiStand(\xi,t) = \psiScalar\bigl(\abs{\xi},t\bigr)
\quad\text{and}\quad \psi(\xi,t) = e^{c\cdot \xi} \, \psiStand(\xi,t)
\,.
\end{equation}
The meaning of the indexes ``scalar'' and ``stand'' is the same as for the notation involving the symbol ``$\chi$'' in \cref{subsubsec:def_loc_en}.
Observe that, for all $\xi$ in $\rr^{\dSpace}\setminus B(\cCut t)$,
\begin{equation}
\label{chi_smaller_than_psi_outside_ball_of_radius_cCut_t}
\chi(\xi,t) \le \psi (\xi,t)
\end{equation}
(although the difference is tiny). The definition of $\psiScalar$ as the sum of the two functions $\psiScalarMinus$ and $\psiScalarPlus$ is more convenient than if $\psiScalar(r,t)$ was defined as the supremum of these two functions (it avoids a singularity at $r=0$, which would induce a inconvenient Dirac mass of positive weight at $0_{\rr^{\dSpace}}$ for the Laplacian of $\psi$, see for instance expression \vref{ddt_loc_L2_trav_fr}). 

Let us introduce the ``firewall'' function defined, for every nonnegative time $t$, as
\begin{equation}
\label{def_fff}
\fff(t) = \int_{\rr^{\dSpace}} \psi(\xi,t)\Biggl[ \coeffEn \biggl( \frac{1}{2}\abs{\nabla_\xi v(\xi,t)}^2 + V^\dag\bigl( v(\xi,t)\bigr) \biggr) + \frac{1}{2}v(\xi,t)^2\Biggr]\, d\xi 
\,.
\end{equation}
\subsubsection{Energy decrease up to firewall}
\label{subsubsec:loc_energy_decrease}
\begin{lemma}[decrease up to firewall for localized energy]
\label{lem:ds_en_trav_f_inprogress_relax}
There exists a nonnegative quantity $\KEF$, depending only on $V$ and $m$ and $\dSpace$, such that, for every nonnegative time $t$, 
\begin{equation}
\label{ds_en_trav_f_inprogress_relax}
\eee'(t)\le -\frac{1}{2} \ddd(t) + \KEF\fff(t)
\,.
\end{equation}
\end{lemma}
\begin{proof}
Since $\chi$ takes only nonnegative values, it follows from \cref{ds_en_trav_f_prel_relax} that, for every nonnegative time $t$,
\begin{equation}
\label{dt_en_trav_f_firewall_1}
\eee'(t) \le -\frac{1}{2} \ddd(t) + \kappa\int_{\rr^{\dSpace}\setminus B(\cCut t)} \chi \biggl( 
\frac{\cCut+\kappa}{2}\abs{\nabla_\xi v}^2  + \cCut \Bigl(  V^\dag(v)+ \frac{1}{2\coeffEn}v^2\Bigr) 
\biggr)\, d\xi 
\,.
\end{equation}
It follows from inequality \vref{def_weight_en_V_dag} derived from the definition of $\coeffEn$ that the quantity 
\[
V^\dag(v)+ \frac{1}{2\coeffEn}v^2 = \frac{1}{\coeffEn}\Bigl(\coeffEn V^\dag(v)+ \frac{1}{2}v^2\Bigr)
\]
is nonnegative; as a consequence, inequality \cref{dt_en_trav_f_firewall_1} remains true if the following three changes to the integral of the right-hand side are made:
\begin{itemize}
\item in the integrand the factor $\cCut$ is replaced by the larger factor $\cCut +\kappa$,
\item $\chi$ is replaced with $\psi$, 
\item the integration domain is extended to $\rr^{\dSpace}$ (after $\chi$ is replaced with $\psi$). 
\end{itemize}
After these three changes inequality \cref{dt_en_trav_f_firewall_1} reads:
\begin{equation}
\label{ds_en_trav_f_inprogress_relax_proof}
\eee'(t) \le -\frac{1}{2} \ddd(t) + \frac{\kappa(\cCut+\kappa)}{\coeffEn} \fff(t)
\,.
\end{equation}
Thus, introducing the following quantity (depending only on $V$ and $m$ and $\dSpace$):
\[
\KEF = \frac{\kappa (\cCut + \kappa )}{\coeffEn}
\,,
\]
inequality \cref{ds_en_trav_f_inprogress_relax} in \cref{lem:ds_en_trav_f_inprogress_relax} follows from inequality \cref{ds_en_trav_f_inprogress_relax_proof}. \Cref{lem:ds_en_trav_f_inprogress_relax} is proved. 
\end{proof}
Remark: according to the definitions \cref{def_kappa_cCut} of $\kappa$ and $\cCut$, \cref{lem:ds_en_trav_f_inprogress_relax} and inequality \cref{ds_en_trav_f_inprogress_relax} still hold if the quantity $\KEF$ is chosen equal to $2/\coeffEn^2$, which depends only on $V$; thus the dependence of $\KEF$ on $m$ and $\dSpace$ could be removed from the statement of \cref{lem:ds_en_trav_f_inprogress_relax} (this change would have no implication on the following). 
\subsubsection{Firewall linear decrease up to pollution}
\label{subsubsec:der_fire_dichot}
For every nonnegative time $t$, let 
\[
\begin{aligned}
\SigmaEscStand(t) &= \{ x\in\rr^{\dSpace} : \abs{u(x,t)} > \dEsc(m) \} \\
\text{and}\qquad
\SigmaEscTrav(t) &= \{ \xi\in\rr^{\dSpace} : \abs{v(\xi,t)} > \dEsc(m) \}
\,.
\end{aligned}
\]
According to the notation of \cref{subsubsec:notation_trav_frame}, for all $x$ and $\xi$ in $\rr^{\dSpace}$ satisfying $x=ct+\xi$, 
\[
x\in \SigmaEscStand(t) \iff \xi \in \SigmaEscTrav(t)
\,.
\]
\begin{lemma}[firewall decrease up to pollution]
\label{lem:der_fire_dichot}
There exist positive quantities $\nuF$ and $\KF$, depending only on $V$ and $m$, such that, for every nonnegative time $t$, 
\begin{equation}
\label{der_fire_dichot}
\fff'(t) \le -\nuF \fff(t) + \KF \int_{\SigmaEscTrav(t)} \psi (\xi,t) \, d\xi
\,.
\end{equation}
\end{lemma}
\begin{proof}
\renewcommand{\qedsymbol}{}
According to expressions \vref{ddt_loc_en_trav_fr,ddt_loc_L2_trav_fr} for the time derivatives of localized energy and $L^2$ functional, for every nonnegative time $t$,
\begin{equation}
\label{dt_fff_prel_proof_1}
\begin{aligned}
\fff'(t) = & \int_{\rr^{\dSpace}} \Biggl[ \psi \Bigl( - \coeffEn v_t^2 - v\cdot \nabla V^\dag(v) - \abs{\nabla_\xi v}^2 \Bigr) 
+ \coeffEn \psi_t   \Bigl( \frac{1}{2}\abs{\nabla_\xi v}^2 + V^\dag(v) \Bigr)   \\
& + \coeffEn (\psi c  - \nabla_\xi \psi) \cdot \nabla_\xi v \cdot v_t 
+ \frac{\psi_t + \Delta_\xi \psi-c\cdot \nabla_\xi\psi}{2}v^2 \Biggr] \, d\xi
\,.
\end{aligned}
\end{equation}
The proof of \cref{lem:der_fire_dichot} resumes after the statement and the proof of the following intermediary result.
\end{proof}
\begin{lemma}[bounds on combinations of derivatives of the weight function $\psi$]
\label{lem:bounds_combinations_derivatives_psi}
For every nonnegative time $t$ and for all $\xi$ in $\rr^{\dSpace}$, 
\begin{align}
\label{bound_partial_t_psi}
\abs{\psi_t(\xi,t)} &\le \kappa \cCut \psi(\xi,t) \,, \\
\text{and}\qquad
\label{bound_psi_c_minus_nabla_xi_psi}
\abs{\psi c - \nabla_\xi \psi} &\le \kappa\psi \,, \\
\text{and}\qquad
\label{bound_Delta_psi_minus_c_nabla_xi_psi}
\Delta_\xi \psi-c\cdot \nabla_\xi\psi &\le \kappa(\dSpace \kappa +\abs{c}) \psi
\,.
\end{align}
\end{lemma}
\begin{proof}[Proof of \cref{lem:bounds_combinations_derivatives_psi}]
For every nonnegative time $t$ and for every real quantity $r$, 
\[
\abs{\partial_t \psiScalarPlusMinus(r,t)} = \kappa \cCut \psiScalarPlusMinus(r,t) 
\,,
\]
thus 
\[
\abs{\partial_t \psiScalar(r,t)} \le \kappa \cCut \psiScalar(r,t)
\,,
\]
and as a consequence, for every $\xi$ in $\rr^{\dSpace}$, 
\[
\abs{\partial_t \psiStand(\xi,t)} \le \kappa \cCut \psiStand(\xi,t) 
\,, 
\]
and \cref{bound_partial_t_psi} follows. 
Besides, it follows from the definitions \cref{def_psi_zero_and_psi} of $\psiStand$ and $\psi$ that
\[
\nabla_\xi \psi = \psi c + e^{c\cdot\xi}\nabla_\xi\psiStand 
\quad\text{thus}\quad
\psi c - \nabla_\xi \psi = - e^{c\cdot\xi}\nabla_\xi\psiStand
\,,
\]
and that
\[
\nabla_\xi \psiStand (\xi,t) = 
\left\{
\begin{aligned}
&\kappa \bigl(\psiScalarPlus(\abs{\xi},t)-\psiScalarMinus(\abs{\xi},t)\bigr)\frac{\xi}{\abs{\xi}}&&\quad\text{if}\quad \abs{\xi}<\cCut t \\
&-\kappa \psiStand(\xi,t)\frac{\xi}{\abs{\xi}} &&\quad\text{if}\quad \abs{\xi}>\cCut t
\,,
\end{aligned}
\right.
\]
and as a consequence, for all $\xi$ in $\rr^{\dSpace}$, 
\begin{equation}
\label{upper_bound_nabla_psiZero}
\abs{\nabla_\xi \psiStand(\xi,t)} \le \kappa\psiStand(\xi,t)
\,,
\end{equation}
and \cref{bound_psi_c_minus_nabla_xi_psi} follows. 
Besides, it again follows from the definitions \cref{def_psi_zero_and_psi} of $\psiStand$ and $\psi$ that
\begin{equation}
\label{expression_Delta_psi_minus_c_nabla_psi}
\begin{aligned}
\Delta_\xi \psi &=  c \cdot \nabla_\xi\psi   + e^{c\cdot\xi}( c\cdot\nabla_\xi\psiStand + \Delta_\xi \psiStand )\\
\text{thus}\qquad
\Delta_\xi \psi - c \cdot \nabla_\xi \psi &= e^{c\cdot\xi}( c\cdot\nabla_\xi\psiStand + \Delta_\xi \psiStand )
\,,
\end{aligned}
\end{equation}
and that
\begin{equation}
\label{expression_laplacian_psi_zero}
\Delta_\xi \psiStand (\xi,t) = \partial_r^2\psiScalar(\abs{\xi},t) + \frac{\dSpace-1}{\abs{\xi}} \partial_r\psiScalar(\abs{\xi},t)
\,.
\end{equation}
For all $\rho$ in $\rr$, 
\begin{equation}
\label{upper_bound_drr_psiZeroRadial}
\partial_\rho^2\psiScalar(\rho,t) \le \kappa^2 \psiScalar(\rho,t)
\end{equation}
(indeed the difference $\partial_\rho^2\psiScalar(\rho,t)-\kappa^2 \psiScalar(\rho,t)$ is made of two Dirac masses of negative weight, at $\pm\cCut t$ for the argument $\rho$). On the other hand, 
\[
\psiScalar(\rho,t) = 
\left\{
\begin{aligned}
& e^{\kappa(\rho-\cCut t)} + e^{\kappa(-\rho-\cCut t)}= 2 e^{-\kappa\cCut t}\cosh(\kappa \rho) &&\quad\text{if}\quad \rho<\cCut t \\
& e^{-\kappa(\rho-\cCut t)} + e^{\kappa(-\rho-\cCut t)}= 2 e^{-\kappa \rho} \cosh(\kappa \cCut t)&&\quad\text{if}\quad \rho>\cCut t
\,,
\end{aligned}
\right.
\]
and as a consequence, 
\begin{equation}
\label{upper_bound_dr_psiZeroRadial_over_r_times_psiZeroRadial}
\frac{\partial_\rho\psiScalar(\rho,t)}{\rho\psiScalar(\rho,t)}= 
\left\{
\begin{aligned}
& \frac{\kappa \sinh(\kappa \rho)}{\rho\cosh(\kappa \rho)} = \kappa^2 \frac{\tanh(\kappa \rho)}{\kappa \rho}\le \kappa^2 &&\quad\text{if}\quad \rho<\cCut t \\
& -\frac{\kappa}{\rho}\le 0 &&\quad\text{if}\quad \rho>\cCut t
\,.
\end{aligned}
\right.
\end{equation}
It follows from \cref{expression_laplacian_psi_zero,upper_bound_drr_psiZeroRadial,upper_bound_dr_psiZeroRadial_over_r_times_psiZeroRadial} that, for all $\xi$ in $\rr^{\dSpace}$, 
\[
\Delta_\xi \psiStand (\xi,t) \le \dSpace \kappa^2 \psiStand (\xi,t)
\,, 
\]
and according to \cref{expression_Delta_psi_minus_c_nabla_psi,upper_bound_nabla_psiZero}, \cref{bound_Delta_psi_minus_c_nabla_xi_psi} follows. 
\end{proof}
\begin{proof}[End of the proof of \cref{lem:der_fire_dichot}]
It follows from inequality \cref{dt_fff_prel_proof_1} and from \cref{lem:bounds_combinations_derivatives_psi} that, for every nonnegative time $t$,
\[
\begin{aligned}
\fff'(t) = & \int_{\rr^{\dSpace}} \psi\Biggl[ - \coeffEn v_t^2 - v\cdot \nabla V^\dag(v) - \abs{\nabla_\xi v}^2 
+ \coeffEn\kappa\cCut  \Bigl( \frac{1}{2}\abs{\nabla_\xi v}^2 + \abs{V^\dag(v)} \Bigr)   \\
& + \coeffEn \kappa \abs{ \nabla_\xi v \cdot v_t }
+ \frac{\kappa(\dSpace \kappa +\abs{c}+\cCut)}{2} v^2 \Biggr] \, d\xi
\,.
\end{aligned}
\]
Using the inequality
\[
\coeffEn \kappa \abs{ \nabla_\xi v \cdot v_t } \le \coeffEn v_t^2 + \coeffEn \frac{\kappa^2}{4} \abs{\nabla_\xi v}^2
\,,
\]
it follows that
\[
\begin{aligned}
\fff'(t)\le \int_{\rr^{\dSpace}} \psi \Biggl[ &
\biggl( \coeffEn \kappa\Bigl( \frac{\cCut}{2} + \frac{\kappa}{4} \Bigr) - 1 \biggr) \abs{\nabla_\xi v}^2 - v\cdot \nabla V^\dag(v) \\
& + \coeffEn \, \kappa\,  \cCut  \, \abs{V^\dag(v)} + \frac{\kappa(\dSpace \kappa +\abs{c}+\cCut)}{2} v^2  \Biggr] \, d\xi
\,,
\end{aligned}
\]
and according to the conditions \vref{conditions_kappa_cCut_coeffEn,conditions_c} satisfied by $\coeffEn$ and $\kappa$ and $\cCut$ and $c$, it follows that
\begin{equation}
\label{dt_fire_before_nufire}
\fff'(t) \le \int_{\rr^{\dSpace}} \psi\left[ -\frac{1}{2}\abs{\nabla_\xi v}^2 - v\cdot \nabla V^\dag(v) + \frac{1}{4} \abs{V^\dag(v)} + \frac{\eigVmin(m)}{8} v^2 \right] \, d\xi
\end{equation}
Let $\nuF$ be a positive quantity to be chosen below. It follows from the previous inequality and from the definition \cref{def_fff} of $\fff(t)$ that
\begin{equation}
\label{ds_fire_prel}
\begin{aligned}
\fff'(t) + \nuF\fff(t) \le \int_{\rr^{\dSpace}} \psi \biggl[&-\frac{1}{2}(1-\nuF\,\coeffEn)\abs{\nabla_\xi v}^2 - v\cdot \nabla V^\dag(v)  \\
&  + \Bigl(\frac{1}{4}+\nuF\coeffEn\Bigr) \abs{V^\dag(v)} +  \Bigl(\frac{\eigVmin(m)}{8} + \frac{\nuF}{2}\Bigr)v^2  \biggr]\, d\xi 
\,.
\end{aligned}
\end{equation}
In view of this expression and of inequalities \vref{v_nablaV_controls_square_around_loc_min,v_nablaV_controls_pot_around_loc_min}, let us assume that $\nuF$ is small enough so that
\begin{equation}
\label{conditions_on_nuFire}
\nuF\, \coeffEn \le 1 
\quad\text{and}\quad
\nuF\, \coeffEn \le \frac{1}{4}
\quad\text{and}\quad
\frac{\nuF}{2} \le \frac{\eigVmin(m)}{8} 
\,;
\end{equation}
the quantity $\nuF$ may be chosen as
\[
\nuF = \min\Bigl(\frac{1}{4\coeffEn}, \frac{\eigVmin(m)}{4} \Bigr)
\,.
\]
Then, it follows from \cref{ds_fire_prel,conditions_on_nuFire} that
\begin{equation}
\label{ds_fire_prel_bis}
\fff'(t) + \nuF\fff(t) \le \int_{\rr^{\dSpace}} \psi \left[- v\cdot \nabla V^\dag(v) + \frac{1}{2}\abs{V^\dag(v)} + \frac{\eigVmin(m)}{4}v^2 \right]\, d\xi
\,.
\end{equation}
According to \cref{v_nablaV_controls_square_around_loc_min,v_nablaV_controls_pot_around_loc_min}, the integrand of the integral at the right-hand side of this inequality is nonpositive as long as $\xi$ is \emph{not} in $\SigmaEscTrav(t)$.
Therefore this inequality still holds if the domain of integration of this integral is changed from $\rr$ to $\SigmaEscTrav(t)$. Thus, according to the uniform bound \vref{maximal_radius_excursion_Linfty} on the solution, if $\KF$ is chosen as the quantity $\KFzero$ defined in \cref{def_KFzero}, then inequality \cref{der_fire_dichot} follows from \cref{ds_fire_prel_bis} --- with the domain of integration of the integral on the right-hand side restricted to $\SigmaEscTrav(t)$. This finishes the proof of \cref{lem:der_fire_dichot}.
\end{proof}
\subsubsection{Control over the pollution in the time derivative of the firewalls}
\label{subsubsec:flux_der_fire}
The following lemma calls upon the notation $T$ introduced to state inequalities \vref{hyp_rHom_and_rEsc_up_to_origin_of_time}.
\begin{lemma}[firewall linear decrease up to pollution, continuation]
\label{lem:der_fire_next}
There exists a positive quantity $K'_\fff$, depending only on $V$ and $m$ and $\dSpace$, such that, for every time $t$ greater than or equal to $T$,
\begin{equation}
\label{der_fire_dichot_next}
\fff'(t) \le - \nuF \fff(t) + \KFprime \exp\Bigl(-\frac{\kappa\cCut}{2}t\Bigr)
\,.
\end{equation}
\end{lemma}
\begin{proof}
Let $t$ denote a time greater than or equal to $T$. According to the definition of $\rEsc(t)$,
\begin{equation}
\label{supset_SigmaEscStand}
\SigmaEscStand(t) \subset B\bigl(\rEsc(t)\bigr) \sqcup \Bigl(\rr^{\dSpace}\setminus B\bigl(\rHom(t)\bigr)\Bigr)
\,,
\end{equation}
thus
\[
\SigmaEscTrav(t) \subset B\bigl(\rEsc(t)+\abs{c}t\bigr) \sqcup \Bigl(\rr^{\dSpace}\setminus B\bigl(\rHom(t)-\abs{c}t\bigr)\Bigr)
\,.
\]
Since according to the conditions \vref{conditions_c} the quantity $\abs{c}$ is not larger than $\cCut/6$, it follows from inequalities \cref{hyp_rHom_and_rEsc_up_to_origin_of_time} that
\[
\rEsc(t)+\abs{c}t \le \frac{\cCut}{3}t
\quad\text{and}\quad
\frac{5\cCut}{3}t\le \rHom(t)-\abs{c}t
\,,
\]
thus
\[
\SigmaEscTrav(t) \subset B\Bigl(\frac{\cCut}{3}t\Bigr) \sqcup \biggl[\rr^{\dSpace}\setminus B\Bigl(\frac{5\cCut}{3}t\Bigr)\biggr]
\,.
\]
As a consequence, it follows from inequality \cref{der_fire_dichot} of \cref{lem:der_fire_dichot} that
\[
\fff'(t) + \nuF \fff(t) \le \int_{B\Bigl(\frac{\cCut}{3}t\Bigr)} \psi(\xi,t) \, d\xi + \int_{\rr^{\dSpace}\setminus B\Bigl(\frac{5\cCut}{3}t\Bigr)} \psi(\xi,t) \, d\xi
\,.
\]
According to the definition of $\psi$, 
\[
\psi(\xi,t)\le 
\left\{
\begin{aligned}
2e^{\abs{c}\cdot\abs{\xi}-\kappa(\cCut t-\abs{\xi})}
\quad&\text{for}\quad 
\xi \in B\Bigl(\frac{\cCut}{3}t\Bigr) \\
2e^{\abs{c}\cdot\abs{\xi}+\kappa(\cCut t-\abs{\xi})}
\quad&\text{for}\quad 
\xi \in \rr^{\dSpace}\setminus B\Bigl(\frac{5\cCut}{3}t\Bigr)
\,.
\end{aligned}
\right.
\]
It follows from these upper bounds that (with the notation of \cref{subsubsec:notation_for_integrals})
\[
e^{\frac{\kappa\cCut}{2}t}\bigl(\fff'(t) + \nuF \fff(t)\bigr)\le 2 S_{\dSpace-1} (\Icentre(t)+\Iperih(t))
\,,
\]
where
\[
\begin{aligned}
\Icentre(t) &= e^{-\frac{\kappa\cCut}{2} t}\int_0^{\frac{\cCut}{3}t} \rho^{\dSpace-1} e^{(\kappa+\abs{c})\rho} \, d\rho \\
\text{and}\qquad
\Iperih(t) &= e^{\frac{3\kappa\cCut}{2} t}\int_{\frac{5\cCut}{3}t}^{+\infty} \rho^{\dSpace-1} e^{-(\kappa-\abs{c})\rho}\, d\rho 
\,.
\end{aligned}
\]
All what remains to be done is to prove that the two quantities $\Icentre(t)$ and $\Iperih(t)$ are bounded from above by quantities depending only on $V$ and $m$ and $\dSpace$. On the one hand,  
\[
\Icentre \le e^{-\frac{\kappa\cCut}{2} t} \Bigl(\frac{\cCut}{3}t\Bigr)^{\dSpace-1}\frac{1}{\kappa+\abs{c}}e^{(\kappa+\abs{c})\frac{\cCut}{3}t}
\,,
\]
and since according to the assumptions \cref{conditions_c} the quantity $\abs{c}$ is not larger than $\kappa/4$, 
\begin{align}
\Icentre(t) &\le \frac{1}{\kappa}\Bigl(\frac{\cCut}{3}t\Bigr)^{\dSpace-1} e^{-\frac{\kappa\cCut}{12} t} 
\nonumber\\
&\le \frac{4^{\dSpace-1}}{\kappa^{\dSpace}} \Bigl[\max_{\tau\in\rr} \tau^{\dSpace-1}e^{-\tau}\Bigr] 
\nonumber\\
&= \frac{4^{\dSpace-1}}{\kappa^{\dSpace}}\cdot (\dSpace-1)^{\dSpace-1}e^{-(\dSpace-1)}
\,.
\label{upper_bound_Icentre}
\end{align}
On the other hand, writing $\tilde{\rho} = (\kappa-\abs{c})\rho$ in the expression of $\Iperih(t)$ yields
\[
\Iperih(t) = \frac{e^{\frac{3\kappa\cCut}{2} t}}{(\kappa-\abs{c})^{\dSpace}}\int_{\frac{5(\kappa-\abs{c})\cCut}{3}t}^{+\infty} \tilde{\rho}^{\dSpace-1} e^{-\tilde{\rho}}\, d\tilde{\rho}
\,,
\]
thus, with the notation of \cref{subsubsec:notation_for_integrals},  
\[
\Iperih(t) = \frac{e^{\frac{3\kappa\cCut}{2} t}}{(\kappa-\abs{c})^{\dSpace}} \, (\dSpace-1)! \ e^{-\frac{5(\kappa-\abs{c})\cCut}{3}t}\ e_{\dSpace-1}\Bigl(\frac{5\cCut(\kappa-\abs{c})}{3}t\Bigr)
\,,
\]
and since according to the assumptions \cref{conditions_c} the quantity $\abs{c}$ is not larger than $\kappa/20$,
\begin{align}
\Iperih(t) &\le \frac{(\dSpace-1)!}{(\kappa-\abs{c})^{\dSpace}} \, e^{-\frac{\kappa\cCut}{12} t}\ e_{\dSpace-1}\Bigl(\frac{5(\kappa-\abs{c})\cCut}{3}t\Bigr) 
\nonumber\\
&\le \frac{2^{\dSpace-1}(\dSpace-1)!}{\kappa^{\dSpace}}\,  e^{-\frac{\kappa\cCut}{12} t}\ e_{\dSpace-1}\Bigl(\frac{5\kappa\cCut}{3}t\Bigr)
\,.
\label{upper_bound_Iperiph}
\end{align}
Both quantities \cref{upper_bound_Icentre,upper_bound_Iperiph} are bounded from above by quantities depending only on $V$ and $m$ and $\dSpace$. \Cref{lem:der_fire_next} is proved. 
\end{proof}
\subsubsection{Nonnegativity of firewall} 
%
\begin{lemma}[nonnegativity of firewall]
\label{lem:nonnegativity_F}
For every nonnegative time $t$, 
\begin{equation}
\label{nonnegativity_F}
\fff(t) \ge 0
\,.
\end{equation}
\end{lemma}
\begin{proof}
Inequality \cref{nonnegativity_F} follows from inequality \vref{def_weight_en_V_dag} and from the definition \cref{def_fff} of the firewall. 
\end{proof}
\subsubsection{Energy decrease up to pollution}
%
\begin{lemma}[energy decrease up to pollution]
\label{lem:dt_en_final_dichot}
There exist positive quantities $\KE$ and $\nuE$ (depending only on $V$ and $m$ and $\dSpace$) such that, for every time $t$ greater than or equal to $T$,
\begin{equation}
\label{dt_en_final_dichot}
\eee'(t) \le -\frac{1}{2} \ddd(t) + \KE \exp\bigl(-\nuE (t-T)\bigr)
\,.
\end{equation}
\end{lemma}
\begin{proof}
Let 
\[
\nuE = \min\Bigl(\nuF,\frac{\kappa\cCut}{4}\Bigr)
\,.
\]
According to Grönwall's inequality, it follows from inequalities \cref{der_fire_dichot_next} of \cref{lem:der_fire_next} that, for every time $t$ greater than or equal to $T$, 
\begin{align}
\fff(t) &\le \exp\bigl(-\nuF(t-T)\bigr)\fff(T) + \KFprime\int_T^t \exp\bigl(- \nuF(t-s)\bigr)\exp\Bigl(-\frac{\kappa\cCut}{2}s\Bigr) \, ds \nonumber\\
&\le \exp\bigl(- \nuE(t-T)\Biggl(\fff(T) + \KFprime\exp\Bigl(-\frac{\kappa\cCut}{2}T\Bigr)\times \nonumber\\
& \qquad\int_T^t \exp\Bigl(- \bigl(\nuF-\nuE\bigr)(t-s)\Bigr)\exp\biggl(-\Bigl(\frac{\kappa\cCut}{2}-\nuE\Bigr)(s-T)\biggr)\, ds \Biggr) \nonumber\\
&\le \exp\bigl(- \nuE(t-T)\bigr)\left(\fff(T) + \KFprime\int_T^t \exp\Bigl(-\frac{\kappa\cCut}{4}(s-T)\Bigr)\, ds\right)\nonumber\\
&\le \left(\fff(T) + \frac{4 \KFprime}{\kappa \cCut}\right)\exp\bigl(- \nuE(t-T)\bigr)
\,.
\label{exponential_decrease_fff_of_t}
\end{align}
According to the $H^1_\text{ul}$-bound \vref{hyp_attr_ball_Honeul} for the solution, there exists a positive quantity $\fffInit$, depending only on $V$ and $m$ and $\dSpace$, such that $\fff(T)$ is not larger than $\fffInit$.
Thus, introducing the quantity
\[
\KE = \KEF \left(  \fffInit + \frac{4 \KFprime}{\kappa \cCut}\right)
\,,
\]
inequality \cref{dt_en_final_dichot} follows from inequality \cref{exponential_decrease_fff_of_t} and from inequality \cref{ds_en_trav_f_inprogress_relax} of \cref{lem:ds_en_trav_f_inprogress_relax}. \Cref{lem:dt_en_final_dichot} is proved. 
\end{proof}
\subsection{Nonnegative residual asymptotic energy}
\label{subsec:low_bd_en}
\subsubsection{Notation}
Let us keep the notation and hypotheses of the previous \namecref{subsec:relax_sc_stand}. For every velocity $c$ in $\rr^{\dSpace}$ close enough to $0_{\rr^{\dSpace}}$ so that conditions \vref{conditions_c} be satisfied, let us denote by 
\[
\begin{aligned}
&v^{(c)}(\cdot,\cdot)
\quad\text{and}\quad
\chi^{(c)}(\cdot,\cdot)
\quad\text{and}\quad
\eee^{(c)}(\cdot)
\quad\text{and}\quad
\SigmaEscTrav^{(c)}(\cdot)
\quad\text{and}\quad
\ddd^{(c)}(\cdot) \,, \\
\text{and}\quad
&\psi^{(c)}(\cdot,\cdot)
\quad\text{and}\quad
\fff^{(c)}(\cdot)
\end{aligned}
\]
the objects that were defined in \cref{subsec:relax_sc_stand} (with the same notation except the ``$(c)$'' superscript that is here to remind that these objects depend on $c$). For every such $c$, let us introduce the quantity $\eeeResAsympt^{(c)}[u]$ in $\{-\infty\}\cup\rr$ defined as
\[
\eeeResAsympt^{(c)}[u] = \liminf_{t\to+\infty} \eee^{(c)}(t)
\,,
\]
and let us call ``residual asymptotic energy at velocity $c$'' this quantity. According to estimate \cref{dt_en_final_dichot} above, for every such $c$, 
\begin{equation}
\label{convergence_energy_almost_standing_frame}
\eee^{(c)}(t)\to \eeeResAsympt^{(c)}[u]
\quad\text{as}\quad
t\to + \infty
\,.
\end{equation}
\subsubsection{Statement}
The aim of this \namecref{subsec:low_bd_en} is to prove the following proposition.
\begin{proposition}[nonnegative residual asymptotic energy at zero velocity]
\label{prop:nonneg_asympt_en}
The quantity\\ 
$\eeeResAsympt^{(0)}[u]$ (the residual asymptotic energy at velocity zero) is nonnegative.
\end{proposition}
The proof proceeds through the following lemmas and corollaries, that are rather direct consequences of the relaxation scheme set up in the previous \cref{subsec:relax_sc_stand}, and in particular of the inequality \cref{dt_en_final_dichot} for the time derivative of the energy.
\subsubsection{Nonnegative residual asymptotic energy for small nonzero velocities}
\begin{lemma}[nonnegative residual asymptotic energy for small nonzero velocities]
\label{lem:nonneg_en_prel}
For every \emph{nonzero} velocity $c$ close enough to $0_{\rr^{\dSpace}}$ so that conditions \vref{conditions_c} be satisfied, 
\begin{equation}
\label{nonnegativity_Ec_t_to_infty}
\eeeResAsympt^{(c)}[u] \ge 0
\,.
\end{equation}
\end{lemma}
\begin{proof}
Let $c$ be a \emph{nonzero} vector of $\rr^{\dSpace}$, close enough to $0_{\rr^{\dSpace}}$ so that conditions \vref{conditions_c} be satisfied. For every nonnegative time $t$, according to the definition of $\eee$ \vref{def_E},
\[
\begin{aligned}
\eee^{(c)}(t) &= \int_{\rr^{\dSpace}} \chi^{(c)}(\xi,t)\biggl( \frac{1}{2}\abs{\nabla_\xi v^{(c)}(\xi,t)}^2 + V^\dag\bigl( v^{(c)}(\xi,t)\bigr) \biggr) \, d\xi \\
&\ge \int_{\rr^{\dSpace}} \chi^{(c)}(\xi,t) V^\dag\bigl( v^{(c)}(\xi,t)\bigr) \, d\xi \\
&\ge \int_{\SigmaEscTrav^{(c)}(t)} \chi^{(c)}(\xi,t) V^\dag\bigl( v^{(c)}(\xi,t)\bigr) \, d\xi
\,.
\end{aligned}
\]
Thus, considering the global minimum value of $V^\dag$:
\[
V^\dag_{\min} = \min_{u\in\rr^{\dSpace}} V^\dag(u) \le 0 
\,,
\]
it follows that
\begin{align}
\nonumber
\eee^{(c)}(t) &\ge V^\dag_{\min} \int_{\SigmaEscTrav^{(c)}(t)} \chi^{(c)}(\xi,t) \, d\xi \\
\label{upper_bound_Ec_proof}
&= V^\dag_{\min} \int_{\SigmaEscStand(t)} \chi^{(c)}(x-ct,t) \, dx
\,.
\end{align}
As already mentioned in \cref{supset_SigmaEscStand}, according to the definition of $\rEsc(t)$ and to the inequalities \vref{hyp_rHom_and_rEsc_up_to_origin_of_time},
\begin{equation}
\label{SigmaEscStand_subset_}
\SigmaEscStand(t) \subset B\bigl(\rEsc(t)\bigr) \sqcup \Bigl(\rr^{\dSpace}\setminus B\bigl(\rHom(t)\bigr)\Bigr)
\,,
\end{equation}
thus
\[
\eee^{(c)}(t) \ge V^\dag_{\min}\bigl(\Jcentre(t) + \Jperih(t)\bigr)
\,
\]
where 
\[
\begin{aligned}
\Jcentre(t) &= \int_{B\bigl(\rEsc(t)\bigr)}\chi^{(c)}(x-ct,t) \, dx \\
\text{and}\qquad
\Jperih(t) &= \int_{\rr^{\dSpace}\setminus B\bigl(\rHom(t)\bigr)}\chi^{(c)}(x-ct,t) \, dx
\,.
\end{aligned}
\]
All what remains to be done is to prove that the two (nonnegative) quantities $\Jcentre(t)$ and $\Jperih(t)$ go to $0$ as $t$ goes to $+\infty$. 
\begin{itemize}
\item For every $x$ in $B\bigl(\rEsc(t)\bigr)$, according to the definition of $\chi$ \vref{def_chi_zero_chi}, 
\[
\chi^{(c)}(x-ct,t) \le \exp\bigl(c\cdot(x-ct)\bigr)\le \exp(\abs{c}\cdot\abs{x}-c^2 t)\le \exp\bigl(\abs{c}\, \rEsc(t) - c^2 t\bigr)
\,,
\]
thus, with the notation of \cref{subsubsec:notation_for_integrals}, 
\[
\begin{aligned}
\Jcentre(t) &\le \exp\bigl(\abs{c}\, \rEsc(t) - c^2 t\bigr) \int_{B\bigl(\rEsc(t)\bigr)}\, dx \\
&= \exp\bigl(\abs{c}\, \rEsc(t) - c^2 t\bigr) \int_0^{\rEsc(t)} S_{\dSpace-1}\, r^{\dSpace-1} \, dr \\
&= \exp\bigl(\abs{c}\, \rEsc(t) - c^2 t\bigr) \frac{S_{\dSpace-1} \, \rEsc(t)^{\dSpace}}{\dSpace} \\
&= t^\dSpace\exp\Biggl(t\biggl(-c^2 + \abs{c}\frac{\rEsc(t)}{t}\biggr)\Biggr)\frac{S_{\dSpace-1}}{\dSpace}\left(\frac{\rEsc(t)}{t}\right)^{\dSpace}
\,.
\end{aligned}
\]
Since according to hypothesis \hypNoInv the ratio $\rEsc(t)/t$ goes to $0$ as $t$ goes to $+\infty$ and sine $c$ is not equal to $0_{\rr^{\dSpace}}$, it follows that $\Jcentre(t)$ goes to $0$ as $t$ goes to $+\infty$. 
\item For every $x$ in $\rr^{\dSpace}\setminus B\bigl(\rHom(t)\bigr)$, according to the definition of $\chi$ \vref{def_chi_zero_chi}, 
\[
\begin{aligned}
\chi^{(c)}(x-ct,t) &\le \exp\bigl(c\cdot(x-ct)- \kappa(\abs{x-ct}-\cCut t)\bigr) \\
&\le \exp\bigl(- (\kappa-\abs{c})(\abs{x}-\abs{c} t) + \kappa\cCut t\bigr) \\
&\le \exp\bigl(- (\kappa-\abs{c})\abs{x}+\kappa(\cCut+\abs{c}) t\bigr)
\,.
\end{aligned}
\]
It follows that, with the notation of \cref{subsubsec:notation_for_integrals},
\[
\Jperih(t) \le \exp\bigl(\kappa(\cCut+\abs{c}) t\bigr) \int_{\rHom(t)}^{+\infty} S_{\dSpace-1}\, r^{\dSpace-1} \exp\bigl(- (\kappa-\abs{c}) r \bigr)\, dr
\,,
\]
or in other words, writing $\tilde{r}=(\kappa-\abs{c}) r$ in the integral on the right-hand side, 
\[
\begin{aligned}
\Jperih(t) &\le\exp\bigl(\kappa(\cCut+\abs{c}) t\bigr)\,\frac{S_{\dSpace-1}}{(\kappa-\abs{c})^{\dSpace}}\int_{(\kappa-\abs{c})\rHom(t)}^{+\infty} \tilde{r}^{\dSpace-1} e^{-\tilde{r}}\, d\tilde{r} \\
= \frac{(\dSpace-1)!\,S_{\dSpace-1}}{(\kappa-\abs{c})^{\dSpace}}\, &\exp\bigl[\kappa(\cCut+\abs{c}) t -(\kappa-\abs{c})\rHom(t)\bigr] \ e_{\dSpace-1}\bigl((\kappa-\abs{c})\rHom(t)\bigr)
\,,
\end{aligned}
\]
and since according to inequality \vref{hyp_rHom_and_rEsc_up_to_origin_of_time} the quantity $\rHom(t)$ is not smaller than $11\cCut t/6$ for $t$ greater than or equal to $T$, and in view of the conditions \vref{conditions_c} on $c$, it follows that $\Jperih(t)$ goes to $0$ as $t$ goes to $+\infty$. 
\end{itemize}
\Cref{lem:nonneg_en_prel} is proved. 
\end{proof}
\subsubsection{Almost nonnegative energy at small nonzero velocities}
\begin{corollary}[almost nonnegative energy at small nonzero velocities]
\label{cor:low_bd_en_c}
For every \emph{nonzero} velocity $c$ close enough to $0_{\rr^{\dSpace}}$ so that conditions \vref{conditions_c} be satisfied and for every time $t$ greater than or equal to $T$, 
\[
\eee^{(c)}(t) \ge -\frac{\KE}{\nuE} \exp\left(-\nuE (t-T)\right)
\,.
\]
\end{corollary}
\begin{proof}
This lower bound follows from the nonnegativity of $\eeeResAsympt^{(c)}[u]$ (\cref{lem:nonneg_en_prel}) and the upper bound \cref{dt_en_final_dichot} of \cref{lem:dt_en_final_dichot} for the time derivative of the energy $\eee^{(c)}(t)$.
\end{proof}
\subsubsection{Continuity of energy with respect to the velocity at \texorpdfstring{$c=0_{\rr^{\dSpace}}$}{c=0}}
\begin{lemma}[continuity of energy with respect to the velocity at $c=0_{\rr^{\dSpace}}$]
\label{lem:cont_en_c}
For every nonnegative time $t$, 
\[
\eee^{(c)}(t)\to \eee^{(0)}(t)
\quad\text{as}\quad
c \to 0
\,.
\]
\end{lemma}
\begin{proof}
For every nonnegative time $t$, 
\[
\eee^{(0)}(t) = \int_{\rr^{\dSpace}} \chi^{(0)}(x,t)\biggl( \frac{1}{2}\abs{\nabla_x u^\dag(x,t)}^2 + V^\dag\bigl(u^\dag(x,t)\bigr) \biggr) \, dx
\,,
\]
and, for every velocity $c$ close enough to $0_{\rr^{\dSpace}}$ so that conditions \vref{conditions_c} be satisfied (substituting the notation $\xi$ used in the definition \cref{def_E} of $\eee(\cdot)$ with $x$),
\[
\eee^{(c)}(t) = \int_{\rr^{\dSpace}} \chi^{(c)}(ct+x,t)\biggl( \frac{1}{2}\abs{\nabla_x u^\dag(ct+x,t)}^2 + V^\dag\bigl(u^\dag(ct+x,t)\bigr) \biggr) \, dx
\,,
\]
and the result follows from the continuity of $\chi^{(c)}(\cdot,\cdot)$ with respect to $c$, the exponential decrease to $0$ of $\chi^{(c)}(x,t)$ when $\abs{x}\to+\infty$, and the $\HoneulAlone$-bounds \vref{hyp_attr_ball_Honeul} for the solution.  
\end{proof}
\subsubsection{Almost nonnegative energy in a standing frame}
\begin{corollary}[almost nonnegative energy in a standing frame]
\label{cor:low_bd_en_0}
For every time $t$ greater than or equal to $T$,
\[
\eee^{(0)}(t) \ge -\frac{\KE}{\nuE} \exp\left(-\nuE (t-T)\right)
\,.
\]
\end{corollary}
\begin{proof}
This lower bound follows from \cref{cor:low_bd_en_c} and \cref{lem:cont_en_c}.
\end{proof}
\Cref{prop:nonneg_asympt_en} (``nonnegative residual asymptotic energy at zero velocity'') follows from \cref{cor:low_bd_en_0}.
\subsection{End of the proof of \texorpdfstring{\cref{thm:no_inv_implies_relaxation}}{Theorem \ref{thm:no_inv_implies_relaxation}}}
\label{subsec:end_pf_thm_inv_no_inv_dichotomy_hom}
\begin{lemma}[integrability of dissipation in the standing frame]
\label{lem:dissip_is_integrable}
The function $t\mapsto \ddd^{(0)}(t)$ is integrable on a neighbourhood of $+\infty$. 
\end{lemma}
\begin{proof}
The statement follows from \cref{prop:nonneg_asympt_en} (``nonnegative residual asymptotic energy at zero velocity'') and from the upper bound \cref{dt_en_final_dichot} on the time derivative of energy.
\end{proof}
\begin{lemma}[relaxation]
\label{lem:relaxation}
The following limit holds:
\[
\sup_{x\in B\bigl(\rHom(t)\bigr)} \abs{ u_t(x,t) } \to 0
\quad\text{as}\quad
t\to+\infty
\,.
\]
\end{lemma}
\begin{proof}
For every positive quantity $L$, it follows from inclusion \cref{supset_SigmaEscStand} and the first inequality of \cref{hyp_rHom_and_rEsc_up_to_origin_of_time} and hypothesis \hypHom that, for every large enough positive time $t$, 
\[
\SigmaEscStand(t)\subset B\left(\frac{1}{6}\cCut t\right)\sqcup\rr^{\dSpace}\setminus B\bigl(\rHom(t)+L\bigr)
\,.
\]
Proceeding as in the proof of \cref{lem:exponential_decrease_firewall}, it follows, using the notation $\fff_0(\cdot,\cdot)$ introduced in \cref{def_firewall_sf}, that
\[
\sup_{x\in B\bigl(\rHom(t)\bigr)\setminus B(\cCut t)} \fff_0(x,t)\to0
\quad\text{as}\quad
t\to+\infty
\,.
\]
As a consequence, according to the bounds \cref{bound_u_ut_ck} on the solution,
\[
\sup_{x\in B\bigl(\rHom(t)\bigr)\setminus B(\cCut t)}\abs{u(x,t)}+\abs{\nabla_x u(x,t)}+\abs{\Delta_x u(x,t)}\to 0
\quad\text{as}\quad
t\to+\infty
\,.
\]
Thus, according to system \cref{syst_sf}, 
\begin{equation}
\label{ut_goes_to_zero_beyond_cCut_t}
\sup_{x\in B\bigl(\rHom(t)\bigr)\setminus B(\cCut t)} \abs{u_t(x,t)}\to 0
\quad\text{as}\quad
t\to+\infty
\,.
\end{equation}
It remains to prove that 
\[
\sup_{x\in B(\cCut t)} \abs{ u_t(x,t) } \to 0
\quad\text{as}\quad
t\to+\infty
\,.
\]
Let us proceed by contradiction and assume that the converse assumption holds. Then, there exists a positive quantity $\varepsilon$ and a sequence $(x_n,t_n)_{n\in\nn}$ such that $t_n\to+\infty$ as $n\to+\infty$, and such that, for every $n$ in $\nn$, 
\begin{equation}
\label{contradiction_hyp_u_t_not_small}
\abs{u_t(x_n,t_n)}\ge\varepsilon
\,.
\end{equation}
According to \cref{ut_goes_to_zero_beyond_cCut_t}, it may be assumed (up to dropping the first terms of the sequence $(x_n,t_n)_{n\in\nn}$) that, for every $n$ in $\nn$, $x_n$ is in the ball $B(\cCut t_n)$.
By compactness (\cref{lem:compactness}), there exists an entire solution $\widebar{u}$ of system \cref{syst_sf} such that, up to replacing the sequence $(x_n,t_n)_{n\in\nn}$ by a subsequence, with the notation of \cref{compactness},
\begin{equation}
\label{convergence_up_to_subsequence}
D^{2,1}u(x_n+\cdot,t_n+\cdot)\to D^{2,1}\widebar{u}
\quad\text{as}\quad
n\to+\infty
\,,
\end{equation}
uniformly on every compact subset of $\rr^{\dSpace}\times\rr$. It follows from \cref{contradiction_hyp_u_t_not_small,convergence_up_to_subsequence} that the quantity $\abs{\widebar{u}_t(0,0)}$ is positive, so that the quantity
\[
\int_0^1 \left(\int_{\rr^{\dSpace}}e^{-\kappa\abs{\xi}}\widebar{u}_t(\xi,s)^2 \, d\xi\right)\, ds
\] 
is also positive. This quantity is less than or equal to the quantity
\[
\liminf_{n\to+\infty}\int_0^1 \ddd^{(0)}(t_n+s) \, ds
\,,
\]
which is therefore also positive, a contradiction with the integrability of $t\mapsto\ddd^{(0)}(t)$ (\cref{lem:dissip_is_integrable}). 
\Cref{lem:relaxation} is proved. 
\end{proof}
The following lemma calls upon the notation $\eee_c(t)$ introduced in \cref{def_eee_c_ddd_c_bbb_c}, and its proof upon the notation $E^\dag(x,t)$ introduced in \cref{def_E_F_integrands_of_eee_fff}. 
\begin{lemma}[convergence towards residual asymptotic energy]
\label{lem:asympt_energy_for_various_bounds}
For every quantity $c$ in the interval $(0,\cHom)$, 
\begin{equation}
\label{equ_asympt_energy_for_various_bounds}
\eee_c(t) \to \eeeResAsympt^{(0)}[u]
\end{equation}
as time goes to $+\infty$.  
\end{lemma}
\begin{proof}
Let $c$ be a quantity in the interval $(0,\cHom)$, and, for every nonnegative time $t$, let us introduce the set
\[
\Sigma(t) = B\bigl(\max(\cCut,c)t\bigr)\setminus B\bigl(\min(\cCut,c)t\bigr)
\]
and the quantity
\[
\delta\eee(t) = \eee^{(0)}(t) - \eee_c(t) = \int_{\rr^{\dSpace}\setminus B(\cCut t)} \chi^{(0)}(x,t) E^\dag(x,t)\, dx + \int_{\Sigma(t)} E^\dag(x,t)\, dx
\,.
\]
According to \cref{convergence_energy_almost_standing_frame}, the quantity $\eee^{(0)}(t)$ goes to $\eeeResAsympt^{(0)}[u]$ as $t$ goes to $+\infty$. As a consequence, all what remains to be proved is that $\delta\eee(t)$ goes to $0$ as $t$ goes to $+\infty$. 
Let us introduce the integrals
\[
\begin{aligned}
\iii(t) &= \int_{B\bigl(\rHom(t)\bigr)\setminus B(\cCut t)} \chi^{(0)}(x,t) \abs{E^\dag(x,t)}\, dx \,, \\
\text{and}\quad
\jjj(t) &= \int_{\rr^{\dSpace}\setminus B\bigl(\rHom(t)\bigr)} \chi^{(0)}(x,t) \abs{E^\dag(x,t)}\, dx \,, \\
\text{and}\quad
\kkk(t) &= \int_{\Sigma(t)} \abs{E^\dag(x,t)}\, dx
\,.
\end{aligned}
\]
According to this notation,
\begin{equation}
\label{upper_bound_delta_eee_of_t}
\abs{\delta\eee(t)} \le \iii(t) + \jjj(t) + \kkk(t)
\,.
\end{equation}
\begin{itemize}
\item It follows from the definition of $\rEsc(t)$, from the first inequality of \cref{hyp_rHom_and_rEsc_up_to_origin_of_time}, and from hypothesis \hypNoInv that, for every large enough positive time $t$,
\[
\text{for every $x$ in $B\bigl(\rHom(t)\bigr)\setminus B(\cCut t)$,}\quad \abs{u^\dag(x,t)}\le \dEsc(m)
\,,
\]
so that, actually,
\[
\iii(t) = \int_{B\bigl(\rHom(t)\bigr)\setminus B(\cCut t)} \chi^{(0)}(x,t) E^\dag(x,t)\, dx
\,.
\]
Thus, according to inequality \cref{chi_smaller_than_psi_outside_ball_of_radius_cCut_t}, 
\[
\begin{aligned}
\iii(t) &\le \int_{B\bigl(\rHom(t)\bigr)\setminus B(\cCut t)} \psi^{(0)}(x,t) E^\dag(x,t)\, dx \\
&= \frac{1}{\coeffEn} \int_{B\bigl(\rHom(t)\bigr)\setminus B(\cCut t)} \psi^{(0)}(x,t) \coeffEn E^\dag(x,t)\, dx \\
&\le \frac{1}{\coeffEn} \int_{B\bigl(\rHom(t)\bigr)\setminus B(\cCut t)} \psi^{(0)}(x,t) \left(\coeffEn E^\dag(x,t) + \frac{1}{2}\abs{u^\dag(x,t)^2}\right)\, dx \\
&\le \frac{1}{\coeffEn} \fff^{(0)}(t)
\,,
\end{aligned}
\]
so that, according to inequality \cref{exponential_decrease_fff_of_t}, $\iii(t)$ goes to $0$ as $t$ goes to $+\infty$. 
\item According to the bounds \cref{hyp_attr_ball_infty,hyp_attr_ball_Honeul} on the solution, there exists a (positive) quantity $E^\dag_{\max}$ such that, for every real quantity $x$ and every nonnegative time $t$, 
\[
E^\dag(x,t) \le E^\dag_{\max}
\,.
\]
It follows that
\[
\begin{aligned}
\jjj(t) &\le E^\dag_{\max} \int_{\rr^{\dSpace}\setminus B\bigl(\rHom(t)\bigr)} \chi^{(0)}(x,t) \, dx \\
&= E^\dag_{\max} \int_{\rHom(t)}^{+\infty} S_{\dSpace-1} r^{\dSpace-1} \exp\bigl(-\kappa(r-\cCut t)\bigr) \, dr \\
&= E^\dag_{\max}S_{\dSpace-1}\exp(\cCut t) \int_{\rHom(t)}^{+\infty}r^{\dSpace-1} \exp(-\kappa r) \, dr \\
&= \frac{E^\dag_{\max}S_{\dSpace-1}}{\kappa^{\dSpace}}\exp(\cCut t) \int_{\rHom(t)}^{+\infty}\rho^{\dSpace-1} \exp(-\rho) \, d\rho \\
&= \frac{E^\dag_{\max}S_{\dSpace-1}(\dSpace-1)!}{\kappa^{\dSpace}}\exp\bigl(\cCut t-\rHom(t)\bigr)\cdot e_{\dSpace-1}\bigl(\rHom(t)\bigr)
\,,
\end{aligned}
\]
and it follows from the second inequality of \cref{hyp_rHom_and_rEsc_up_to_origin_of_time} that $\jjj(t)$ goes to $0$ as $t$ goes to $+\infty$. 
\item Since $c$ is positive and smaller than $\cHom$, it follows, proceeding as in the proof of \cref{lem:exponential_decrease_firewall}, that the quantity
\[
\sup_{x\in \Sigma(t)} \fff_0(x,t)
\]
goes to $0$ at an exponential rate when $t$ goes to $+\infty$. As a consequence, according to the bounds \cref{bound_u_ut_ck}, the same is true for the quantities
\begin{equation}
\label{sup_of_u_and_nabla_x_u_in_Sigma}
\sup_{x\in \Sigma(t)}\abs{u(x,t)}
\quad\text{and}\quad
\sup_{x\in \Sigma(t)}\abs{\nabla_x u(x,t)}
\,,
\end{equation}
so that $\kkk(t)$ goes to $0$ as $t$ goes to $+\infty$. 
\end{itemize}
In view of inequality \cref{upper_bound_delta_eee_of_t}, \cref{lem:asympt_energy_for_various_bounds} is proved.  
\end{proof}
The following lemma calls upon the notation $\ddd_c(t)$ and $\bbb_c(t)$ introduced in \cref{def_eee_c_ddd_c_bbb_c}. 
\begin{lemma}[integrability of dissipation, 2]
\label{lem:integrability_of_dissipation_Dc}
For every quantity $c$ in the interval $(0,\cHom)$, the function $t\mapsto\ddd_c(t)$ is integrable on a neighbourhood of $+\infty$. 
\end{lemma}
\begin{proof}
The conclusion follows from the limit \cref{equ_asympt_energy_for_various_bounds}, from equality \cref{time_derivavive_eee_c}, and from the fact that, since the quantities \cref{sup_of_u_and_nabla_x_u_in_Sigma} go to $0$ at an exponential rate, the same is true for the quantity $\bbb_c(t)$. 
\end{proof}
In view of \cref{prop:nonneg_asympt_en,lem:relaxation,lem:asympt_energy_for_various_bounds,lem:integrability_of_dissipation_Dc}, \cref{thm:no_inv_implies_relaxation} is proved. 
\section{Proof of \texorpdfstring{\cref{thm:main}}{Theorem \ref{thm:main}}}
As everywhere else, let us consider a function $V$ in $\ccc^2(\rr^{\dState},\rr)$ satisfying the coercivity hypothesis \cref{hyp_coerc}. Let $m$ be a point of $\mmm$ and $(x,t)\mapsto u(x,t)$ be a solution of system \cref{syst_sf} stable close to $m$ at infinity. 
\subsection{Asymptotics of derivatives beyond invasion speed}
The following lemma will be called upon in the next two \namecrefs{subsec:proof_of_thm_main_no_invasion}. 
\begin{lemma}[asymptotics of time derivative beyond invasion speed]
\label{lem:derivatives_go_to_zero_beyond_invasion_speed}
For every positive quantity $c$ larger than $\cInv[u]$, there exists positive quantities $\nu'$ and $K'[u]$ such that, for every nonnegative time $t$, 
\[
\sup_{x\in\rr^{\dSpace}, \,\abs{x}\ge ct}\bigl(\abs{u_x(x,t)}+\abs{u_{xx}(x,t)}+\abs{u_t(x,t)}\bigr) \le K'[u]\exp(-\nu' t)
\,.
\]
The quantity $\nu'$ depends on $V$ and $m$ and the difference $c-\cInv[u]$ (only), whereas $K'[u]$ depends additionally on $u$.  
\end{lemma}
\begin{proof}
According to \cref{lem:exponential_decrease_beyond_invasion_speed} and to the bounds \cref{bound_u_ut_ck} on the solution, the conclusion holds for the quantity 
\[
\sup_{x\in\rr^{\dSpace}, \,\abs{x}\ge ct}\bigl(\abs{u_x(x,t)}+\abs{u_{xx}(x,t)}\bigr)
\,,
\]
and in view of system \cref{syst_sf}, the conclusion also holds for the quantity 
\[
\sup_{x\in\rr^{\dSpace},\, \abs{x}\ge ct}\abs{u_t(x,t)}
\,.
\]
\Cref{lem:derivatives_go_to_zero_beyond_invasion_speed} is proved.
\end{proof}
\subsection{Proof of conclusion \texorpdfstring{\cref{item:thm_main_no_invasion} of \cref{thm:main}}{\ref{item:thm_main_no_invasion} of Theorem \ref{thm:main}}}
\label{subsec:proof_of_thm_main_no_invasion}
Let us assume that the invasion speed $\cInv[u]$ is equal to $0$. Then, introducing the function $\rHom:[0,+\infty)\to[0,+\infty)$ defined as 
\[
\rHom(t) = t
\,,
\]
it follows from the definition of the invasion speed (\cref{def:invasion_speed}) that both hypotheses \hypHom and \hypNoInv of \cref{thm:no_inv_implies_relaxation} hold. According to \cref{prop:asympt_en} and to conclusion \cref{item:thm_no_inv_implies_relaxation_nonnegative_res_asympt_en} of \cref{thm:no_inv_implies_relaxation}, the asymptotic energy defined in \cref{prop:asympt_en} is equal to the residual asymptotic energy defined in \cref{thm:no_inv_implies_relaxation}, and according to conclusion \cref{item:thm_no_inv_implies_relaxation_nonnegative_res_asympt_en} of \cref{thm:no_inv_implies_relaxation} this asymptotic energy is nonnegative. In addition, according to conclusion \cref{item:thm_no_inv_implies_relaxation_time_derivative_goes_to_zero} of \cref{thm:no_inv_implies_relaxation}, 
\begin{equation}
\label{u_t_goes_to_zero_close_to_origin}
\sup_{x\in B(t)}\abs{u_t(x,t)}\to 0
\quad\text{as}\quad
t\to +\infty
\,.
\end{equation}
Besides, it follows from \cref{lem:derivatives_go_to_zero_beyond_invasion_speed} that, for every positive quantity $\varepsilon$,
\begin{equation}
\label{u_t_goes_to_zero_away_from_origin}
\sup_{x\in\rr^{\dSpace}\setminus B(\varepsilon t)}\abs{u_t(x,t)}\to 0
\quad\text{as}\quad
t\to +\infty
\,,
\end{equation}
and the limit \cref{time_derivative_goes_uniformly_to_zero} follows from \cref{u_t_goes_to_zero_close_to_origin,u_t_goes_to_zero_away_from_origin}. Conclusion \cref{item:thm_main_no_invasion} of \cref{thm:main} is proved. 
\subsection{Proof of conclusion \texorpdfstring{\cref{item:thm_main_invasion} of \cref{thm:main}}{\ref{item:thm_main_invasion} of Theorem \ref{thm:main}}}
\label{subsec:invasion_implies_infinitely_neg_asympt_energy}
Let us assume that the invasion speed $\cInv[u]$ is positive, and let $\eeeAsympt[u]$ denote the asymptotic energy of the solution $(x,t)\mapsto u(x,t)$. The task is to prove that $\eeeAsympt[u]$ equals $-\infty$. Let us proceed by contradiction and assume that 
\begin{equation}
\label{assume_by_contradiction_eeeAsypmt_is_finite}
\eeeAsympt[u]>-\infty
\,.
\end{equation}
\subsubsection{Uniform convergence towards zero of the time derivative of the solution}
\label{subsubsec:uniform_cv_towards_zero_of_ut}
The aim of this \namecref{subsubsec:uniform_cv_towards_zero_of_ut} is to prove the following proposition. 
\begin{proposition}[the time derivative of the solution goes to $0$ uniformly in space]
\label{prop:time_derivative_solution_goes_uniformly_to_zero}
The following limit hold:
\begin{equation}
\label{time_derivative_solution_goes_uniformly_to_zero}
\sup_{x\in\rr^{\dSpace}}\abs{u_t(x,t)}\to 0 
\quad\text{as}\quad
t\to+\infty
\,.
\end{equation}
\end{proposition}
\begin{proof}
\renewcommand{\qedsymbol}{}
Let us introduce the quantity
\[
c = \cInv[u] + 1
\,.
\]
and let us use the notation $V^\dag$ and $u^\dag$ introduced in \cref{def_normalized_potential_solution}. 

Objects similar to but different from some of them introduced in \cref{subsubsec:def_loc_en,subsubsec:der_loc_en} will now be introduced. They will be denoted similarly, with an additional tilde (``$\tilde{\cdot}$'') to avoid any confusion. 
Let us introduce the function $(\rho,t)\mapsto \tildechiScalar(\rho,t)$ defined on $\rr\times [0,+\infty)$ as
\[
\tildechiScalar(\rho,t) = 
\left\{
\begin{aligned}
1 & \quad\text{if} \quad \abs{\rho} \le t \\
\exp\bigl( -  ( \abs{\rho}- t ) \bigr) & \quad \text{if} \quad \abs{\rho}\ge  t 
\,,
\end{aligned}
\right. 
\]
and let us introduce the function $(\xi,t)\mapsto \tilde{\chi}(\xi,t)$ defined on $\rr^{\dSpace} \times [0,+\infty)$ as
\begin{equation}
\label{def_tilde_chi}
\tilde{\chi}(\xi,t) = \tildechiScalar\bigl(\abs{\xi},t\bigr)
\,.
\end{equation}
Let us recall the notation $E^\dag(x,t)$ introduced in \cref{def_E_F_integrands_of_eee_fff}. 
For every nonnegative quantity $t$, let us define the quantities 
\begin{equation}
\label{def_tilde_eee_tidle_ddd}
\tilde{\eee}(t) = \int_{\rr^{\dSpace}} \tilde{\chi}(\xi,t) E^\dag(x,t) \, dt 
\quad\text{and}\quad
\tilde{\ddd}(t) = \int_{\rr^{\dSpace}} \tilde{\chi}(\xi,t) u^\dag_t(x,t)^2 \, dt 
\,.
\end{equation}
The continuation of the proof calls upon the following intermediary results.
\end{proof}
\begin{lemma}[time derivative of localized energy]
\label{lem:time_der_loc_energy_tilde}
For every positive time $t$, 
\begin{equation}
\label{time_der_loc_energy_tilde}
\tilde{\eee}'(t) \le -\frac{1}{2}\tilde{\ddd}(t) + \int_{\rr^{\dSpace}\setminus B(t)} \tilde{\chi}(x,t)\left(\abs{\nabla_x u^\dag(x,t)}^2 + V^\dag\bigl(u^\dag(x,t)\bigr)\right)\, dx 
\,.
\end{equation}
\end{lemma}
\begin{proof}[Proof of \cref{lem:time_der_loc_energy_tilde}]
Apply the proof of \cref{lem:time_der_energy_without_fire} substituting $\kappa$ and $\cCut$ with $1$ and $c$ with $0$. 
\end{proof}
\begin{corollary}[integrability of the dissipation]
\label{cor:integrability_of_dissip_tilde}
The function $t\mapsto \tilde{\ddd}(t)$ is integrable on $[1,+\infty)$. 
\end{corollary}
\begin{proof}[Proof of \cref{cor:integrability_of_dissip_tilde}]
According to inequality \cref{time_der_loc_energy_tilde} and to \cref{lem:derivatives_go_to_zero_beyond_invasion_speed}, the function
\[
t\mapsto\tilde{\eee}'(t) + \frac{1}{2}\tilde{\ddd}(t)
\]
is bounded from above by a function going to $0$ at an exponential rate as $t$ goes to $+\infty$. Besides, it follows from \cref{prop:asympt_en,lem:derivatives_go_to_zero_beyond_invasion_speed} that $\tilde{\eee}(t)$ converges towards the finite quantity $\eeeAsympt[u]$ as $t$ goes to $+\infty$. The statement of \cref{cor:integrability_of_dissip_tilde} follows. 
\end{proof}
\begin{proof}[End of the proof of \cref{prop:time_derivative_solution_goes_uniformly_to_zero}]
The end of the proof of \cref{prop:time_derivative_solution_goes_uniformly_to_zero} is identical to the proof of \cref{lem:relaxation}. 
\end{proof}
\subsubsection{Contradiction with the positivity of the invasion speed}
The following corollary of \cref{prop:time_derivative_solution_goes_uniformly_to_zero} calls upon the notation $\fff_0(x,t)$ introduced in \cref{def_firewall_sf}. 
\begin{corollary}[the time derivative of the firewall goes to $0$ uniformly in space]
\label{cor:time_derivative_firewall_goes_uniformly_to_zero}
The following limit hold:
\begin{equation}
\label{time_derivative_firewall_goes_uniformly_to_zero}
\sup_{x\in\rr^{\dSpace}}\abs{\partial_t\fff_0(x,t)}\to 0 
\quad\text{as}\quad
t\to+\infty
\,.
\end{equation}
\end{corollary}
\begin{proof}
It follows from equality \cref{ddt_loc_en_stand_fr} that, for every $\widebar{x}$ in $\rr^{\dSpace}$ and for every positive time $t$, 
\[
\partial_t\fff_0(\widebar{x},t) = \int_{\rr^{\dSpace}}\left(T_{\widebar{x}}\psi_0\left(-\coeffEn (u^\dag_t)^2 + u^\dag \cdot u^\dag_t\right) - \coeffEn \nabla_x T_{\widebar{x}}\psi_0\cdot\nabla_x u^\dag\cdot u^\dag_t\right)\, dx
\,,
\]
and the conclusion \cref{time_derivative_firewall_goes_uniformly_to_zero} follows from \cref{prop:time_derivative_solution_goes_uniformly_to_zero} and from the bounds \cref{bound_u_ut_ck} on the solution. 
\end{proof}
\begin{lemma}[$\fff_0(x,t)$ controls $\abs{u(x,t)}$]
\label{lem:escape_Escape}
There exists a positive quantity $\desc(m)$ such that, for every $x$ in $\rr^{\dSpace}$ and for every time $t$ greater than or equal to $1$, the following implication holds: 
\begin{equation}
\label{escape_Escape}
\fff_0(x,t)\le \desc(m)^2 \implies \abs{u(x,t)}\le \dEsc(m)
\,.
\end{equation}
\end{lemma}
\begin{proof}
Implication \cref{escape_Escape} follows from the coercivity \cref{coerc_fire} of $\fff_0(\cdot,\cdot)$ and the bounds \cref{bound_u_ut_ck} on the solution (note that $\dEsc(m)$ and $\desc(m)$ do not denote the same quantity). 
\end{proof}
The following lemma calls upon the quantities $\nuFzero$ and $\KFzero$ introduced in \cref{lem:approx_decrease_fire}. 
\begin{lemma}[control over pollution]
\label{lem:control_over_pollution}
There exists a positive quantity $L$ such that the following inequality holds: 
\begin{equation}
\label{control_over_pollution}
\KFzero \int_{\rr^{\dSpace}\setminus B(L)} \psi_0(x) \, dx \le \nuFzero \frac{\desc(m)^2}{8}
\,.
\end{equation}
\end{lemma}
\begin{proof}
For every positive quantity $L$, using the notation introduced in \cref{integral_of_r_power_n_exp_minus_r}, 
\[
\begin{aligned}
\int_{\rr^{\dSpace}\setminus B(L)} \psi_0(x) \, dx &= \int_L^{+\infty}S_{\dSpace-1}\exp(-\kappa_0 r)\, dr \\
&= \frac{S_{\dSpace-1}}{\kappa_0^{\dSpace}} \int_{\kappa_0 L}^{+\infty}\rho^{\dSpace-1}\exp(-\rho)\, d\rho \\
&= \frac{S_{\dSpace-1}(\dSpace-1)!}{\kappa_0^{\dSpace}}\exp(-\kappa_0 L ) e_{\dSpace-1}(\kappa_0 L)
\,.
\end{aligned}
\]
Since this last expression goes to $0$ as $L$ goes to $+\infty$, the conclusion follows.
\end{proof}
Let us introduce the function $\tilde{\eta}_0:\rr\to\rr$ defined as
\[
\tilde{\eta}_0(\rho) = 
\left\{
\begin{aligned}
&\frac{\desc(m)^2}{2}- \frac{\desc(m)^2}{2L}\rho\quad\text{if}\quad \rho\le 0\,,\\ 
&\frac{\desc(m)^2}{2}\quad\text{if}\quad \rho\ge 0
\,,
\end{aligned}
\right.
\]
see \cref{fig:graph_hull_no_invasion}.
\begin{figure}[!htbp]
\centering
\includegraphics[width=\textwidth]{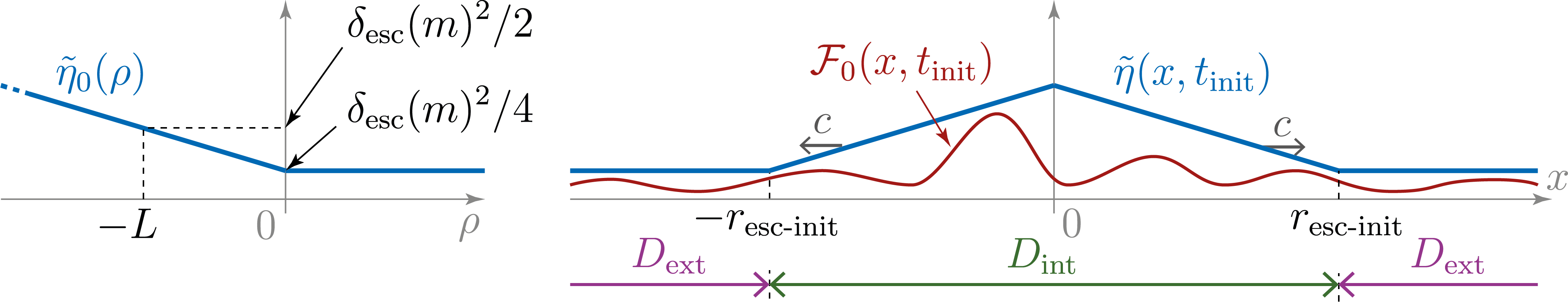}
\caption{Graphs of the functions $\rho\mapsto\tilde{\eta}_0(\rho)$ (left) and $x\mapsto\tilde{\eta}(x,\tInit)$ (right), illustration of \cref{lem:choice_of_tIinit_and_rescInit} (right) and of the sets $\Dint$ and $\Dext$ at $t$ equal to $\tInit$ (right).}
\label{fig:graph_hull_no_invasion}
\end{figure}
Let $c$ denote a (fixed, arbitrarily small) positive quantity, let $\tInit$ and $\rescInit$ denote two positive quantities to be chosen below, and let us introduce the function
\[
\tilde{\eta}:\rr^{\dSpace}\times[\tInit,+\infty)\to\rr
\,,\quad
(x,t)\mapsto \tilde{\eta}_0\Bigl(\abs{x}-\rescInit-c\bigl(t-\tInit\bigr)\Bigr)
\,,
\]
see \cref{fig:graph_hull_no_invasion}.
\begin{lemma}[choice of the quantities $\tInit$ and $\rescInit$]
\label{lem:choice_of_tIinit_and_rescInit}
If the quantity $\tInit$ is large enough positive and if the quantity $\rescInit$ is also large enough positive (depending on the choice of $\tInit$), the following inequalities hold: for every $x$ in $\rr^{\dSpace}$,
\begin{equation}
\label{tInit_and_rescInit_large_enough_so_that_fffZero_below_tildeta_at_tInit}
\fff_0(x,\tInit)\le \tilde{\eta}(x,\tInit) 
\,, 
\end{equation}
and, for every $x$ in $\rr^{\dSpace}$ and $t$ in $[\tInit,+\infty)$,
\begin{equation}
\label{time_derivative_of_fffZero_small_enough}
\partial_t\fff_0(x,t)\le c\frac{\desc(m)^2}{4L}
\,.
\end{equation}
\end{lemma}
\begin{proof}
The fact that inequality \cref{time_derivative_of_fffZero_small_enough} holds, if $\tInit$ is large enough positive, follows from \cref{cor:time_derivative_firewall_goes_uniformly_to_zero}, and the fact that inequality \cref{tInit_and_rescInit_large_enough_so_that_fffZero_below_tildeta_at_tInit} holds, if $\tInit$ is large enough positive and $\rescInit$ is large enough positive once $\tInit$ has been chosen, follows from \cref{lem:derivatives_go_to_zero_beyond_invasion_speed} and from the bounds \cref{bound_u_ut_ck} on the solution. 
\end{proof}
From now on, let us assume that the quantities $\tInit$ and $\rescInit$ are chosen so that the conclusions of \cref{lem:choice_of_tIinit_and_rescInit} hold. 
\begin{lemma}[$\fff_0(x,t)$ remains below $\tilde{\eta}(x,t)$ forever from time $\tInit$ on]
\label{lem:fffZero_remains_below_tildeEta_for_ever}
For every time $t$ greater than or equal to $\tInit$, 
\[
\fff_0(x,s) \le \tilde{\eta}(x,t) 
\quad\text{for all $x$ in $\rr^{\dSpace}$.}
\]
\end{lemma}
\begin{proof}
\renewcommand{\qedsymbol}{}
Let us introduce the function
\[
\Delta:\rr^{\dSpace}\times[\tInit,+\infty)\to\rr\,,\quad
(x,t)\mapsto \fff_0(x,t) - \tilde{\eta}(x,t)
\,.
\]
Proving \cref{lem:fffZero_remains_below_tildeEta_for_ever} amounts to prove that this function is nonpositive on $\rr^{\dSpace}\times[0,+\infty)$. 
Let us introduce the domains
\[
\begin{aligned}
\Dint = \left\{(x,t)\in\rr^{\dSpace}\times[\tInit,+\infty): \abs{x}< \rescInit + c\bigl(t-\tInit\bigr)\right\}\,,\\
\text{and}\quad
\partial\Dint = \left\{(x,t)\in\rr^{\dSpace}\times[\tInit,+\infty): \abs{x}= \rescInit + c\bigl(t-\tInit\bigr)\right\}\,,\\
\text{and}\quad
\Dext(t) = \left\{(x,t)\in\rr^{\dSpace}\times[\tInit,+\infty): \abs{x}> \rescInit + c\bigl(t-\tInit\bigr)\right\}
\,.
\end{aligned}
\]
The following observations can be made concerning the function $\Delta$:
\begin{itemize}
\item according to inequality \cref{tInit_and_rescInit_large_enough_so_that_fffZero_below_tildeta_at_tInit}, it is nonpositive on $\rr^{\dSpace}\times\{\tInit\}$;
\item it is continuous on $\rr^{\dSpace}\times[0,+\infty)$;
\item its partial derivative $\partial_t\Delta$ is defined on $\Dint\sqcup\Dext$;
\item for every $(x,t)$ in $\Dint$, since, according to the definition of $\tilde{\eta}$,
\[
\partial_t \tilde{\eta}(x,t) = c\frac{\desc(m)^2}{4L}
\,,
\]
it follows from inequality \cref{time_derivative_of_fffZero_small_enough} that $\partial_t\Delta(x,t)$ is nonpositive. 
\item for every $(x,t)$ in $\Dext$, according to inequality \cref{dt_fire} of \cref{lem:approx_decrease_fire}, 
\begin{equation}
\label{upper_bound_partial_t_Delta_with_fffZero}
\partial_t\Delta(x,t) \le - \nuFzero\fff_0(x,t) + \KFzero \int_{\SigmaEsc(t)} T_x\psi_0(y) \, dy
\,.
\end{equation}
\end{itemize}
Let us proceed by contradiction and assume that the set
\[
\bigl\{t\in[\tInit,+\infty): \text{ there exists $x$ in $\rr^{\dSpace}$ such that $\Delta(x,t)>0$} \bigr\}
\]
is nonempty. Let us denote by $\tExit$ the infimum of this set, which is a nonnegative quantity, and let us introduce the quantity
\begin{equation}
\label{def_and_condition_tau}
\tau = \frac{L}{c}
\,,\quad\text{so that}\quad
\tau c\frac{\desc(m)^2}{2L} \le \frac{\desc(m)^2}{2}
\,.
\end{equation}
The following lemma conflicts the definition of $\tExit$ (and thus completes the proof).
\end{proof}
\begin{lemma}[$\Delta(\cdot,\cdot)$ cannot reach any positive value on $\rr\times{[\tExit,\tExit+\tau]}$]
\label{lem:Delta_remains_nonpositive_between_tExit_and_tExit_plus_tau}
For every $x$ in $\rr^{\dSpace}$, 
\begin{equation}
\label{Delta_nonpositive_at_time_tExit}
\Delta(x,\tExit)\le 0
\,,
\end{equation}
and for every $(x,t)$ in $\rr^{\dSpace}\times[\tExit,\tExit+\tau]$, 
\begin{equation}
\label{implication_partial_t_non_positive}
\Bigl(\Delta(x,t)\ge -\frac{\desc(m)^2}{8} 
\quad\text{and}\quad (x,t)\not\in\partial\Dint\Bigr)
\implies
\partial_t \Delta (x,t)\le 0
\,.
\end{equation}
\end{lemma}
\begin{proof}
Since the function $\Delta(\cdot,\cdot)$ is continuous, inequality \cref{Delta_nonpositive_at_time_tExit} must hold (or else the would be a contradiction with the definition of $\tExit$). Besides, since $\partial_t\Delta(x,t)$ is nonpositive on $\Dint$, implication \cref{implication_partial_t_non_positive} readily holds for $(x,t)$ in $\Dint$. It remains to prove that this implication also holds when $(x,t)$ is in $\Dext$. 

Observe that, according to inequality in \cref{def_and_condition_tau}, to inequality \cref{time_derivative_of_fffZero_small_enough}, to inequality \cref{escape_Escape} of \cref{lem:escape_Escape} and to the definition of $\tilde{\eta}$, for every $(x,t)$ in $\Dext$ with $t$ in $[\tExit,\tExit+\tau]$,
\[
\SigmaEsc(t)\cap B(x,L) = \emptyset
\,,
\]
so that, according to inequality \cref{upper_bound_partial_t_Delta_with_fffZero}, 
\[
\partial_t\Delta(x,t) \le - \nuFzero\fff_0(x,t) + \KFzero\int_{\rr^{\dSpace}\setminus B(x,L)} T_x\psi_0(y)\, dy 
\,,
\]
and thus, according to inequality \cref{control_over_pollution} of \cref{lem:control_over_pollution},
\[
\partial_t\Delta(x,t) \le - \nuFzero\left(\fff_0(x,t) - \frac{\desc(m)^2}{8}\right) = - \nuFzero\left(\Delta(x,t) + \frac{\desc(m)^2}{8}\right)
\,,
\]
so that implication \cref{implication_partial_t_non_positive} holds when $(x,t)$ is in $\Dext$. \Cref{lem:Delta_remains_nonpositive_between_tExit_and_tExit_plus_tau} is proved. 
\end{proof}
\begin{proof}[End of the proof of \cref{lem:fffZero_remains_below_tildeEta_for_ever}]
It follows from \cref{lem:Delta_remains_nonpositive_between_tExit_and_tExit_plus_tau} and from the continuity of $\Delta(\cdot,\cdot)$ that $\Delta(x,t)$ remains nonpositive on $\rr^{\dSpace}\times[\tExit,\tExit+\tau]$, a contradiction with the definition of $\tExit$. \Cref{lem:fffZero_remains_below_tildeEta_for_ever} is proved. 
\end{proof}
It follows from \cref{lem:fffZero_remains_below_tildeEta_for_ever}, from the definition of $\tilde{\eta}(\cdot,\cdot)$, and from inequality \cref{escape_Escape} that, for every time $t$ greater than or equal to $\tInit$, 
\[
\SigmaEsc(t)\subset \Dint(t)
\,.
\]
Then it follows from \cref{lem:exponential_decrease_firewall} and from the bounds \cref{bound_u_ut_ck} on the solution that the invasion speed $\cInv[u]$ cannot be larger than $c$. Since the positive quantity $c$ was chosen arbitrarily small, the invasion speed $\cInv[u]$ must be equal to $0$, a contradiction (indeed $\cInv[u]$ was assumed to be positive, see the beginning of \cref{subsec:invasion_implies_infinitely_neg_asympt_energy}). This shows that inequality \cref{assume_by_contradiction_eeeAsypmt_is_finite} cannot occur, in other words the asymptotic energy $\eeeAsympt[u]$ must be equal to $-\infty$. Conclusion \cref{item:thm_main_invasion} of \cref{thm:main} is proved. 
\subsubsection*{Acknowledgements} 
I am indebted to Thierry Gallay and Romain Joly for their help and interest through numerous fruitful discussions. 
\emergencystretch=1em
\printbibliography 
\bigskip
\mySignature
\end{document}